\documentclass[11pt]{amsart}

\usepackage{subfiles}
\usepackage{geometry}
\usepackage{fullpage}
\usepackage{amstext}
\usepackage{amsthm}
\usepackage{amssymb}
\usepackage[colorlinks=true, linkcolor=blue, citecolor=green, urlcolor=cyan]{hyperref}

\usepackage{tikz}
\usetikzlibrary{mindmap}
\usepackage{subfigure, bbm}
\usepackage{ulem}

\usepackage[ruled]{algorithm2e}

\newtheorem{theorem}{Theorem}[section]
\newtheorem*{theorem*}{Theorem}

\newtheorem*{theoremA*}{Theorem~A}
\newtheorem*{theoremB*}{Theorem~B}
\newtheorem*{theoremC*}{Theorem~C}

\newtheorem*{problem*}{Problem}
\newtheorem*{corollary*}{Corollary}

\newtheorem{lemma}[theorem]{Lemma}
\newtheorem*{lemma*}{Lemma}
\newtheorem{corollary}[theorem]{Corollary}
\newtheorem{definition}[theorem]{Definition}
\newtheorem*{definition*}{Definition}
\newtheorem{proposition}[theorem]{Proposition}
\newtheorem*{proposition*}{Proposition}

\theoremstyle{definition}
\newtheorem{remark}[theorem]{Remark}
\newtheorem*{remark*}{Remark}

\theoremstyle{plain}

\newcommand{\N}{\mathbb{N}}

\newcommand{\R}{\mathbb{R}}
\newcommand{\Event}{\mathcal{E}}
\def\Prob{{\mathbb P}}

\usepackage{xcolor}
\definecolor{b}{HTML}{4472c4}
\definecolor{o}{HTML}{ED7D31}
\definecolor{g}{HTML}{70ad47}
\definecolor{t}{RGB}{40,154,150}

\def\Id{{\rm Id}}

\def\R{{\mathbb R}}
\def\N{{\mathbb N}}

\def\Prob{{\mathbb P}}
\def\Exp{{\mathbb E}}
\def\Event{{\mathcal E}}

\def\Proj{{\rm Proj}}

\def\Net{{\mathcal N}}

\def\conv{{\rm conv}}

\newcommand{\tref}[1]{\text{\tiny{\ref{#1}}}}
\newcommand{\step}[1]{\smallskip\noindent{\it #1.}}

\title{Cotype of random polytopes}
\author{Han Huang\address{Department of Mathematics, University of Missouri, Columbia}\email{hhuang@missouri.edu} 
and Konstantin Tikhomirov\address{Department of Mathematical Sciences, Carnegie Mellon University}\email{ktikhomi@andrew.cmu.edu}}

\date{\today}

\begin{document}

\maketitle

\begin{abstract}
For $N\geq n$,
let $P_{N,n}$ be a random polytope in $\R^n$
with vertices $\pm X_i$, $1\leq i\leq N$,
where $X_1,\dots,X_N$ are i.i.d standard Gaussian vectors in
$\R^n$.
Random polytopes $P_{N,n}$, as well as their duals,
are classical objects of interest in high-dimensional
convex geometry and local Banach space theory.
In this paper, we provide a 
{\it dimension-independent} bound
on the cotype of the corresponding 
normed space $(\R^n,\|\cdot\|_{P_{N,n}})$,
generated by $P_{N,n}$.
Let $K'\geq K>1$, and assume that
$K'\geq \frac{N}{n}\geq K$.
We show that with probability $1-o(1)$,
for any $k\geq 1$, and any collection 
$y_1,\dots,y_k$ of vectors in $\R^n$,
$$
\Exp_\sigma\,\Big\|\sum_{i=1}^k \sigma_i y_i\Big\|_{P_{N,n}}^q
\geq \frac{1}{C_q^q}\sum_{i=1}^k \big\|y_i\big\|_{P_{N,n}}^q,
$$
where $\sigma=(\sigma_1,\dots,\sigma_k)$
is a vector of random signs,
and where
$q\in [2,\infty)$ and $C_q\in[1,\infty)$
may {\it only} depend on $K,K'$.
We discuss the result in context of infinite-dimensional
Banach spaces.
\end{abstract}


\section{Introduction}

\subsection{Main results}
Let $N\geq n$, and let $X_1,\dots,X_N$
be i.i.d random vectors in $\R^n$.
Define a random symmetric polytope 
\begin{equation}\label{def:polydef}
P_{N,n}:=\conv\big\{\pm X_i,\;1\leq i\leq N\big\}.    
\end{equation}
Polytopes of type \eqref{def:polydef}
for various models of randomness of $X_i$'s
have been actively studied in literature,
and play a fundamental role in
high-dimensional convex geometry
and geometric probability.
We refer, in particular, to books and surveys
\cite{WW93,Barany2004,Schneider2008,Reitzner2010}, \cite[Section~8.2]{SW2008},
as well as research papers
\cite{RS1963,E1965,B1989,Muller89,AS92,BV94,K1994,HMR04,Retzner2005-,Retzner2005,LPRT2005,Vu2005,Vu2006,Bor2009,AGP2015,GKZ20,KZ19,KZ2020,Ka2021} and references therein
for a comprehensive background.
Random polytopes naturally arise in the problem of
approximation of convex bodies, average-case and
smoothed analysis of linear programming,
and serve as (near-)extremizers
within asymptotic theory of normed spaces.

Here, we focus on polytopes
\eqref{def:polydef} generated by i.i.d standard Gaussian
random vectors $X_1,\dots,X_N$.
For that model, or in a closely related setting
of asymmetric Gaussian polytopes $\conv\{X_1,\dots,X_N\}$,
the global geometry and associated parameters
are relatively well studied; in particular, that includes
the expected volume,
absorption probabilities, and the number of
facets \cite{AS92,BV94,KZ19,KZ2020}.
On the other hand, geometry of sections of such polytopes remains
much less explored.
In this work, we make a step
towards characterization of 
the local geometry of $P_{N,n}$ via the notions of
Banach--Mazur distance and cotype.


\smallskip

Recall that the {\it Minkowski functional}
of a centrally-symmetric convex body $\mathcal C$ in $\R^n$
is defined as
$$
\|x\|_{\mathcal C}:=\inf\limits\{\lambda> 0:\,\lambda^{-1} x\in \mathcal C\}.
$$
Since for $N\geq n$, the polytope $P_{N,n}$
has a non-empty interior almost surely, and
with probability one the random functional $\|\cdot\|_{P_{N,n}}$
is a valid norm in $\R^n$.
We recall that, given a normed space ${\bf X}$
(either finite or infinite-dimensional), and
a number $q\in[2,\infty)$,
${\bf X}$
is said to have {\it cotype $q$ with a constant $C_q<\infty$} \cite{MP76}
if 
\begin{equation}\label{eq:cotype}
\Exp_\sigma\,\Big\|\sum_{i=1}^k \sigma_i y_i\Big\|_{{\bf X}}^q
\geq \frac{1}{C_q^q}\sum_{i=1}^k \big\|y_i\big\|_{{\bf X}}^q,
\quad
\mbox{for every $k\geq 1$ and every
$y_1,\dots,y_k$ in ${\bf X}$,}
\end{equation}
where the expectation is taken with respect to the
uniform random sign vector $\sigma=(\sigma_1,\dots,\sigma_k)$.
An infinite-dimensional Banach space ${\bf X}$
has {\it infinite cotype} if \eqref{eq:cotype}
fails for every choice of finite $q$ and $C_q$.
The main result of the paper is the following theorem,
which provides dimension-independent bounds on the cotype
of the random normed spaces generated by $P_{N,n}$:
\begin{theoremA*}[Cotype of Gaussian polytopes]
For every $K'\geq K>1$ there are numbers $q\in[2,\infty)$ and
$C_A,C_q\in[1,\infty)$ depending only on $K,K'$ with the following property.
Let $K'\geq \frac{N}{n}\geq K$, and let $X_1,X_2,\dots,X_N$
be i.i.d standard Gaussian random vectors in $\R^n$.
Let polytope $P_{N,n}$ be defined according to \eqref{def:polydef}.
Then with probability at least $1-\frac{C_A}{n}$,
the normed space $(\R^n,\|\cdot\|_{P_{N,n}})$
has cotype $q$ with constant at most $C_q$.
\end{theoremA*}

\begin{remark*}
It is a standard fact that any finite-dimensional normed space
has a finite cotype with constants depending only on the dimension.
We emphasize that in the above theorem, $q$ and $C_q$ {\it do not}
depend on $n$.    
\end{remark*}

\begin{remark*}
The above theorem would follow
immediately if typical realizations of $(\R^n,\|\cdot\|_{P_{N,n}})$
could be uniformly isomorphically embedded into some Banach space known to have
a finite cotype.
Even though the space $(\R^n,\|\cdot\|_{P_{N,n}})$
can be trivially represented as a rank $n$ {\it projection} of $\ell_1^N$,
no relevant result regarding {\it sections} (subspaces) of a
finite cotype Banach space appears to be known.
\end{remark*}

Given any integer $k\geq 1$,
and any pair of $k$--dimensional normed spaces $E,F$,
the {\it Banach--Mazur} distance between $E$ and $F$ 
is defined as
$$
d_{BM}(E,F)=\inf\limits_T\,\|T\|_{E\to F}\cdot\|T^{-1}\|_{F\to E},
$$
where the infimum is taken over all invertible 
linear operators $T:E\to F$, and $\|\cdot\|_{E\to F}$
and $\|\cdot\|_{F\to E}$ denote the respective induced operator
norms. 
For two normed spaces $E$ and ${\bf X}$ and an injective map $f:E\to{\bf X}$,
we will refer to the Banach--Mazur distance between $f(E)$ (with the norm induced by ${\bf X}$)
and $E$ as the {\it distortion} of the embedding.
A seminal result of Maurey and Pisier \cite{MP76}
gives a characterization of infinite cotype in terms of embedding distortion:
a Banach space ${\bf X}$ has infinite cotype
if and only if spaces
$\{\ell_\infty^m\}_{m=1}^\infty$
can be embedded into ${\bf X}$ with a uniformly bounded distortion.
In our context, the result of Maurey--Pisier, allows to derive Theorem~A from the following
statement which is the second main result of the paper:
\begin{theoremB*}[Sections of Gaussian polytopes]
Let $K'\geq K>1$, and assume that $K'\geq \frac{N}{n}\geq K$.
Let $X_1,X_2,\dots,X_N$
be i.i.d standard Gaussian random vectors in $\R^n$,
and let $P_{N,n}$ be defined by \eqref{def:polydef}.
Then with probability at least $1-\frac{C_B}{n}$,
for all $1\leq k\leq n$
and all $k$--dimensional subspaces $E$ of
$(\R^n,\|\cdot\|_{P_{N,n}})$,
$$
d_{BM}(E,\ell_\infty^k)\geq c\,k^\alpha.
$$
Here, $C_B,c>0$ may only depend on $K',K$,
and $\alpha>0$ is a universal constant.
\end{theoremB*}



\subsection{Banach spaces with maximal local inhomogeneity}
A standard fact within the Banach space theory is {\it universality}
of $\ell_\infty$, or, more generally, of any space
which contains $\{\ell_\infty^m\}_{m=1}^\infty$
uniformly.
Specifically, any Banach space ${\bf X}$ which contains 
$\{\ell_\infty^m\}_{m=1}^\infty$
uniformly, admits almost isometric embeddings of arbitrary 
finite-dimensional normed spaces, as well as finite metric spaces.
The aforementioned result of Maurey and Pisier \cite{MP76}
provides a probabilistic characterization of universal
Banach spaces via the notion of cotype: ${\bf X}$
is universal in the above sense if and only if ${\bf X}$ has infinite cotype.
In view of the above, Banach spaces with infinite cotype
are natural extremizers of functionals depending on
local (finite-dimensional) structure of a space.
Here, we consider a functional
measuring local {\it inhomogeneity}
of Banach spaces.

\smallskip


\smallskip

Given a normed space ${\bf X}$ (finite or
infinite-dimensional) and an integer $1\leq k\leq \dim{\bf X}$,
define
$$
D_{\bf X}(k):=\sup\big\{
d_{BM}(E,F):\;\dim E=\dim F=k;\;E,F\subset {\bf X}
\big\},
$$
i.e $D_{\bf X}(k)$ is the largest Banach--Mazur
distance between any two $k$--dimensional subspaces of ${\bf X}$.
The quantity $D_{\bf X}(k)$ can be viewed as
a measure of inhomogeneity of ${\bf X}$
in regard to structure of its subspaces,
and have been considered in literature under varying names.

\smallskip

Note that, by a classical result of Jordan and von Neuman \cite{JvN35},
$D_{\bf X}(2)=1$ if and only if ${\bf X}$ is a Hilbert (Euclidean) space.
Further, the Dvoretzky--Milman theorem \cite{Dv,Mil}
implies that for every separable infinite-dimensional
Banach space ${\bf X}$, $\sup\nolimits_k D_{\bf X}(k)<\infty$
if and only if ${\bf X}$ is isomorphic to a Hilbert space \cite{Kwapien}.
In the finite-dimensional setting,
the quantity $D_{\bf X}(k)$ was considered, in particular,
in \cite{BourgainD,MTJD} where it was shown that
for every $\alpha\in(0,1)$ and any $n$--dimensional
normed space ${\bf X}$, the condition
$D_{\bf X}(\lfloor \alpha n\rfloor)\leq K$
for some $K<\infty$ implies $d_{BM}({\bf X},\ell_2^n)\leq f(K,\alpha)$
for some function $f$ of $K$ and $\alpha$ only.
The general phenomenon underlying the aforementioned
results, can be summarized as follows: if $D_{\bf X}(k)$
is small then the corresponding Banach space ${\bf X}$
is Hilbert or nearly Hilbert.

Departing from the problem of characterizing
Hilbertian spaces and looking at the opposite extreme,
one can ask what the largest magnitude of $D_{\bf X}(k)$
for ${\bf X}$ from a given family of spaces can be?
Without any restriction on ${\bf X}$,
the question is equivalent to estimating the largest
possible Banach--Mazur distance between
two $k$--dimensional subspaces.
The latter was solved in a seminal work of Gluskin \cite{Gluskin},
which, combined with
a classical theorem of John,
gives order of magnitude $\Theta(k)$
for the diameter of the {\it Banach--Mazur compactum}
of $k$--dimensional spaces.
Combined with the discussion from the beginning of the subsection,
this yields: for every Banach space ${\bf X}$
of infinite cotype, $D_{\bf X}(k)=\Theta(k)$ as $k\to\infty$
(we refer to Preliminaries for asymptotic notation used in this paper).
The question that we address below is {\it whether
spaces of infinite cotype are the only extremizers
of the local inhomogeneity $D_{\bf X}(k)$ in the asymptotic setting $k\to\infty$}?

\smallskip

The proof of Gluskin's theorem in \cite{Gluskin} is based
on a randomized construction
different from the one provided by \eqref{def:polydef},
yet the same result can be obtained in the setting of random Gaussian
polytopes \cite{LGl}. Namely,
it can be shown that, for $N=\lfloor Cn\rfloor$
for an appropriate universal constant $C>1$,
two independent copies of the normed space
$(\R^n,\|\cdot\|_{P_{N,n}})$ are at the Banach--Mazur
distance of order $\Theta(n)$ with high probability.
The cotype bound which we obtain in our work,
together with a standard Banach space
construction, yield:
\begin{theoremC*}[A Banach space 
with maximal local inhomogeneity]
There exists a separable Banach space ${\bf X}$
of finite cotype with
$
D_{\bf X}(k)=\Theta(k)
$ as $k\to\infty$.
\end{theoremC*}

\subsection{Outline of the argument}

Proof of the main result of the paper --- Theorem~B (see Theorem~\ref{th: Gluskinlinfty})
--- is rather technical, and comprises multiple reduction steps
and treatment of various subcases dealing with different types of
potential low-distortion embeddings of $\ell_\infty$.
In order to organize our discussion of the proof ideas,
we will start with a few key definitions and several ``global'' objects.

We denote by $X_1,\dots,X_N$ i.i.d standard Gaussian vectors in $\R^n$,
and let $A$ be the $n\times N$ random matrix with columns $X_1,\dots,X_N$.
We will often work with a typical realization of $A$, for which standard
bounds on the singular values of $A$ itself as well as
of its submatrices are readily available (see Lemma~\ref{lem:sparse_singular_values}).
Furthermore, those bounds, combined with the Blaschke--Santal\'o inequality
and a volumetric argument, imply estimates on the in-radius
of $P_{N,n}$, as well as (relative) in-radius of convex hulls of subsets
of vectors $X_1,\dots,X_N$ (Lemma~\ref{PN:inradius}).

Given $k\leq n$, our goal
is to show that for a typical realization
of the random polytope $P_{N,n}$, no $k$--tuple of vectors $y_1,\dots,y_k$
in $(\R^n,\|\cdot\|_{P_{N,n}})$ spans a subspace close to $\ell_\infty^k$.
We will test the embedding distortion under the mapping $e_i\to y_i$,
with $e_i$ denoting the $i$-th standard basis vector in $\ell_\infty^k$.
In the course of the argument, we will apply transformations to the original
$k$--tuple, producing other vectors which we will denote by $y_i', \tilde y_i$ etc.,
as well as passing to subsets of vectors, with new index sets denoted by $L,U$, etc.

In testing the distortion under the mapping $e_i\to y_i$, $i\leq k$,
a key role is played by the vertices of the associated parallelotope, given by
$\sum_{i=1}^k \sigma_i y_i$, where $\sigma$ denotes a fixed or random vector of signs.
A considerable portion of the Preliminaries section is devoted to studying
basic properties of such vectors, in particular, connecting tail estimates for
the distribution of $\langle Z,\sum_{i=1}^k \sigma_i y_i\rangle$ (for a given non-random vector $Z$)
with the structure of the covariance matrix of the random vector $\sum_{i=1}^k \sigma_i y_i$
under the uniform distribution of signs (see Lemma~\ref{lem:dyaddecomp}),
as well as statistical properties of the inner products $\langle X_j,\sum_{i=1}^k \sigma_i y_i\rangle$ (Lemma~\ref{lem:dotprodstat}).

To facilitate handling $\|\cdot\|_{P_{N,n}}$--norms of vectors, we introduce the notion
of a {\it coefficient vector}. Given any $y\in\R^n$, let $\beta(y)$ be a vector in $\R^N$
such that $y=\sum_{j=1}^N \beta_j(y) X_j$, and $\|\beta(y)\|_1=\|y\|_{P_{N,n}}$.
A vector $\beta(y)$ with the above properties exists for every $y$
simply by the definition of a convex hull.
When working with the sign combinations $\sum_{i=1}^k \sigma_i y_i$
and when $y_1,\dots,y_k$ are clear from context, we use a shorter notation $\beta^\sigma$
in place of $\beta\big(\sum_{i=1}^k \sigma_i y_i\big)$
(see Definitions~\ref{def: beta 09812720894} and~\ref{def: betasigma 15262}).
Note that $\sum_{i\leq k} \sigma_i y_i = A\beta^\sigma$.
A simple observation, which highlights challenges
in proving the main result, is that the mapping $y\to\beta(y)$
is in general non-linear (for $N>n$), and, in particular,
$\beta^\sigma$ and $\sum_{i=1}^k \sigma_i \beta(y_i)$ are distinct vectors.
A critical property of the differences $\beta^\sigma-\sum_{i=1}^k \sigma_i \beta(y_i)$,
and, more generally, of any $N$--tuple $\beta$ such that
$\sum_{j\leq N}\beta_j X_j=0$, is that
for a typical realization of $X_j$'s,
any such vector is {\it incompressible},
i.e, after normalization, is at constant
Euclidean distance from the set of sparse vectors in $\R^N$
(see Definition~\ref{def: compincomp} and Lemma~\ref{PN:incompcomb}).

\subsubsection{Linear span of an orthonormal $k$--tuple}
To illustrate the core idea of the proof, 
let us suppose that we have managed to find
a $k$--tuple of {\it orthonormal} vectors $y_1,\dots,y_k$ in $\R^n$ such that
$$\|y_i\|_{P_{N,n}}=  \Theta(1), \quad i\leq k; \qquad
\Big\|\sum_{i\leq k} \sigma_i y_i\Big\|_{P_{N,n}}=\Theta(1), \quad \sigma\in\{-1,1\}^k,$$
so that the mapping $e_i\longrightarrow y_i$, $i\leq k$,
induces a Lipschitz embedding of $\ell_\infty^k$ into $(\R^n,\|\cdot\|_{P_{N,n}})$.
We want to show that this assumption ultimately comes in contradiction with
the typical structure of $P_{N,n}$.

A standard volumetric estimate for $P_{N,n}$ implies that if we pick a uniform
random direction $\theta \in S^{n-1}$ then with high probability
$\|\theta\|_{P_{N,n}} =\Theta(1)$. If the direction of the vector
$\sum_{i\leq k} \sigma_i y_i$ were ``random-like'' then $\|\sum_{i\leq k} \sigma_i y_i\|_{P_{N,n}}$ would have been
of order $\Theta(\sqrt{k})$, since $\sum_{i\leq k} \sigma_i y_i$ has Euclidean length $\sqrt{k}$. Thus, assuming that $\|\sum_{i\leq k} \sigma_i y_i\|_{P_{N,n}}=\Theta(1)$ for every $\sigma$, we
get that all vectors $\sum_{i\leq k} \sigma_i y_i$ are atypical representatives of $\sqrt{k}\,S^{n-1}$.
In fact, we can show that in this setting $\beta^\sigma$ has $\ell_2$--norm of order $\sqrt{k/N}$
whereas $\|\beta^\sigma\|_1 = \|\sum_{i\leq k} \sigma_i y_i\|_{P_{N,n}} = \Theta(1)$,
which implies that $\beta^\sigma/\|\beta^\sigma\|_2$ is {\it compressible} i.e
is close to being sparse (see Definition~\ref{def: compincomp}).

The above structural properties of $\beta^\sigma$
lead to the following approximate formula for all $j\leq N$ such that 
$\beta^\sigma_j$ is ``sufficiently'' large:
\begin{equation*}    
\begin{split}
\langle X_j, \sum_{i\leq k} \sigma_i y_i\rangle &= \langle X_j, \sum_{\ell\leq N}\beta^\sigma_\ell X_\ell\rangle 
= \langle X_j, A\beta^\sigma\rangle \\
&=
\big(A^\top A \beta^\sigma\big)_j
= (1+o(1)) \beta_j^\sigma \|X_j\|_2^2
= (1+o(1)) \beta_j^\sigma\,n
\end{split}
\end{equation*}
(see Lemma~\ref{lem:approxlemma}).
Assume for a moment that we were able to make use of the last formula
by finding a subset $J\subset[N]$
of size $n/k^{\Omega(1)}$ having the property that for every $\sigma\in\{-1,1\}^k$,
$$A_{J}^\top A_{J}\beta^\sigma_{J} =(1+o(1)) (\beta_j^\sigma n)_{j\in J},$$
and $\|A_{J^c}\beta^\sigma_{J^c}\|_2=o(\sqrt{k})$,
where we use notation $\cdot_J$ and $\cdot_{J^c}$
to denote restriction of a vector or of matrix columns to the index set $J$ or $J^c$
(see Definition~\ref{def:restriction}).
Denote $H={\rm span}\{X_j : j \in J\}$.
The assumptions on $J$ imply that the orthogonal
projection of $\sum_{i\leq k} \sigma_i y_i$ onto $H^\perp$
(which can equivalently be represented as the projection 
of vector $A_{J^c}\beta^\sigma_{J^c}$ onto $H^\perp$)
has a small Euclidean norm for all $\sigma$, which in turn implies that, for some
indices $i$, the vector $P_{H^\perp}y_i$ has a small Euclidean norm.
Since $P_{N,n}$ has in-radius of order $1$ (Lemma \ref{PN:inradius}), this yields
$$
  \|P_{H^\perp}y_i\|_{P_{N,n}} \lesssim \|P_{H^\perp}y_i\|_2=o(1).
$$
On the other hand,
the section $P_{N,n} \cap H$ contains the convex hull of $\pm X_j$, $j\in J$, which is essentially a $|J|$-dimensional $\ell_1$--ball rescaled by $\sqrt{n}$. Hence its
relative in-radius is at least of order $\frac{\sqrt{n}}{\sqrt{|J|}} = k^{\Omega(1)}$ (Lemma \ref{PN:inradius}), and therefore
$$
  \|P_H y_i\|_{P_{N,n}} \lesssim k^{-\Omega(1)} \|P_H y_i\|_2=o(1).
$$
The bounds on $\|P_{H^\perp}y_i\|_{P_{N,n}}$
and $\|P_{H}y_i\|_{P_{N,n}}$ contradict the assumption $\|y_i\|_{P_{N,n}}=\Theta(1)$,
completing the argument.

The actual set $J$ which we construct in the proof,
has weaker properties compared to the simplified discussion above:
rather than requiring that
$A_{J}^\top A_{J}\beta^\sigma_{J} =(1+o(1)) (\beta_j^\sigma n)_{j\in J}$
and $\|A_{J^c}\beta^\sigma_{J^c}\|_2=o(\sqrt{k})$ for all $\sigma$,
we ask for the properties to be satisfied for an appropriately large
selection of sign vectors.
The set $J$ is taken to be the union of all indices $j$ 
for which either a fraction of sign vectors $\sigma$ satisfies
$$
  |\langle X_j, \sum_{i\leq k} \sigma_i y_i\rangle| \ge k^{1/2+\Omega(1)}\,,
$$
or $|\langle X_j, \sum_{i\leq k} \sigma_i y_i\rangle| \ge k^C$ for a some choice of $\sigma$ and a large constant $C>0$ (see Definition~\ref{def: Jt}). Either condition means that the projections of $X_j$ onto the linear span
of $y_i, i\leq k$, carry an atypically large weight. By statistical properties of 
orthogonal projections of $X_j$'s 
(Lemma \ref{lem:dotprodstat} and Remark~\ref{rem: Jtsize}), this 
implies $|J|=n/k^{\Omega(1)}$.
The technical aspect of this part of the proof is to justify 
that a set $J$ defined as above actually satisfies the required conditions
on $A_{J}^\top A_{J}\beta^\sigma_{J}$
and $A_{J^c}\beta^\sigma_{J^c}$
(see Lemma~\ref{lem: Abetasigma}).
To get the required estimates,
we relate the set $J$
to its ``$\sigma$--specific'' counterparts,
namely, to sets $T_\sigma$ defined
as collections of indices $j\leq N$ for which
$\beta^\sigma_j$ exceeds a specially chosen threshold
(see Definition~\ref{def: Tsigmatau} and Corollary~\ref{cor:approxcor}).

\subsubsection{Linear span of a $k$--tuple without orthonormality}
In the actual proof, we do not use the assumption
that vectors $y_i$'s form an orthonormal sequence,
which we employed above to illustrate a few key ideas.
Hence, in the actual setting, we can no longer rely on sharp estimates for 
$\|\sum_{i\leq k}\sigma_i y_i\|_2$, except
for the trivial observation that the second moment (w.r.t
randomness of $\sigma$) of the norm is of order $k$.
To deal with arbitrary 
configurations of vectors $y_1,\dots,y_k$ of unit Euclidean length
with comparable $\|\cdot\|_2$ and $\|\cdot\|_{P_{N,n}}$--norms,
we rely on certain decomposition of the eigenspace
of the covariance matrix of $\sum_{i\leq k}\sigma_i y_i$ (see Lemma~\ref{lem:dyaddecomp}).
It turns out that the argument presented above can be
adapted to the anisotropic setting, at expense of somewhat
more technical computations.

\subsubsection{Union bound over $k$--tuples}
The argument discussed above allows to rule out
any particular $k$--tuple of unit vectors $y_1,\dots,y_k$ 
with $\|y_i\|_{P_{N,n}}=\Theta(1)$, $i\leq k$,
as a candidate for a low-distortion embedding of $\ell_\infty^k$
via $e_i\to y_i$.
However, to guarantee that no small distortion embeddings exist,
we need to consider all possible $k$--tuples simultaneously.
As a standard technique within high-dimensional probability
and convex geometry, we apply a discretization argument
at this stage. A central piece of the discretization
is Lemma \ref{lem:Grassnets} which deals with $\varepsilon$--nets
on the Grassmannian of $d$--dimensional subspaces.
A classical result of Szarek \cite{Szarek81}
implies that there exists an $\varepsilon$--net $\Net$ on $G_{n,d}$
of size at most
$\left( \frac{C}{\varepsilon} \right)^{d(n - d)}$.
It turns out that, furthermore,
for any $d$-dimensional subspace $F\subset\R^n$, we can express the orthogonal projection onto $F$ as an infinite sum 
\begin{align}
\label{eq:intro_proj_representation}
  \Proj_{F}= \sum_{j=1}^\infty D_j \Proj_{E_{j}}, \quad \text{where } D_j \mbox{ is matrix  with } \|D_j\| \le 2 \varepsilon^{j-1} \mbox{ and } E_j \in {\mathcal N}\,.
\end{align}
Clearly, the above formula is an operator analogue of the standard vector
$\varepsilon$-net expansion on the sphere $S^{n-1}$, which establishes that any $x \in S^{n-1}$ can be expressed as a convergent series $x = \sum_{i \ge 0} \lambda_i x_i$ with $|\lambda_i| \le \varepsilon^i$, and with $x_i$ in an $\varepsilon$--net on $S^{n-1}$ (see, for example, \cite[Exercise 4.34]{Ver18}). We do not know
if formula \eqref{eq:intro_proj_representation}
has appeared in literature earlier. In our opinion, the identity 
is of independent interest.
The $\varepsilon$--argument
provides a uniform control on the
statistics of magnitudes of projections of $X_j$'s onto linear
subspaces of $\R^n$. 
There is, however, a constraint: the admissible value of $\varepsilon$ depends implicitly on the ratio between the $\|\cdot\|_{P_{N,n}}$ and $\|\cdot\|_2$ norms of the vectors $y_i$, and the net-argument
fails whenever the gap between $\|y_i\|_{P_{N,n}}$
and $\|y_i\|_2$ is large. Summarizing the argument
up to this point, we obtain (with explicit powers of $k$):
\begin{proposition*}[informal statement; see Proposition~\ref{prop:spansofcomp}]
Condition on a typical realization of $P_{N,n}$.
Let $1\leq k\leq n/2$, and let $y_1,\dots,y_k$
be vectors in $\R^n$ of unit Euclidean length
such that $\|y_i\|_{P_{N,n}}\geq C\,k^{-1/9}$
for every $i\leq k$. Then,
$$
\Exp_\sigma\,\Big\|
\sum\nolimits_{i=1}^k \sigma_i y_i
\Big\|_{P_{N,n}}\geq k^{1/8},
$$
and, in particular, the embedding of $\ell_\infty^k$
given by $e_i\longrightarrow y_i$, $i\leq k$,
induces a distortion polynomial in $k$.
\end{proposition*}

\subsubsection{Proof of Theorem~B}
The above proposition effectively forbids low-distortion embeddings
of $\ell_\infty^k$ under the extra constraint
that the standard basis of $\ell_\infty^k$ is mapped into
vectors with not-to-different $\|\cdot\|_{P_{N,n}}$ and $\|\cdot\|_2$--norms.
To pass from this statement to Theorem~B in its full generality,
we have to consider the scenario where the unit vectors $y_1,\dots,y_k$
have very small $\|\cdot\|_{P_{N,n}}$--norm.
The core statement which drives this part of the proof
is the aforementioned Lemma~\ref{PN:incompcomb}
which asserts that for a typical realization of $P_{N,n}$,
any unit $N$--tuple of coefficients $\beta$
such that $\sum_{i\leq N}\beta_i X_i=0$,
must be incompressible. As a technical corollary of Lemma~\ref{PN:incompcomb},
we obtain the next result:

\begin{lemma*}[informal statement; see Lemma~\ref{lem:generalPNbasic}]
Condition on a typical realization of $P_{N,n}$.
Assume that $J\subset [N]$
is a non-empty subset of indices with
$|J|\leq c\,n$.
Let $y_i$, $i\in L$, be a finite collection of non-zero vectors
in $\R^n$, let $\sigma$ be a uniform random
vector of signs indexed over $L$, and 
let $\beta^\sigma$ be the coefficient vector of the sum
$\sum_{i\in L}\sigma_i\,y_i$.
Assume that
\begin{equation}\label{eq: over 0498h34iq}
\sum_{h\in J}\sqrt{\sum_{i\in L}\beta_h(y_i)^2}
=\Omega\Big(
\Exp_{\sigma}\,\|\beta^\sigma\|_1\Big).
\end{equation}
Then there are $\Omega(n)$
indices $j\in J^c$ satisfying
$$
\sqrt{\sum_{i\in L}\beta_j(y_i)^2}
=\Omega\Big(
\frac{1}{\sqrt{|J|\,n}}\,
\sum_{h\in J}
\,\sqrt{\sum_{i\in L}\beta_h(y_i)^2}\Big).
$$
\end{lemma*}

Note that the expression $\sqrt{\sum_{i\in L}\beta_h(y_i)^2}$, $h\in[N]$,
can be viewed as a typical order of magnitude of the $h$--th
component of the random linear combination
$\sum_{i\in L}\sigma_i\,\beta(y_i)$.
Thus, the above statement asserts (roughly) that
whenever the typical magnitudes of the components of $\sum_{i\in L}\sigma_i\,\beta(y_i)$
indexed over a small set $J$ are large then there must be 
$\Omega(n)$ indices $j\in J^c$ for which
$\sum_{i\in L}\sigma_i\,\beta_j(y_i)$ are large on average (up to some loss
in the estimate).
Let us emphasize that if it were true that 
(normalization of)
the vector $\big(\sqrt{\sum_{i\in L}\beta_h(y_i)^2}\big)_{h\in[N]}$
was incompressible then the above statement would be trivial
without any lower bound on $\sum_{h\in J}\sqrt{\sum_{i\in L}\beta_h(y_i)^2}$.
Indeed, given any incompressible vector $v\in S^{N-1}$,
and any small subset $J\subset[N]$, for
$\Omega(N)$ components 
of $v$ we have
$$
|v_j|=\Omega\Big(\frac{1}{\sqrt{N}}\Big)
=\Omega\Big(\frac{1}{\sqrt{|J|\,N}}\sum_{h\in J}|v_h|\Big).
$$
The issue is that the vector $\big(\sqrt{\sum_{i\in L}\beta_h(y_i)^2}\big)_{h\in[N]}$ is not necessarily
incompressible (and, in fact, is {\it not} incompressible
as long as $y_1,\dots,y_k$
have a very small $\|\cdot\|_{P_{N,n}}$-norm),
so the above trivial argument is not applicable. However,
in view of Lemma~\ref{PN:incompcomb}, the normalized difference
$\sum_{i\in L}\sigma_i\,\beta(y_i)-\beta^\sigma$
{\it must be} incompressible, allowing for a more technical variant
of the above argument to go through, under the extra assumption \eqref{eq: over 0498h34iq}.

The ``pseudo-incompressibility'' of the vector $\big(\sqrt{\sum_{i\in L}\beta_h(y_i)^2}\big)_{h\in[N]}$
stated above, has profound consequences, since it allows tracing some of 
structural properties of the coefficient vectors $\beta(y_i)$.
Whereas those properties cannot be expressed in simple terms (such as compressible/incompressible),
they nevertheless affect statistics of magnitudes of the components $\beta_j(y_i)$.

Let $y$ be a vector in $\R^n$, and let $\delta\in(0,1]$
be a parameter.
Define
$$
m_\delta(y,r):=\Big|\Big\{
j\leq N:\;\big|\beta_j(y)\big|\in
\frac{\delta}{n}\big(2^{r-1},2^r\big]\Big\}\Big|,
\quad r\geq 1
$$
(see Definition~\ref{def: myr 334q}).

Assume for a moment that 
for every vector $y_i$, $i\in L$, we have $\|y_i\|_{P_{N,n}}\in[\delta,2\delta]$, and
for a small parameter $\alpha>0$,
\begin{equation}\label{eq: alkjnflkqjflkqjwnf}
m_\delta(y_i,r)\leq |L|^{\alpha}\,2^{-2r}\,n,\quad 0\leq r\leq \log_2 \sqrt{|L|},
\end{equation}
and
$$
\Big\|\sum_{i\in L}v_i\,y_i\Big\|_{P_{N,n}}
\leq \delta\,\|v\|_\infty\,|L|^{\alpha},\quad  v=(v_i)_{i\in L}
$$
(note that the last condition is simply a Lipschitz distortion upperbound
for the linear mapping from $\ell_\infty^{|L|}$ to $(\R^n,\|\cdot\|_{P_{N,n}})$
given by $e_i\to y_i$, $i\in L$).
We show that if for a significant proportion of vectors $y_i$, we had
$$
\sum_{j=1}^N |\beta_j(y_i)|\,
{\bf 1}_{\{|\beta_j(y_i)|>\delta\,\sqrt{|L|}/n\}}
> |L|^{-1/16}\,\delta,
$$
that would eventually contradict the pseudo-incompressibility of $\big(\sqrt{\sum_{i\in L}\beta_h(y_i)^2}\big)_{h\in[N]}$ (see the second part of the proof of Lemma~\ref{lem:stairs}),
and, consequently, a large fraction of vectors $\beta(y_i)$
must be not ``too spiky'' in the sense that
$$
\sum_{j=1}^N |\beta_j(y_i)|\,
{\bf 1}_{\{|\beta_j(y_i)|>\delta\,\sqrt{|L|}/n\}}
\leq |L|^{-1/16}\,\delta\quad\mbox{ for many indices $i\in L$}.
$$
The last condition, in turn, allows us to define truncations
$$
y_i':=\sum\limits_{j=1}^N \beta_j(y_i)\,
{\bf 1}_{\{|\beta_j(y_i)|\leq \delta\,\sqrt{|L|}/n\}}\,X_j,
$$
so that
$\|y_i-y_i'\|_{P_{N,n}}\leq |L|^{-1/16}\,\delta$,
and $\|y_i'\|_{P_{N,n}}=\Theta(\delta)$, $\|y_i'\|_2=\tilde O\big(|L|^{\alpha/2}\,\delta\big)$.
It can be shown that those vectors, after renormalization $y_i'\longrightarrow\frac{y_i'}{\|y_i'\|_2}$, reduce
the analysis back to the setting of Proposition~\ref{prop:spansofcomp}
(see Lemma~\ref{lem:stairs}),
where we have already established impossibility of a low-distortion embedding
of $\ell_\infty$.

The preceding discussion implies that for a collection of vectors $y_i$, $i\in L$,
to generate a low-distortion embedding of $\ell_\infty^{|L|}$,
the inequality \eqref{eq: alkjnflkqjflkqjwnf} above must be violated, i.e
$$
\max\limits_{0\leq r\leq \log_2 \sqrt{|L|}}
\frac{m_\delta(y_i,r)}{2^{-2r}\,n}
>
|L|^{\alpha}.
$$
It turns out, however, that that last condition itself leads
to a high (polynomial) embedding distortion (see Lemma~\ref{lem:spikymdelta}),
completing the proof of Theorem~B.

\subsection{Open problems}
The assumptions on the number of generating vectors $N$
used in this paper, exclude the asymptotic
settings $\frac{N}{n}\to 1$ and $\frac{N}{n}\to\infty$.
In particular, the condition $\frac{N}{n}=1+\Omega(1)$
in our argument is employed to obtain in-radius bounds
which, in turn, provide a one-sided control of 
relative magnitudes of $\|\cdot\|_{P_{N,n}}$
and the standard Euclidean norm.
It would be of interest to remove the constraints on relative magnitude
of $N$ and $n$ in the cotype estimate. Further,
one can ask about the optimal cotype constant in our model of randomness,
as well as cotype estimates for random projections of a high-dimensional
cross-polytope.

\begin{problem*}[Gaussian convex hulls with arbitrary number of vertices]
Do there exist $q,C_q<\infty$ such that
for every $N\geq n$, with high probability the space 
$(\R^n,\|\cdot\|_{P_{N,n}})$
has cotype $q$ with constant at most $C_q$?
\end{problem*}

\begin{problem*}[Optimal cotype of Gaussian convex hulls]
Is it true that for every constant $\varepsilon>0$
there is $C_\varepsilon<\infty$ such that
with high probability
the space $(\R^n,\|\cdot\|_{P_{N,n}})$
has cotype $2+\varepsilon$ with constant at most $C_\varepsilon$?
\end{problem*}

\begin{problem*}[Random projections of cross-polytopes]
Let $N\geq n$, and denote by $\tilde P_{N,n}$
the orthogonal projection of the standard cross-polytope in $\R^N$
onto a uniform random $n$--dimensional subspace. Is it true that
$\tilde P_{N,n}$ has finite (dimension-independent) cotype with probability $1-o(1)$?
\end{problem*}

\bigskip

{\bf Acknowledgments.}
K.T. is partially supported by the NSF grant DMS 2331037.

\section{Preliminaries}

\subsection{Notation and definitions}

Let us recall standard definitions from convex geometry and high-dimensional probability,
and introduce some notation that will be used throughout the paper.

\begin{definition}[Unit vectors]
Everywhere in this paper, the term ``unit vector'' refers
to a vector of Euclidean length one.
\end{definition}

\begin{definition}\label{def:restriction}
Given an $N$--dimensional vector $v$ and a
non-empty subset $J\subset[N]$, denote by $v_J$
the subvector of $v$ with coordinates indexed over $J$.
Similarly,
given a matrix $M$ with $N$ columns and any non-empty subset $J\subset[N]$,
denote by $M_J$ the submatrix of $M$ obtained by deleting columns
indexed over $[N]\setminus J$.
\end{definition}

\begin{definition}[Euclidean ball]
Given $r\geq 0$, denote by $B_2^n(r)$
the Euclidean ball in $\R^n$ of radius $r$ centered at the origin.
\end{definition}

\begin{definition}[Polytope in-radius]
Given a closed origin-symmetric polytope $P$ in $\R^n$,
the {\it in-radius} $r(P)$ of $P$ is the largest number $r\geq 0$
such that $B_2^n(r)\subset P$.
\end{definition}

\begin{definition}[Compressible and incompressible vectors]\label{def: compincomp}
Let $\delta,\rho\in(0,1)$ be parameters. A unit vector $x\in\R^n$
is {\it $(\delta,\rho)$--compressible}, if the Euclidean distance
from $x$ to the collection of all $\delta\,n$--sparse
vectors in $\R^n$ is at most $\rho$.
Any unit vector which is not $(\delta,\rho)$--compressible,
is called {\it $(\delta,\rho)$--incompressible}.
\end{definition}


\begin{definition}[Subgaussian variables]
We say that a random variable $\xi$ is {\it $C$--subgaussian}
for some $C>0$ 
if $\Exp\,\exp(\xi^2/C^2)\leq 2$.
\end{definition}

\begin{definition}[Projections]
Given a linear subspace $E$ of some Euclidean space,
denote by $\Proj_E$ the orthogonal projection onto $E$.
\end{definition}

\begin{definition}[Grassmannian]
Given parameters $1\leq d\leq n$, denote by
$G_{n,d}$ the collection of all $d$--dimensional subspaces
in $\R^n$, with the metric
$$
{\rm dist}_{G_{n,d}}(E,F)
=\|\Proj_E-\Proj_F\|,\quad E,F\in G_{n,d},
$$
where $\|\cdot\|$ denotes the spectral norm.
\end{definition}

\begin{definition}[Asymptotic notation]
For two quantities $f,g$ depending implicitly on a parameter $n\to\infty$,
we write $f=O(g)$ if there is a constant $C>0$ independent of $n$ such that
$|f|\leq C|g|$ for all $n$.
Further, $f=o(g)$ if $\lim\limits_{n\to\infty}(f/g)=0$.
We write $f=\Omega(g)$ if $g=O(f)$ and $f=\omega(g)$ whenever $g=o(f)$.
Finally, if $f=O(g)$ and $g=O(f)$ then we write $f=\Theta(g)$.
\end{definition}

\subsection{Global objects}

Here, we identify the global objects
participating in the main results --- Theorems~A and~B:

\begin{itemize}

\item {\bf Vectors $X_i$.}
We assume that $X_1,\dots,X_N$ are i.i.d standard Gaussian vectors 
in $\R^n$ (with the dimension $n$ always clear from context).

\item {\bf Matrix $A$.} We let $A$ be the $n\times N$ matrix with columns $X_1,\dots,X_N$
(dimensions of the matrix are always clear from context).

\item {\bf Polytope $P_{N,n}$.} We define
$$
P_{N,n}:=\conv\,\big\{
\pm X_i,\quad i\in[N]\big\}.
$$

\end{itemize}

\subsection{A sufficient condition for near-$\ell_\infty$ subspaces}
Our study of $\ell_\infty$--subspaces of $(\R^n,\|\cdot\|_{P_{N,n}})$
will be considerably simplified with help of the following lemma:
\begin{lemma}\label{lem: ellinftyreg}
There are universal constants
$c_{\text{\tiny\ref{lem: ellinftyreg}}},
C_{\text{\tiny\ref{lem: ellinftyreg}}}>0$
with the following property.
Let $\|\cdot\|$ be an arbitrary norm in $\R^n$,
let $2\leq k\leq n/2$, let $E$
be a $k$--dimensional linear subspace of $(\R^n,\|\cdot\|)$,
and let $\rho$ be the Banach--Mazur distance from
$E$ (with the induced norm) to $\ell_\infty^k$.
Then there exist at least $\tilde k:=\lfloor c_{\text{\tiny\ref{lem: ellinftyreg}}}\,k/\log k\rfloor$
vectors $y_1,y_2,\dots,y_{\tilde k}$
in $E$ and a number $\delta>0$ with the following property:
\begin{itemize}
    \item $\|y_i\|\in[\delta,2\delta]$
    for every $i\leq \tilde k$;
    \item Each $y_i$ is a unit (Euclidean) vector in $\R^n$;
    \item For every choice of numbers $v=(v_i)_{i\leq \tilde k}$,
    $$
    \Big\|
    \sum_{i=1}^{\tilde k}v_i\,y_i
    \Big\|\leq C_{\text{\tiny\ref{lem: ellinftyreg}}}\,\delta\,\rho\,\|v\|_\infty.
    $$
\end{itemize}
\end{lemma}
\begin{proof}
  Let $K_E$ be the unit ball in $E$ with respect to the induced norm.
  By the definition of Banach--Mazur distance, there exists a centrally
  symmetric parallelepiped $P\subset E$ such that
  \[
  \frac{1}{\rho}P\subset K_E\subset P.
  \]
  In particular, there exist linearly independent unit vectors $u_1,\dots,u_k\in E$ and $r_1,\dots,r_k>0$
  such that
  $$
    P = \Big\{\sum_{i=1}^k a_i u_i:\ |a_i|\le r_i\Big\}.
  $$
  Let $d$ be the largest integer such that $2d \le k$, and let us assume without loss of generality that  
\[
  r_1 \le r_2 \le \dots \le r_{2d}.
\]
Choose a number \(r\) as a median of this sequence, i.e.\
\[
  r_{d} \;\le\; r \;\le\; r_{d+ 1}.
\]

For every \(\alpha\in[d]\) there exists a unit vector $x_\alpha$
representable as $x_\alpha = a_\alpha u_\alpha + b_\alpha u_{d+\alpha}$
for some numbers $a_\alpha$ and $b_\alpha$, such that 
$$
    \min \left\{ \frac{r_\alpha}{|a_\alpha|}, \frac{r_{d+\alpha}}{|b_\alpha|} \right\} = r.
$$
To see that, consider an arbitrary continuous mapping $\theta\in[0,1]\to a_\alpha(\theta) u_\alpha + b_\alpha(\theta)u_{d+\alpha}\in S^{n-1}$
satisfying $a_\alpha(0)=b_\alpha(1)=1$ and $a_\alpha(1)=b_\alpha(0)=0$.
Then $\min \big\{ \frac{r_\alpha}{|a_\alpha(\theta)|}, \frac{r_{d+\alpha}}{|b_\alpha(\theta)|} \big\}$ is a continuous function of $\theta$, which has value $r_\alpha$ at $\theta=0$ and $r_{d+\alpha}$ for $\theta=1$. The continuity ensures
existence of the desired vector $x_\alpha$.  

Denote by $H$ the linear span of $x_\alpha$, $\alpha\leq d$.
By construction,
\[
  \Bigl\{\;\sum_{\alpha=1}^{d} v_\alpha x_\alpha
         : v_\alpha\in[-r,r],\;\;\alpha\leq d\Bigr\}
  \;=\;
  \Bigl\{\;\sum_{\alpha=1}^{d} v_\alpha
            \bigl(a_\alpha\,u_\alpha
                  +b_\alpha\,u_{d+\alpha}\bigr)
         : v_\alpha\in[-r,r],\;\;\alpha\leq d\Bigr\}
  \;=\;
  P \cap H,
\]
and, in particular, for every $v\in \R^d$,
$$
  \Big\| \sum_{\alpha=1}^{d} v_\alpha x_\alpha\Big\|_P = \frac{\|v\|_\infty}{r}.
$$
  In view of the inclusions $(1/\rho)P\subset K_E\subset P$, we obtain
  \[
  \tfrac{1}{r}\|v\|_\infty
  \le
  \Big\|\sum_{\alpha=1}^{d} v_\alpha x_\alpha\Big\|
  \le
  \tfrac{1}{r}\rho\,\|v\|_\infty
  \quad\text{for all }v\in\R^{d},
  \]
  and $\|x_i\|\in[1/r,\rho/r]$ for all $i\le d$.

  Since the Banach--Mazur distance between any two $k$-dimensional symmetric
  convex bodies is at most $k$, we have $\rho\le k$.
  By considering
  a dyadic partition of $\R_+$ and
  by the pigeonhole principle,
  there is an integer $j_0\geq 1$ and a set $I\subset[d]$ such that
  \[
  |I|\ge \frac{d}{\lceil 1+\log_2 k\rceil},
  \qquad
  \|x_i\|\in[2^{j_0-1}/r,2^{j_0}/r]
  \ \text{for all }i\in I.
  \]
  By choosing $c_{\text{\tiny\ref{lem: ellinftyreg}}}$ sufficiently small,
  we guarantee existence of distinct indices $i_1,\dots,i_{\tilde k}\in I$, where
  \[
  \tilde k=\Big\lfloor c_{\text{\tiny\ref{lem: ellinftyreg}}}\,
  \frac{k}{\log k}\Big\rfloor.
  \]
  Set
  \[
  y_m:=x_{i_m},
  \quad m=1,\dots,\tilde k,
  \qquad
  \delta:=\frac{
  2^{j_0-1}}{r}.
  \]
  Then each $y_m$ is a vector of unit Euclidean length and
  \[
  \|y_m\| 
  \in[\delta,2\delta].
  \]
  Finally, for every $v=(v_m)_{m\le\tilde k}$,
  \[
  \Big\|\sum_{m=1}^{\tilde k} v_m y_m\Big\|
  =
  \Big\|\sum_{m=1}^{\tilde k} v_m x_{i_m}\Big\|
  \le
  \frac{\rho}{r}\,\|v\|_\infty
  \le
  \rho\,\delta\,\|v\|_\infty,
  \]
  since $\delta=2^{j_0-1}/r\ge1/r$.
\end{proof}

\subsection{Nets on the Grassmannian}
As one of technical aspects of our argument, we consider a discretization
of the Grassmannian of linear subspaces of $\R^n$. The next lemma is a corollary of 
a classical result \cite{Szarek81} of Szarek.
An interesting feature of the lemma is a representation formula \eqref{eq: grass 209548y3}
for the orthogonal projection onto arbitrary subspace of $\R^n$
in terms of projections onto the subspaces forming the discretization.

\begin{lemma}[Nets on the Grassmannian]\label{lem:Grassnets}
For every $1\leq d\leq n/2$, and every
$\varepsilon\in(0,1/4]$
there is a collection of at most 
$$
M\leq \left( \frac{C_{\text{\tiny\ref{lem:Grassnets}}}}{\varepsilon} \right)^{d(n - d)}
$$
subspaces $\{E_1, \dots, E_M\}\subset G_{n,d}$
such that
\[
\forall F \in G_{n,d} \;\; \exists\; i \le M \text{ satisfying } \|\Proj_F - \Proj_{E_i}\| \le \varepsilon,
\]
and, moreover,
for every $F \in G_{n,d}$, there is a sequence of indices
$i_1,i_2,\dots$ in $[M]$ and operators $D_1,D_2,\dots:\R^n\to\R^n$
such that $\|D_j\| \le 2\, \varepsilon^{\,j-1}$, $j\geq 1$, and
\begin{equation}\label{eq: grass 209548y3}
\Proj_{F}= \sum_{j=1}^\infty D_j \Proj_{E_{i_j}}.
\end{equation}
Here, $C_{\text{\tiny\ref{lem:Grassnets}}}\geq 1$ is a universal constant
and $\|\cdot\|$ denotes the standard spectral norm.
\end{lemma}
\begin{proof}
The first part of the lemma regarding the size of 
$\varepsilon$--nets, is proved in \cite{Szarek81}
(in fact, in a much more general form),
and we turn to verifying the second part.
For simplicity, we will write $\Proj_i$
in place of $\Proj_{E_i}$ below.  

For a moment, fix any pair of subspaces 
$E, F \in G_{n,d}$ such that 
\[
\|\Proj_F - \Proj_E\| \le \varepsilon.
\]
Note that for any vector $v\in (E+F)^\perp$,
$\Proj_F(v) - \Proj_E(v)=0$, and hence
$$
    (\Proj_F - \Proj_E) = (\Proj_F - \Proj_E)\Proj_{E+F}\,.
$$
It will be useful to think of $E+F$ as the sum of two {\it orthogonal}
subspaces $E$ and $\tilde F$, where $\tilde F\perp E$
and $\dim F\leq d$.
Then,
$$
    \Proj_{E+F} = \Proj_{E+F}\Proj_E +
    \Proj_{E+F} \Proj_{E^\perp \cap (E+\tilde F)}
    +\Proj_{E+F}\Proj_H,
$$
where $H$ is the orthogonal complement of $E$ and $E^\perp \cap (E+F)=\tilde F$,
so that $H\subset E^\perp,\tilde F^\perp$ and hence $\Proj_{E+F}\Proj_H=0$.
Thus,
$$
\Proj_{E+F} = \Proj_{E+F}\Proj_E +
\Proj_{E+F} \Proj_{E^\perp \cap (E+F)}.
$$
If the dimension of $\tilde F=E^\perp \cap (E+F)$ is strictly less than $d$,
we arbitrarily define a linear subspace
$\tilde F\subset F^*\subset E^\perp$ of dimension exactly $d$;
otherwise we just let $F^*:=\tilde F$.
Further, let $F'$ be the orthogonal complement of $\tilde F$ within $F^*$.
Then, by definition, $F'\perp \tilde F$ and $F'\perp E$,
so that $\Proj_{E+F}\Proj_{F'}=0$.
We conclude that 
$$
    \Proj_{E+F} \Proj_{E^\perp \cap (E+F)} = \Proj_{E+F} \Proj_{F^*},
$$
and therefore
\begin{equation}\label{eq:ojhbakfjhbasfksdjhbfksdjh}
\Proj_F - \Proj_E = 
(\Proj_F - \Proj_E)\Proj_{E+F}\Proj_E
+(\Proj_F - \Proj_E)\Proj_{E+F} \Proj_{F^*}    
\end{equation}

\medskip

Now, fix any $F \in G_{n,d}$ and set $F_0 = F$.  
Choose an index $i_1 \le M$ such that
\[
\|\Proj_{F_0} - \Proj_{i_1}\| \le \varepsilon.
\]
We can write, using \eqref{eq:ojhbakfjhbasfksdjhbfksdjh}:
\[
\Proj_{F_0} - \Proj_{i_1} 
= A_1 \Proj_{i_1} + B_1 \Proj_{F_1},
\]
where $F_1$ is some $d$-dimensional subspace, and
\[
E_{i_1}^\perp \cap (E+F_0)
\subset F_1 \subset E_{i_1}^\perp,
\quad \text{and} \quad
\|A_1\|, \|B_1\| \le \varepsilon.
\]
Equivalently, we may write
\[
\Proj_{F_0} = \tilde A_1 \Proj_{i_1} + B_1 \Proj_{F_1},
\]
where $\|\tilde A_1\| \le 1 + \varepsilon \le 2$ and $\|B_1\| \le \varepsilon$.  
Applying the above construction to subspace $F_1$,
we get
$$
\Proj_{F_1} = \tilde A_2 \Proj_{i_2} + B_2 \Proj_{F_2},
$$
where $i_2\leq M$ and
$\|\tilde A_2\| \le 1 + \varepsilon \le 2$, $\|B_2\| \le \varepsilon$,
and $F_2$ is some $d$--dimensional subspace.
Rearranging, we get
$$
\Proj_{F_0} = \tilde A_1 \Proj_{i_1} +
B_1 \tilde A_2 \Proj_{i_2} + B_1 B_2 \Proj_{F_2}.
$$
Repeating this construction recursively yields
\[
\Proj_{F}
  = \sum_{j=1}^\infty
      \Bigg( \prod_{\ell=1}^{j-1} B_\ell \Bigg)
      \tilde A_j \Proj_{i_j}.
\]
It remains to define for each $j$,
$$
D_j: = \Bigg( \prod_{\ell=1}^{j-1} B_\ell \Bigg)
      \tilde A_j,
$$
so that our estimates on the spectral norms of $B_\ell$'s and $\tilde A_j$ imply
$\|D_j\| \le 2\, \varepsilon^{\,j-1}$.  
\end{proof}

\subsection{Random sign combinations}
In this section, we consider certain properties
of random vectors of the form $\sum_{i=1}^k \sigma_i y_i$,
where $(y_1,\dots,y_k)$
is a $k$--tuple of fixed unit vectors in $\R^n$,
and $\sigma_1,\dots,\sigma_k$ are i.i.d random signs.
As an immediate observation, for every choice of unit vectors $y_1,\dots,y_k$,
$$
\Exp\,\Big\|\sum_{i=1}^k \sigma_i y_i\Big\|_2^2
=k.
$$
The covariance matrix $\Sigma$
of the vector $\sum_{i=1}^k \sigma_i y_i$ is given by
$$
\Sigma=\Exp\,\Big(\sum_{i=1}^k \sigma_i y_i\Big)
\Big(\sum_{i=1}^k \sigma_i y_i\Big)^\top
=\sum_{i=1}^k y_i\,y_i^\top;
$$
in particular, ${\rm tr}\,(\Sigma)=k$, leading to the earlier observation.
The problem we study in this subsection is to identify
necessary conditions on $y_1,\dots,y_k$, a vector $Z\in\R^n$, and a parameter $t\geq 1$
so that probability
$$
\Prob_\sigma
\Big\{
\big|\big\langle Z,\sum_{i=1}^k \sigma_i y_i\big\rangle\big|
\geq t\,\sqrt{k}
\Big\}
$$
has order $k^{-O(1)}$.
The resulting estimates will be employed in a discretization/union bound argument needed to reveal certain structural properties of subspaces of $(\R^n,\|\cdot\|_{P_{N,n}})$.


\begin{lemma}\label{lem:dyaddecomp}
Let $2\leq k\leq n$, let $y_1,y_2,\dots,y_k$
be vectors in $\R^n$ of unit Euclidean length,
and let $Z$ be any fixed vector in $\R^n$.
Further, let $\sigma=(\sigma_1,\dots,\sigma_k)$
be uniform random vector with $\pm 1$ components,
and assume that for some $t\geq 1$
$$
\Prob_\sigma
\Big\{
\big|\big\langle Z,\sum_{i=1}^k \sigma_i y_i\big\rangle\big|
\geq t\,\sqrt{k}
\Big\}\geq \frac{1}{k^{100}}.
$$
Denote by $\Sigma$ the covariance matrix of $\sum_{i=1}^k \sigma_i y_i$.
For each integer $-\infty<p<\infty$,
let $E_p$ 
be the linear span of the eigenvectors of $\Sigma$
corresponding to the eigenvalues in the
range $(2^{p},2^{p+1}]$.
Then 
\begin{itemize}
\item[(a)] either the orthogonal projection of $Z$ onto the linear span of $\{y_1,\dots,y_k\}$
is of length at least $$\frac{c_{\text{\tiny\ref{lem:dyaddecomp}}}\,t\,\sqrt{k}}{\log^{0.5} k},$$
\item[(b)] or, there is $1\leq p\leq\log_2 k$
such that $E_p\neq \{0\}$ and the orthogonal projection of $Z$
onto $E_p$ has Euclidean length at least $$\frac{c_{\text{\tiny\ref{lem:dyaddecomp}}}\,t\,\sqrt{\dim E_p}}{\log^{1.5} k}.$$
\end{itemize}
Here, $c_{\text{\tiny\ref{lem:dyaddecomp}}}>0$
is a universal constant.
\end{lemma}
\begin{proof}
Denote by $\Proj$ the orthogonal projection
onto the linear span of $\{y_1,\dots,y_k\}$, and
for all $-\infty< p< \infty$, let $\Proj_p$
be the orthogonal projection onto $E_p$.
Note that whenever the inner product $\big|\big\langle Z,\sum_{i=1}^k \sigma_i y_i\big\rangle\big|
\geq t\,\sqrt{k}$ with a positive probability,
we necessarily have $\|\Proj(Z)\|_2\geq t/\sqrt{k}$ (as the Euclidean length
of $\sum_{i=1}^k \sigma_i y_i$ can never exceed $k$).
For that reason, by choosing an appropriate constant 
$c_{\text{\tiny{\ref{lem:dyaddecomp}}}}>0$, we get that the assertion (a)
above is trivially true for small $k$.
In what follows, we assume that parameter $k$ is large.
Let $\Sigma$ be the covariance matrix of the random vector
$\sum_{i=1}^k \sigma_i y_i$.

For each $-\infty<p<\infty$, $\Proj_p\big(\sum_{i=1}^k \sigma_i y_i\big)$
is a centered random vector in $E_p$
with covariance matrix
$$
\Exp\;\Proj_p\big(\sum_{i=1}^k \sigma_i y_i\big)\;\Proj_p\big(\sum_{i=1}^k \sigma_i y_i\big)^\top
=\Proj_p\,\Sigma\, \Proj_p^\top,
$$
where all non-zero eigenvalues lie in the interval $(2^p,2^{p+1}]$.
Thus, we have
$$
\Exp\,\big\langle Z,\Proj_p\Big(\sum_{i=1}^k \sigma_i y_i\Big)\big\rangle^2
\in (\|\Proj_p(Z)\|_2^2\,2^p,\|\Proj_p(Z)\|_2^2\,2^{p+1}].
$$
On the other hand, in view of Khintchine's inequality (see,
for example, \cite[Section~2.6]{Ver18}),
the random variable
$$
\frac{\big\langle Z,\Proj_p\Big(\sum_{i=1}^k \sigma_i y_i\Big)\big\rangle}
{\sqrt{\Exp\,\big\langle Z,\Proj_p\big(\sum_{i=1}^k \sigma_i y_i\big)\big\rangle^2}}
$$
is $C$--subgaussian, for a universal constant $C>0$.
Combining the last two observations, we get
\begin{equation}\label{aksjnfakfjnsk}
\Prob\Big\{\big|\big\langle Z,\Proj_p
\Big(\sum_{i=1}^k \sigma_i y_i\Big)\big\rangle\big|\geq s\,
\|\Proj_p(Z)\|_2\,2^{p/2}\Big\}
\leq 2\exp(-cs^2),\quad s>0,
\end{equation}
for a universal constant $c>0$.

Write
$$
\big\langle Z,\sum_{i=1}^k \sigma_i y_i\big\rangle
=
\sum\limits_{-\infty< p< \infty}
\big\langle Z,\Proj_p\Big(\sum_{i=1}^k \sigma_i y_i\Big)\big\rangle.
$$
As we remarked before,
the trace of $\Sigma$ equals $k$, implying, in particular,
that $\Sigma$ does not have eigenvalues exceeding $k$.
Thus, $\Proj_p={\bf 0}$ for all $p>\log_2 k$,
and, we can write for a sufficiently small universal constant 
$\tilde c>0$:
\begin{align*}
\Prob
\Big\{
\big|\big\langle Z,\sum_{i=1}^k \sigma_i y_i\big\rangle\big|
\geq t\,\sqrt{k}
\Big\}
&\leq 
\Prob
\Big\{
\big|\big\langle Z,\Proj_p\Big(\sum_{i=1}^k \sigma_i y_i\Big)
\big\rangle\big|
\geq \frac{t\,\sqrt{k}}{2\,\log_2 k}
\mbox{ for some $1\leq p\leq \log_2 k$}
\Big\}\\
&+\sum\limits_{p\leq 0}
\Prob
\Big\{
\big|\big\langle Z,\Proj_p\Big(\sum_{i=1}^k \sigma_i y_i\Big)
\big\rangle\big|
\geq \tilde c\,t\,\sqrt{k}\,2^{p/4}
\Big\}.
\end{align*}
By the assumptions of the lemma,
the left hand side of the last inequality is at least $k^{-100}$.
Therefore, either
$$
\Prob
\Big\{
\big|\big\langle Z,\Proj_p\Big(\sum_{i=1}^k \sigma_i y_i\Big)
\big\rangle\big|
\geq \frac{t\,\sqrt{k}}{2\,\log_2 k}\Big\}
\geq \frac{1}{2k^{100}\log_2 k}\;\;\;
\mbox{for some $1\leq p\leq \log_2 k$},
$$
or there is $p\leq 0$ such that
$$
\Prob
\Big\{
\big|\big\langle Z,\Proj_p\Big(\sum_{i=1}^k \sigma_i y_i\Big)
\big\rangle\big|
\geq \tilde c\,t\,\sqrt{k}\,2^{p/4}
\Big\}\geq \frac{c'\,2^{p/4}}{k^{100}},
$$
for an appropriate universal constant $c'>0$.
In view of \eqref{aksjnfakfjnsk}, we obtain that
either for some $1\leq p\leq \log_2 k$ we have
\begin{equation}\label{aldjhfbajfhbojab}
\tilde C\,\sqrt{\log k}\,
\|\Proj_p(Z)\|_2\,2^{p/2}\geq 
\frac{t\,\sqrt{k}}{2\,\log_2 k},
\end{equation}
or there is $p\leq 0$ such that
\begin{equation}\label{ajhdbfaojhfbsojafhbs}
\tilde C\,\,\sqrt{-p+\log k}\,
\|\Proj(Z)\|_2\,2^{p/2}\geq
\tilde C\,\,\sqrt{-p+\log k}\,
\|\Proj_p(Z)\|_2\,2^{p/2}\geq
\tilde c\,t\,\sqrt{k}\,2^{p/4},
\end{equation}
for a constant $\tilde C>0$.
It remains to note that for every $1\leq p\leq \log_2 k$,
$\dim E_p\leq k\,2^{-p}$ (in view of the relation
${\rm tr}\,(\Sigma)\leq k$), so \eqref{aldjhfbajfhbojab}
implies
$$
\|\Proj_p(Z)\|_2\geq 
\frac{t\,\sqrt{\dim E_p}}{2\,\tilde C\,\log_2 k\;\sqrt{\log k}}.
$$
On the other hand, $\frac{2^{-p/4}}{\sqrt{-p+\log k}}
\geq \frac{c''}{\sqrt{\log k}}$ for all $p\leq 0$
and a universal constant $c''>0$,
so whenever \eqref{ajhdbfaojhfbsojafhbs} holds,
we have
$$
\|\Proj(Z)\|_2\geq
\frac{c''\,\tilde c\,t\,\sqrt{k}}{\tilde C\,\sqrt{\log k}}.
$$
The result follows.
\end{proof}

\subsection{Gaussian vector tuples}
In this subsection, we study
concentration properties of the Gaussian tuple $X_1,\dots,X_N$.
Recall that we denote by $A$ the $n\times N$
random matrix with columns $X_1,\dots,X_N$.
We start with a direct consequence of standard
concentration results on random matrices with
subgaussian entries (see, for example, \cite[Section~4.6]{Ver18}).

\begin{lemma}[Singular‐value stability for sparse column subsets]
  \label{lem:sparse_singular_values}
   For every choice of parameters $K'\geq K>1$ there exist $c_{\text{\tiny\ref{lem:sparse_singular_values}}}\in(0,1/100]$
  depending only on $K',K$ such that the following holds.
  Assume that $n\geq c_{\text{\tiny\ref{lem:sparse_singular_values}}}^{-1}$
  and that $K'\geq \frac{N}{n}\geq K$.
  Define event
  \begin{align*}
  \Omega_{\text{\tiny\ref{lem:sparse_singular_values}}}:=
  &\big\{\|A\| \le 4\sqrt{N}\big\}\,\cap\,
  \Big\{s_{\min}(A^\top) 
\ge \tfrac{1}{2}(\sqrt{N} - \sqrt{n})\Big\}\,\cap\\
  &\bigg\{\mbox{For every non-empty $J\subset[N]$ with $|J|\leq c_{\text{\tiny\ref{lem:sparse_singular_values}}}\,n$},\\
  &\hspace{0.3cm}\sqrt{n}\Bigl(1- 
  \sqrt{{|J|}/{n}}\,\log\frac{n}{|J|}  \Bigr)
      \le
      s_{\min}(A_J)
      \le
      s_{\max}(A_J)
      \le
      \sqrt{n}\Big(1+
      \sqrt{{|J|}/{n}}\,\log\frac{n}{|J|} \Big)
  \bigg\}.
  \end{align*}
Then 
$$
\Prob(\Omega_{\text{\tiny\ref{lem:sparse_singular_values}}})
\geq 1-\frac{1}{n}.
$$
\end{lemma}

The next lemma deals with statistical properties of projections
of the random vectors $X_1,\dots,X_N$,
and is based on a discretization of the Grassmannian
supplied by Lemma~\ref{lem:Grassnets}.

\begin{lemma}
\label{lem:d_dim_proj}
There is a universal constant $C_{\tref{lem:d_dim_proj}}\ge 1$
with the following property.
Let $1\leq d\leq n/2$ and $N\geq n$.
Then with probability at least
$1 - \exp(- C_{\tref{lem:d_dim_proj}}\; dn)$,  
for every $s \ge C_{\tref{lem:d_dim_proj}}$ and
for every $d$-dimensional subspace $F$ of $\R^n$, the number of indices $1 \le j \le N$ such that
\[
\|\Proj_F X_j\|_2 > 8s \sqrt{d},
\]
is bounded above by
$\frac{2\,C_{\tref{lem:d_dim_proj}}^2\,N\, \log_2 s}{s^2}$.
Here, $\Proj_F$ denotes the orthogonal projection onto $F$.
\end{lemma}
\begin{proof}
We will assume that the constant
$C_{\tref{lem:d_dim_proj}}$ is large.
We start by applying Lemma~\ref{lem:Grassnets}
with $\varepsilon:=\frac{1}{4}$,
to obtain a collection of at most 
\begin{equation}\label{eq:ljhbafjhbfkjahsbfkjhfb}
M\leq \left( \frac{C_{\text{\tiny\ref{lem:Grassnets}}}}{\varepsilon} \right)^{d(n - d)}    
\end{equation}
subspaces $\{E_1, \dots, E_M\}\subset G_{n,d}$
such that
\[
\forall F \in G_{n,d} \;\; \exists\; i \le M \text{ satisfying } \|\Proj_F - \Proj_{E_i}\| \le \varepsilon,
\]
and, moreover,
for every $F \in G_{n,d}$, there is a sequence of indices
$i_1,i_2,\dots$ in $[M]$ and operators $D_1,D_2,\dots:\R^n\to\R^n$
such that $\|D_j\| \le 2\, \varepsilon^{\,j-1}$, $j\geq 1$, and
\[
\Proj_{F}= \sum_{j=1}^\infty D_j \Proj_{E_{i_j}}.
\]
In what follows, we denote $\Proj_{E_{i}}$ by $\Proj_{i}$
for all $i\leq M$.

Define event $\Omega$ that for every positive integer $r$ and for each $E_i$ ($1 \le i \le M$),  
the set
\[
  J_{i,r} := \Big\{ j \in [N] : \|\Proj_{i} X_j\|_2 \ge C_{\tref{lem:d_dim_proj}}\, 2^r \sqrt{d} \Big\}
\]
has size
\[
  |J_{i,r}| \le \frac{N\,r}{4^r}.
\]
Since each $\Proj_{i} X_j$ is a standard Gaussian vector in
a $d$--dimensional subspace, its expected Euclidean norm is of order $(1 + o_d(1)) \sqrt{d}$.  
By a standard concentration inequality (see, for example, \cite{Ver18}), and assuming that $C_{\tref{lem:d_dim_proj}}\, 2^r \sqrt{d}\geq
2\,\Exp\,\|\Proj_{i} X_j\|_2$,
we obtain
\[
  \Prob\big\{
  \|\Proj_{i} X_j\|_2 \ge C_{\tref{lem:d_dim_proj}}\, 2^r \sqrt{d} \big\}
  \le \underbrace{\exp\big(-c\, C_{\tref{lem:d_dim_proj}}^2 4^r d \big)}_{:= \Delta_r},
\]
for some universal constant $c > 0$. Thus, $|J_{i,r}|$ is an
$(N,p)$ binomial random variable with $p\leq\Delta_r$, implying
\[
  \Exp|J_{i,r}| \le N \Delta_r.
\]
Further, in view of
Chernoff’s inequality (see \cite[Section~2.3]{Ver18}), for a binomial random variable $Z$ with a mean
$\mu$
we have for every $t>\mu$:
\[
  \Prob\{Z \ge t\}
  \le \Big(\frac{e\mu}{t}\Big)^t.
\]
Applying this with $\mu \leq N \Delta_r$, we obtain, for every
$i \le M$, $r \ge 1$,
\[
  \Prob\!\left\{ |J_{i,r}| \ge \frac{r\,N}{4^r } \right\}
  \leq
  \Big(\frac{e\,4^r\Delta_r}{r}\Big)^{N\,r/4^r}
  \le \exp\!\big( - c'\, C_{\tref{lem:d_dim_proj}}^2 N r d \big),
\]
for some universal constant $c'>0$,
as long as $C_{\tref{lem:d_dim_proj}}$ is sufficiently large.
Taking a union bound over all $i \le M$ and $r \ge 1$, 
and using \eqref{eq:ljhbafjhbfkjahsbfkjhfb}, we obtain
\begin{align*}
  \Prob(\Omega^c)
  &\le \sum_{r=1}^\infty M \exp\!\big(-c' C_{\tref{lem:d_dim_proj}}^2 N r d \big)\le \exp(-C_{\tref{lem:d_dim_proj}}\,Nd),
\end{align*}
for $C_{\tref{lem:d_dim_proj}}$ sufficiently large.

\medskip

Now, condition on any realization of $X_1,\dots,X_N$
from $\Omega$.
Fix any $s\geq C_{\tref{lem:d_dim_proj}}$, and
let $r_0$ be the integer such that 
$$
  C_{\tref{lem:d_dim_proj}}\, 2^{r_0} \le s < C_{\tref{lem:d_dim_proj}}\, 2^{r_0+1}.
$$
Further, fix any $d$--dimensional subspace $F$.
Consider the decomposition
\[
  \Proj_F = \sum_{\alpha = 1}^\infty D_\alpha \Proj_{i_\alpha},
\]
as established above, so that $\|D_\alpha\|\leq
2\,\big(\frac{1}{4}\big)^{\alpha-1}$, $\alpha\geq 1$.
Observe that, by the definition of $\Omega$,
\[
  \Bigg| \bigcup_{\alpha \ge 1} J_{i_\alpha,r_0+\alpha} \Bigg|
  \le \sum_{\alpha = 1}^\infty |J_{i_\alpha,r_0+\alpha}|
  \le \sum_{\alpha = 1}^\infty \frac{ N(r_0+\alpha)}{4^{r_0+\alpha}}
  \le  \frac{2 N (r_0+1)}{4^{r_0+1}}
  \le \frac{2C_{\tref{lem:d_dim_proj}}^2\,N\, \log_2 s}{s^2}.
\]
Finally, for any index $j \notin \bigcup_{\alpha \ge 1} J_{i_\ell,r_0+\alpha}$, we have
\[
  \|\Proj_{i_\alpha} X_j\|_2 \le C_{\tref{lem:d_dim_proj}}\, 2^{r_0+\alpha} \sqrt{d}, \quad \forall\, \alpha \ge 1.
\]
Hence, for such $j$,
\begin{align*}
  \|\Proj_F X_j\|_2
  &\le \sum_{\alpha = 1}^\infty \|D_\alpha\| \, \|\Proj_{i_\alpha} X_j\|_2\\
  &\le 2\,\sum_{\alpha = 1}^\infty 
  \Big(\frac{1}{4}\Big)^{\alpha-1} \,C_{\tref{lem:d_dim_proj}} 2^{r_0+\alpha} \sqrt{d}\\ 
  &= 8\,C_{\tref{lem:d_dim_proj}}\, 2^{r_0} \sqrt{d}.
\end{align*}
Therefore, 
\[
  \|\Proj_F X_j\|_2 \le 8s \sqrt{d}
\]
holds for all but at most $\frac{2C_{\tref{lem:d_dim_proj}}^2\,N\, \log_2 s}{s^2}$ indices $j \in [N]$.  
This completes the proof.
\end{proof}

\subsection{Random sign combinations and Gaussian vector tuples}
As a final piece of the preliminaries section,
we consider a corollary of Lemma~\ref{lem:dyaddecomp} and Lemma~\ref{lem:d_dim_proj}:
\begin{lemma}\label{lem:dotprodstat}
For every $K'\geq K>1$ there is a constant
$\tilde C_{\text{\tiny\ref{lem:dotprodstat}}}>
2\,c_{\text{\tiny\ref{lem:sparse_singular_values}}}^{-1}$
depending on $K',K$,
with the following properties. Assume that $K'\geq \frac{N}{n}\geq K$.
Define event $\Omega_{\text{\tiny\ref{lem:dotprodstat}}}$
that $X_j$'s ($1\leq j\leq N$) satisfy all of the following:
\begin{enumerate}
\item 
For every $2\leq k\leq n/2$, every $t\geq
\tilde C_{\text{\tiny\ref{lem:dotprodstat}}}
\log^{2} k$,
and every $k$--tuple $y_1,y_2,\dots,y_k$ of vectors of unit Euclidean length in $\R^n$,
the number of $X_j$'s ($1\leq j\leq N$) such that
$$
\Prob_{\sigma\in\{-1,1\}^k}
\Big\{
\big|\big\langle X_j,\sum_{i=1}^k \sigma_i y_i\big\rangle\big|
\geq t\,\sqrt{k}
\Big\}\geq \frac{1}{k^{100}},
$$
is bounded above by
$$
n\,\frac{\tilde C_{\text{\tiny\ref{lem:dotprodstat}}}\log t\;
\log^{4} k }{t^2}.
$$
\item For every $2\leq k\leq n/2$
and every $k$--tuple $y_1,y_2,\dots,y_k$ of unit vectors in $\R^n$,
the number of $X_j$'s ($1\leq j\leq N$)
whose orthogonal projection onto the linear span of $y_i$'s
has length greater than $\tilde C_{\text{\tiny\ref{lem:dotprodstat}}}\,k^{10}$,
is bounded above by 
$$
n\,k^{-10}.
$$
\end{enumerate}
Then the probability of $\Omega_{\text{\tiny\ref{lem:dotprodstat}}}$
is $1-\exp(-\Omega(n))$.
\end{lemma}
\begin{proof}
In what follows, we assume that constants
$\tilde C_{\text{\tiny\ref{lem:dotprodstat}}}$
and $C_{\text{\tiny\ref{lem:dotprodstat}}}$
are large.
We define $\Omega_{\text{\tiny\ref{lem:dotprodstat}}}$
as the event that the assertion of Lemma~\ref{lem:d_dim_proj}
holds true simultaneously for all $1\leq d\leq n/2$,
that is,
for every $s \ge C_{\tref{lem:d_dim_proj}}$ and
for every $d$-dimensional subspace $F$ of $\R^n$, the number of indices $1 \le j \le N$ such that
\[
\|\Proj_F X_j\|_2 > 8s \sqrt{d},
\]
is bounded above by
$\frac{2\,C_{\tref{lem:d_dim_proj}}^2\,N\, \log_2 s}{s^2}$.
Note that $\Prob(\Omega_{\text{\tiny\ref{lem:dotprodstat}}})=1-\exp(-\Omega(n))$.
In what follows, we condition on arbitrary realization
of $X_1,\dots,X_N$ from $\Omega_{\text{\tiny\ref{lem:dotprodstat}}}$.

\medskip

Fix any $k$--tuple $(y_1, \dots, y_k)$ of unit vectors in $\R^n$, and let 
\[
  t \ge \tilde{C}_{\text{\tiny\ref{lem:dotprodstat}}}
  \, \log^{2} k.
\]
Define the set
\[
  J :=
  \Bigg\{
    j \in [N] : 
    \mathbb{P}_{\sigma \in \{-1,1\}^k}
    \Big(
      \big| \big\langle X_j, \sum_{i=1}^k \sigma_i y_i \big\rangle \big|
      \ge t \sqrt{k}
    \Big)
    \ge \frac{1}{k^{100}}
  \Bigg\}.
\]

\medskip

By Lemma~\ref{lem:dyaddecomp}, for each $j \in J$, one of the following two alternatives holds:
\begin{itemize}
  \item[\textbf{(i)}] 
  \(
    \|\Proj_E(X_j)\|_2
    \ge
    \dfrac{c_{\tref{lem:dyaddecomp}} \, t \, \sqrt{\dim E}}
    {\sqrt{\log k}},
  \)
  where $E = \mathrm{span}\{y_1, \dots, y_k\}$; or
  \item[\textbf{(ii)}]
  there exists $1 \le p \le \log_2 k$ such that
  \(
    \|\Proj_{E_p}(X_j)\|_2
    \ge
    \dfrac{c_{\tref{lem:dyaddecomp}} \, t \,
    \sqrt{\dim E_p}}{(\log k)^{3/2}},
  \)
  where $E_p$ is a subspace of dimension at most $k\leq n/2$,
  defined as in Lemma~\ref{lem:dyaddecomp}.
\end{itemize}

\medskip

In view of the assumptions on $t$,
$\frac{c_{\tref{lem:dyaddecomp}} \, t}{(\log k)^{3/2}}
\geq 8\,C_{\tref{lem:d_dim_proj}}$.
Applying Lemma~\ref{lem:d_dim_proj}, we obtain that 
for every $1\leq p\leq \log_2 k$,
the size
of the set
\[
  \bigg\{
    j \in [N] :
    \|\Proj_{E_p}(X_j)\|_2
    \ge
    \frac{c_{\tref{lem:dyaddecomp}} \, t \, \sqrt{\dim E_p}}
    {(\log k)^{3/2}}
  \bigg\}
\]
is bounded above by
\begin{equation}\label{eq:ljhablsjhbsfljahbfdlskjfhb}
  \frac{
    2\,C_{\tref{lem:d_dim_proj}}^2
    \log_2\!\left(
      \frac{c_{\tref{lem:dyaddecomp}} \, t}{8\,(\log k)^{3/2}}
    \right)
  }{
    \left(
      \frac{c_{\tref{lem:dyaddecomp}} \, t}{8\,(\log k)^{3/2}}
    \right)^2
  }\,
  N
  \;\le\;
  \frac{128\,C_{\tref{lem:d_dim_proj}}^2}{c_{\tref{lem:dyaddecomp}}^2}\;
  \frac{
    \log_2(t) \, (\log k)^3
  }{t^2}
  \,N.
\end{equation}
Similarly, the size
of the set
\[
  \bigg\{
    j \in [N] :
    \|\Proj_E(X_j)\|_2
    \ge
    \dfrac{c_{\tref{lem:dyaddecomp}} \, t \, \sqrt{\dim E}}
    {\sqrt{\log k}}
  \bigg\}
\]
is bounded above by \eqref{eq:ljhablsjhbsfljahbfdlskjfhb}.
Consequently,
\[
  |J|
  \le (\log_2 k + 1) \frac{128\,C_{\tref{lem:d_dim_proj}}^2}{c_{\tref{lem:dyaddecomp}}^2}\;
  \frac{
    \log_2(t) \, (\log k)^3
  }{t^2}
  \,N,
\]
implying the first part of the lemma.
Finally, part (2) follows directly from Lemma~\ref{lem:d_dim_proj} upon choosing $\tilde{C}_{\tref{lem:dotprodstat}}$ sufficiently large.
\end{proof}

\section{Polytopes $P_{N,n}$}

Having considered preparatory statements dealing with statistical properties
of Gaussian vector tuples and random sign combinations in Preliminaries,
here we delve into the study of the random polytopes $P_{N,n}$;
the main outcome of this section being the proof of Theorem~B from the introduction.

\begin{definition}\label{def: beta 09812720894}
For every vector $y$ in $\R^n$,
let $\beta(y)$ be a vector in $\R^N$
such that
$$
y=\sum\limits_{j=1}^N \beta_j(y)\,X_j,
$$
and
$$
\|\beta(y)\|_1=\|y\|_{P_{N,n}}.
$$
The vector $\beta(y)$ may be not uniquely defined;
in that case we fix a realization of $\beta(y)$
satisfying the above conditions, in such a way that the mapping
$y\to \beta(y)$ is Borel measurable.
\end{definition}

The next definition provides a notational
basis for multiscale arguments applied
to coefficient vectors $\beta(y)$:

\begin{definition}\label{def: myr 334q}
Let $y$ be a vector in $\R^n$, and let $\delta\in(0,1]$
be a parameter.
We define
$$
m_\delta(y,r):=\Big|\Big\{
j\leq N:\;\big|\beta_j(y)\big|\in
\frac{\delta}{n}\big(2^{r-1},2^r\big]\Big\}\Big|,
\quad r\geq 1,
$$
and
$$
m_\delta(y,0):=\Big|\Big\{
j\leq N:\;\big|\beta_j(y)\big|\leq 
\frac{\delta}{n}\Big\}\Big|.
$$
\end{definition}

\subsection{In-radius and structure of the kernel of $A^\top$}

\begin{lemma}[In-radius estimates]\label{PN:inradius}
For every $K'\geq K>1$ there are constants
$C_{\text{\tiny\ref{PN:inradius}}},c_{\text{\tiny\ref{PN:inradius}}}>0$
depending only on $K,K'$
with the following property. 
Assume $n\geq c_{\text{\tiny\ref{lem:sparse_singular_values}}}^{-1}$ and $K'\geq N/n\geq K$.
Condition on any realization of $X_1,\dots,X_N$
from event $\Omega_{\text{\tiny\ref{lem:sparse_singular_values}}}$
(defined in Lemma~\ref{lem:sparse_singular_values}).
Then
\begin{itemize}
    \item The in-radius $r(P_{N,n})$
    satisfies $c_{\text{\tiny\ref{PN:inradius}}}\leq r(P_{N,n})
    \leq C_{\text{\tiny\ref{PN:inradius}}}$;
    \item For every non-empty $I\subset[N]$ with $|I|\leq c_{\tref{lem:sparse_singular_values}}\,n$,
    the in-radius of
    $\conv\,\big\{\pm X_j,\quad j\in I\big\}$
    (viewed as a non-degenerate polytope in ${\rm span}\{X_j\,:j \in I\}$)
    satisfies
    $$
    \frac{1}{2}\sqrt{\frac{n}{|I|}}
    \leq 
    r\big(\conv\,\big\{\pm X_j,\quad j\in I\big\}\big)
    \leq 
    2\sqrt{\frac{n}{|I|}}.
    $$
\end{itemize}
\end{lemma}
\begin{proof}

For a symmetric convex body $L \subset \mathbb{R}^n$, its in-radius can be expressed as
\[
r(L) = \sup\{r > 0 : r B_2^n \subset L\} 
= \min_{y \in S^{n-1}} h_L(y),
\]
where $h_L(y) = \sup_{x \in L} \langle x, y \rangle$ is the support function.
For $P_{N,n} = \mathrm{conv}\{\pm X_j,\;1\leq j\leq N\}$, we thus have
\[
r(P_{N,n}) = \min_{y \in S^{n-1}} \max_{j \le N} |\langle X_j, y \rangle|.
\]

\step{Lower bound on in-radius}
For any $y \in \mathbb{R}^n$,
\[
\max_{j \le N} |\langle X_j, y \rangle|
= \|A^\top y\|_\infty
\ge \frac{1}{\sqrt{N}} \|A^\top y\|_2
\ge \frac{s_{\min}(A^\top)}{\sqrt{N}} \|y\|_2,
\]
where, by our conditioning on
$\Omega_{\text{\tiny\ref{lem:sparse_singular_values}}}$,
\[
s_{\min}(A^\top) 
\ge \tfrac{1}{2}(\sqrt{N} - \sqrt{n})
\ge \tfrac{1}{2}(1 - 1/\sqrt{K}) \sqrt{N}
=: c_*(K) \sqrt{N}.
\]
Taking the minimum over $y \in S^{n-1}$ gives
\[
r(P_{N,n}) \ge c_*(K) =: c_{\text{\tiny\ref{PN:inradius}}}.
\]

\step{Upper bound on in-radius}
By the Blaschke--Santal\'o inequality for symmetric convex bodies,
\[
|P_{N,n}|\,|P_{N,n}^\circ| \le |B_2^n|^2,
\]
where $P_{N,n}^\circ$ is the {\it polar body} for the polytope $P_{N,n}$,
and where we use the vertical bars to denote the standard Lebesgue volume.
Since $r(P_{N,n})^n |B_2^n| \le |P_{N,n}|$, we obtain
\begin{equation}\label{eq:jhbfhpfjnlkjnfalsjkfb}
r(P_{N,n}) \le \frac{|B_2^n|^{1/n}}{|P_{N,n}^\circ|^{1/n}}
\le  \frac{C}{|\sqrt{n}\, P_{N,n}^\circ|^{1/n}},    
\end{equation}
for some universal
constant $C > 0$, since $|B_2^n|^{1/n} \asymp n^{-1/2}$. Now, by definition,
\[
P_{N,n}^\circ = \{y \in \mathbb{R}^n : |\langle X_j, y \rangle| \le 1, \, j \in [N]\}.
\]
Let $\gamma$ denote the
standard Gaussian measure in $\mathbb{R}^n$.
Since the corresponding Gaussian density can be (roughly)
bounded above by $1$
everywhere in $\R^n$,
\[
|\sqrt{n}\, P_{N,n}^\circ|
\ge \gamma(\sqrt{n}\, P_{N,n}^\circ)
= \gamma\!\left(\bigcap_{j=1}^N 
\left\{y : \Big|\Big\langle \frac{X_j}{\sqrt{n}}, y \Big\rangle\Big| \le 1\right\}\right).
\]
Applying the Sidak--Khatri lemma 
\cite{Khatri,Sidak},
\[
\gamma\!\left(\bigcap_{j=1}^N 
\left\{y : \Big|\Big\langle \frac{X_j}{\sqrt{n}}, y \Big\rangle\Big| \le 1\right\}\right)
\ge \prod_{j=1}^N \gamma\!\left(
\Bigl\{y : \Big|\Big\langle \frac{X_j}{\sqrt{n}}, y \Big\rangle\Big| \le 1\Bigr\}
\right).
\]
Note that, by conditioning on 
$\Omega_{\text{\tiny\ref{lem:sparse_singular_values}}}$,
we have $\max_j \|X_j\|_2 \le 2\sqrt{n}$. 
Thus, 
the Gaussian measure of each slab 
$$
\Bigl\{y : \Big|\Big\langle \frac{X_j}{\sqrt{n}}, y \Big\rangle\Big| \le 1\Bigr\}
$$
is bounded below by a universal constant $c_0>0$.
Hence, 
$|\sqrt{n} P_{N,n}^\circ| \ge c_0^N$,
and in view of \eqref{eq:jhbfhpfjnlkjnfalsjkfb} we get
\[
r(P_{N,n}) \le \frac{C}{c_0^{N/n}} \le \frac{C}{c_0^{K'}} 
=: C_{\text{\tiny\ref{PN:inradius}}}.
\]

\step{In-radius of $\conv\,\big\{\pm X_j,\quad j\in I\big\}$}
Assume that $I \subset [N]$ is a non-empty subset with $|I| =: m \le c_{\tref{lem:sparse_singular_values}}\,n$.
In view of conditioning on 
$\Omega_{\text{\tiny\ref{lem:sparse_singular_values}}}$,
the $n\times I$ submatrix $A_I$ of $A$ satisfies
\begin{equation}\label{eq:sparse-sv}
\frac{1}{2} \sqrt{n} \|z\|_2 
\le \|A_I z\|_2 
\le 2 \sqrt{n} \|z\|_2,
\quad z \in \mathbb{R}^I.
\end{equation}
Denote $P_I:=\conv\,\big\{\pm X_j,\quad j\in I\big\}$.
Since $A_I(B_1^m) = P_I$, we have
\[
\frac{1}{\sqrt{m}} A_I(B_2^m) 
\subset
A_I(B_1^m)= P_I,
\]
in particular, implying that the in-radius of $P_I$
is bounded below by the in-radius of $\frac{1}{\sqrt{m}} A_I(B_2^m)$.
The latter, in turn, is bounded below using~\eqref{eq:sparse-sv}:
$$
r\Big(
\frac{1}{\sqrt{m}} A_I(B_2^m)
\Big)
= \min\limits_{z\in S^{m-1}}  \|A_Iz\|_2
\geq \frac{1}{2} \sqrt{\frac{n}{m}}.
$$
For an upper bound on the in-radius, take $z = (\pm \frac{1}{m}, \dots, \pm \frac{1}{m}) \in B_2^m$, so $\|z\|_1 = 1$ and $\|A_I z\|_{P_I} = 1$. Then from~\eqref{eq:sparse-sv},
\[
\Big\|\frac{1}{\sqrt{m}} A_I z\Big\|_2 
\le 2 \sqrt{\frac{n}{m}},
\]
so some boundary point of $P_I$ has Euclidean norm at most
$2 \sqrt{\frac{n}{m}}$.
\end{proof}

\begin{lemma}[Zero combinations of $X_i$'s are incompressible]\label{PN:incompcomb}
For every $K'\geq K>1$ there are constants
$\delta_{\text{\tiny\ref{PN:incompcomb}}},
\rho_{\text{\tiny\ref{PN:incompcomb}}}\in(0,1)$
depending only on $K,K'$
with the following property.
Assume $n\geq c_{\text{\tiny\ref{lem:sparse_singular_values}}}^{-1}$ and
$K'\geq N/n\geq K$.
Then, conditioned on event
$
\Omega_{\text{\tiny\ref{lem:sparse_singular_values}}}
$
from Lemma~\ref{lem:sparse_singular_values},
for every choice of a unit vector $\beta=(\beta_i)_{i\in[N]}$
such that
$$
\sum_{i\in[N]}\beta_i\,X_i=0,
$$
the vector $\beta$ is $(\delta_{\text{\tiny\ref{PN:incompcomb}}},
\rho_{\text{\tiny\ref{PN:incompcomb}}})$--incompressible.
\end{lemma}
\begin{proof}
By choosing $\delta_{\tref{PN:incompcomb}}$
sufficiently small, we can assume that
$$
\delta_{\tref{PN:incompcomb}} N \le c_{\tref{lem:sparse_singular_values}} n,
$$
and that
$$
\sqrt{{\delta_{\tref{PN:incompcomb}} N}/{n}}\,\log\frac{n}{\delta_{\tref{PN:incompcomb}} N}\leq \frac{1}{4}.
$$
Further, let $\rho_{\text{\tiny\ref{PN:incompcomb}}}:=\frac{1}{9\sqrt{K'}}$.
Let \( J \subset [N] \) be the subset of indices corresponding to 
\( \lfloor\delta_{\tref{PN:incompcomb}} N\rfloor \) largest
(in absolute value) coordinates of \( \beta \),
and denote by $\beta_J$ the coordinate projection of $\beta$
onto $\R^J$.
If $\|\beta_{J^c}\|_2:=\|\beta-\beta_J\|_2>\rho_{\text{\tiny\ref{PN:incompcomb}}}$
then there is nothing to prove.
Otherwise, 
by the definition of $\Omega_{\text{\tiny\ref{lem:sparse_singular_values}}}$, we have
\[
\|A_J \beta_J\|_2 \ge
\frac{3}{4}\sqrt{n}\,\|\beta_J\|_2
\geq 
\frac{3}{4}\sqrt{n}\,(1-\rho_{\text{\tiny\ref{PN:incompcomb}}})
\geq
\frac{1}{2} \sqrt{n}.
\]
On the other hand,
\[
\|A_{J^c} \beta_{J^c}\|_2 
\le \|A\| \, \|\beta_{J^c}\|_2 
\le 4\sqrt{N} \, \|\beta_{J^c}\|_2.
\]
Since \( A_J \beta_J + A_{J^c} \beta_{J^c} = 0 \), it follows that
\[
\|A_J \beta_J\|_2 = \|A_{J^c} \beta_{J^c}\|_2.
\]
Combining the above inequalities yields
\[
\frac{1}{2}\sqrt{n} \le 4\sqrt{N} \, \|\beta_{J^c}\|_2
\leq 4\sqrt{N}\,\rho_{\text{\tiny\ref{PN:incompcomb}}}
=4\sqrt{N}\,\frac{1}{9\sqrt{K'}}<\frac{1}{2}\sqrt{\frac{N}{K'}},
\]
leading to contradiction since $N\leq K'n$. The result follows.
\end{proof}

\subsection{Spans
of vectors of comparable
$\|\cdot\|_{P_{N,n}}$ and $\|\cdot\|_2$--norms}

The purpose of this subsection
is to study properties of sign combinations
$\sum_{i=1}^k \sigma_i y_i$ of vectors
$y_1,y_2,\dots,y_k$ of unit Euclidean norm
and of nearly constant $\|\cdot\|_{P_{N,n}}$--norm.
The main statement of the
subsection---Proposition~\ref{prop:spansofcomp}---asserts that with high probability any such configuration of $y_i$'s {\it cannot}
generate a near--$\ell_\infty$ subspace of $(\R^n,\|\cdot\|_{P_{N,n}})$.

\begin{definition}[$\beta^\sigma$]\label{def: betasigma 15262}
Let ${\bf F} = (y_i)_{i \in [k]}$ be a collection of arbitrary vectors
in $\R^n$.
For every $\sigma$ in $\{-1,1\}^k$, let 
$\beta^\sigma({\bf F}):=\beta\big(\sum_{i=1}^k \sigma_i y_i\big)$
(cf. Definition~\ref{def: beta 09812720894})
be a vector such that
$$
\sum_{i=1}^k \sigma_i y_i
=\sum_{j\in[N]}\beta_j^\sigma X_j,
$$
and
$$
\Big\|\sum_{i=1}^k \sigma_i y_i\Big\|_{P_{N,n}}
=\|\beta^\sigma({\bf F})\|_1.
$$
We will write $\beta^\sigma$ in place of $\beta^\sigma({\bf F})$
whenever the collection ${\bf F}$ is clear from context.
\end{definition}

\begin{definition}[$T^\sigma(\tau)$]\label{def: Tsigmatau}
Let ${\bf F} = (y_i)_{i \in [k]}$ be a collection of arbitrary vectors
in $\R^n$.
For each $\tau>0$ and $\sigma$ in $\{-1,1\}^k$, define
$$
T^\sigma(\tau, {\bf F}):=\big\{j\leq N:\;|\beta_j^\sigma({\bf F})|\geq \tau\big\}.
$$
We will write $T^\sigma(\tau)$
whenever $y_1,\dots,y_k$ are clear from the context.
The sets $T^\sigma(\tau, {\bf F})$ are used to keep track of magnitudes
of components of the coefficient vectors $\beta^\sigma({\bf F})$.
\end{definition}

\begin{definition}\label{def: Jt}
Let $2\leq k\leq n/2$, let $t\geq
\tilde C_{\text{\tiny\ref{lem:dotprodstat}}}
\log^{2} k$,
and let ${\bf F} = (y_i)_{i \in [k]}$ be a tuple
of unit vectors in $\R^n$.
Define $J(t, {\bf F}) \subseteq [N]$
as the collection of indices $j\in[N]$ satisfying {\bf at least
one} of the following two conditions: 
 \begin{enumerate}
    \item There are at least $\frac{2^k}{k^{100}}$ indices $\sigma \in \{\pm1\}^{k}$ such that $ \big|\big\langle X_j,\sum_{i=1}^k \sigma_i y_i\big\rangle\big|  \ge  t\,\sqrt{k}$; 
    \item There is an
    index $\sigma \in \{\pm1\}^{k}$
    such that $\big|\big\langle X_j,\sum_{i=1}^k \sigma_i y_i\big\rangle\big|
    \ge \tilde C_{\text{\tiny\ref{lem:dotprodstat}}}\,k^{11}$. 
\end{enumerate}
We will write $J(t)$ whenever
${\bf F}$ is clear from the context.
\end{definition}

\begin{remark}\label{rem: Jtsize}
The set $J(t, {\bf F})$ should be viewed as a set of indices corresponding to ``irregular''
(rare) realizations of $X_j$'s.
Indeed, in view of the definition of the event
$\Omega_{\text{\tiny\ref{lem:dotprodstat}}}$ from Lemma~\ref{lem:dotprodstat},
everywhere on $\Omega_{\text{\tiny\ref{lem:dotprodstat}}}$ we have
$$
|J(t, {\bf F})|\leq 
n\,\frac{\tilde C_{\text{\tiny\ref{lem:dotprodstat}}}\log t\;
\log^{4} k }{t^2}
+n\,k^{-10}
$$
for every $k$--tuple ${\bf F}$ of unit vectors in $\R^n$.
\end{remark}

\begin{lemma}[An approximation lemma]\label{lem:approxlemma}
Assume $n\geq c_{\text{\tiny\ref{lem:sparse_singular_values}}}^{-1}$ and
$K'\geq N/n\geq K>1$.
Let $\beta=(\beta_j)_{j\in [N]}$
be a non-random vector, and let
$$
Z:=\sum_{j\in[N]}\beta_j\,X_j.
$$
Then everywhere on event
$\Omega_{\text{\tiny\ref{lem:sparse_singular_values}}}$ from Lemma~\ref{lem:sparse_singular_values},
for every non-empty subset $T\subset[N]$ with
$|T|\leq c_{\text{\tiny\ref{lem:sparse_singular_values}}}\,n$,
$$
\big\|\big(\langle X_\ell, Z \rangle\big)_{\ell \in T }
-n\,\beta_T
\big\|_2
\leq
  O
  \bigg(n\,\|\beta_T\|_2\,\log\Big(\frac{n}{|T|}\Big) \sqrt{\frac{|T|}{n}}\bigg)
  +16 N \|\beta _{T^c}\|_2,
$$
where the implicit constant in $O(\dots)$ may only depend on $K',K$.
\end{lemma}
\begin{proof}
Consider 
\begin{align*}
    (A_T)^\top Z
= \big(\langle X_\ell, Z \rangle\big)_{\ell \in T }\,.
\end{align*}
We express the vector $Z$ as a sum: 
$$Z = A \beta  = A_{T } \beta _{T } + A_{T ^c} \beta_{T^c}.$$
The eigenvalues of the positive semi-definite matrix 
$(A_T)^\top A_{T }$ are squares of the singular values
of $A_T$. Hence, applying the definition of the event
$\Omega_{\text{\tiny\ref{lem:sparse_singular_values}}}$,
we obtain that within the event,
\begin{align*}
  \|(A_T)^\top A_{T } - n\,\Id\|
=&
  O
  \bigg(n\,\log\Big(\frac{n}{|T|}\Big) \sqrt{\frac{|T|}{n}}\bigg),
\end{align*}
and hence
$$
\Big\| (A_T)^\top A_{T } \beta _{T } 
  - n \beta _{T }\Big\|_2=
  O
  \bigg(n\,\|\beta_T\|_2\,\log\Big(\frac{n}{|T|}\Big) \sqrt{\frac{|T|}{n}}\bigg).
$$
On the other hand,
\begin{align*}
  \Big\|
  (A_T)^\top A_{T^c} \beta _{T^c}\Big\|
\le 
  \|A\|^2 \|\beta _{T^c}\|_2
\leq 16 N \|\beta _{T^c}\|_2.
\end{align*}
Thus, everywhere on $\Omega_{\text{\tiny\ref{lem:sparse_singular_values}}}$,
$$
\big\|(A_T)^\top Z
-n\,\beta_T
\big\|_2
\leq
  O
  \bigg(n\,\|\beta_T\|_2\,\log\Big(\frac{n}{|T|}\Big) \sqrt{\frac{|T|}{n}}\bigg)
  +16 N \|\beta _{T^c}\|_2.
$$
\end{proof}

\begin{corollary}\label{cor:approxcor}
Assume $n\geq c_{\text{\tiny\ref{lem:sparse_singular_values}}}^{-1}$
and
$K'\geq N/n\geq K>1$.
Everywhere on the event
$\Omega_{\text{\tiny\ref{lem:sparse_singular_values}}}$,
the following holds.
Let
$1\leq k\leq n/2$, and let
$y_1,y_2,\dots,y_k$ be a $k$--tuple of unit vectors in $\R^n$.
Then for all but $\frac{2^k}{k^{100}}$
choices of signs $\sigma\in\{-1,1\}^k$,
and for any $\tau\geq \frac{1}{c_{\text{\tiny\ref{lem:sparse_singular_values}}}\,n}
\|\beta^\sigma\|_1$,
$$
\Big\|(A_{T^\sigma(\tau)})^\top \sum_{i=1}^k \sigma_i y_i 
-n\,\beta_{T^\sigma(\tau)}^\sigma
\Big\|_2
=
  O
  \bigg(\sqrt{k\log k}\,\log\Big(\frac{\tau\,n}{\|\beta^\sigma\|_1}\Big) \sqrt{\frac{\|\beta^\sigma\|_1}{\tau}}
  +n \sqrt{\tau\,\|\beta^\sigma\|_1}\bigg),
$$
where the implicit constant in $O(\dots)$ may only depend on $K',K$.
\end{corollary}
\begin{proof}
Note that for every choice of $\sigma\in\{-1,1\}^k$,
and for every $\tau\geq \frac{1}{c_{\text{\tiny\ref{lem:sparse_singular_values}}}\,n}
\|\beta^\sigma\|_1$,
the size of $T^\sigma(\tau)$ can be estimated as
$$
|T^\sigma(\tau)|\leq c_{\text{\tiny\ref{lem:sparse_singular_values}}}\,n.
$$
Applying Lemma~\ref{lem:approxlemma}, we obtain
$$
\Big\|(A_{T^\sigma(\tau)})^\top
\sum_{i=1}^k \sigma_i y_i 
-n\,\beta^\sigma_{T^\sigma(\tau)}
\Big\|_2
\leq
  O
  \bigg(n\,\|\beta^\sigma_{T^\sigma(\tau)}\|_2\,
  \log\Big(\frac{n}{|T^\sigma(\tau)|}\Big)
  \sqrt{\frac{|T^\sigma(\tau)|}{n}}\bigg)
  +16 N \|\beta^\sigma_{(T^\sigma(\tau))^c}\|_2.
$$
By Holder's inequality,
\begin{equation}\label{eq:jhblfjnpeirjnrepijnqwlifj}
\|\beta^\sigma_{(T^\sigma(\tau))^c}\|_2
\leq \sqrt{\tau\,\|\beta^\sigma\|_1}.    
\end{equation}
Further, since the matrix $A_{T^\sigma(\tau)}$
is near isometry by the definition of
$\Omega_{\text{\tiny\ref{lem:sparse_singular_values}}}$,
we have
$$
\|\beta^\sigma_{T^\sigma(\tau)}\|_2
\leq \frac{1}{\sqrt{n}}\;
O\big(\|A_{T^\sigma(\tau)}\,\beta^\sigma_{T^\sigma(\tau)}\|_2\big).
$$
We can write
$$
A_{T^\sigma(\tau)}\,\beta^\sigma_{T^\sigma(\tau)}
=
A \beta^\sigma
-
A_{(T^\sigma(\tau))^c}\,\beta^\sigma_{(T^\sigma(\tau))^c},
$$
where, by \eqref{eq:jhblfjnpeirjnrepijnqwlifj},
$$
\big\|
A_{(T^\sigma(\tau))^c}\,\beta^\sigma_{(T^\sigma(\tau))^c}
\big\|_2
=O(\sqrt{\tau\,n\,\|\beta^\sigma\|_1}).
$$
Further, the relation
$$
A \beta^\sigma
=\sum_{i=1}^k \sigma_i y_i,
$$
the assumption on the Euclidean
norms of $y_i$'s and Khintchine's inequality imply that the Euclidean norm
$\|A \beta^\sigma\|_2=O(\sqrt{k\log k})$
for all but $\frac{2^k}{k^{100}}$
choices of signs $\sigma$.
We conclude that for all such choices of signs,
$$
\|\beta^\sigma_{T^\sigma(\tau)}\|_2
\leq
\frac{1}{\sqrt{n}}\;
O\big(\|A \beta^\sigma\|_2
+\|A_{(T^\sigma(\tau))^c}\,\beta^\sigma_{(T^\sigma(\tau))^c}\|_2\big)
\leq O(\sqrt{\tau\,\|\beta^\sigma\|_1})
+O\Big(\sqrt{\frac{k\log k}{n}}\Big).
$$
Combined with the estimate
$$
|T^\sigma(\tau)|
\leq \frac{\|\beta^\sigma\|_1}{\tau},
$$
it implies
\begin{align*}
n\,\|\beta^\sigma_{T^\sigma(\tau)}\|_2\,
  \log\Big(\frac{n}{|T^\sigma(\tau)|}\Big)
  \sqrt{\frac{|T^\sigma(\tau)|}{n}}
&=O\bigg(
n\,\|\beta^\sigma_{T^\sigma(\tau)}\|_2\,
  \log\Big(\frac{\tau\,n}{\|\beta^\sigma\|_1}\Big)
  \sqrt{\frac{\|\beta^\sigma\|_1}{\tau\,n}}
\bigg)\\
&=
O\bigg(
\sqrt{k\log k}\,
  \log\Big(\frac{\tau\,n}{\|\beta^\sigma\|_1}\Big)
  \sqrt{\frac{\|\beta^\sigma\|_1}{\tau}}
\bigg)\\
&+
O\bigg(
\sqrt{n}\,\|\beta^\sigma\|_1
  \log\Big(\frac{\tau\,n}{\|\beta^\sigma\|_1}\Big)
\bigg).
\end{align*}
Using that $\|\beta^\sigma\|_1\leq
c_{\text{\tiny\ref{lem:sparse_singular_values}}}\,n\,\tau$,
we can estimate the second term by
$$
O\big(
n\,\sqrt{\tau\,\|\beta^\sigma\|_1}\big).
$$
Finally,
$$
16 N \|\beta^\sigma_{(T^\sigma(\tau))^c}\|_2
=O\big(
n\,\sqrt{\tau\,\|\beta^\sigma\|_1}\big),
$$
and hence
$$
\Big\|(A_{T^\sigma(\tau)})^\top
\sum_{i=1}^k \sigma_i y_i 
-n\,\beta^\sigma_{T^\sigma(\tau)}
\Big\|_2
\leq
O\bigg(
\sqrt{k\log k}\,
  \log\Big(\frac{\tau\,n}{\|\beta^\sigma\|_1}\Big)
  \sqrt{\frac{\|\beta^\sigma\|_1}{\tau}}
\bigg)
+O\big(
n\,\sqrt{\tau\,\|\beta^\sigma\|_1}\big).
$$
\end{proof}

\begin{lemma}
\label{lem: Abetasigma}
Assume $n\geq c_{\text{\tiny\ref{lem:sparse_singular_values}}}^{-1}$
and
$K'\geq N/n\geq K>1$.
Within the event $\Omega_{\text{\tiny\ref{lem:sparse_singular_values}}}$,
for every $2\leq k\leq n/2$ and
for every choice of $k$--tuple of vectors $y_1,\dots,y_k$
of unit Euclidean length,
there is a subset  
  $ R \subset \{\pm1\}^{k}$
  with $|R| \ge 2^k -
  \frac{C_{\text{\tiny\ref{lem: Abetasigma}}}}{k^{78}}2^k$ 
  such that for every $\sigma \in R$ and
  for every $t\geq
\tilde C_{\text{\tiny\ref{lem:dotprodstat}}}
\log^{2} k$,
we have
  \begin{align*}
    \Big\|A_{ J^c(t)} \beta_{J^c(t)}^\sigma\Big\|_2 
    =  
    O\bigg(1
 +
 \frac{1}{t}\,\sqrt{k}
  + t^2\,\|\beta^\sigma\|_1\bigg),
  \end{align*}
  where the implicit constant depends only on $K,K'$.
\end{lemma}
\begin{proof}
Let $W$ be the collection of the signs $\sigma\in\{-1,1\}^k$
satisfying the assertion of Corollary~\ref{cor:approxcor},
so that $|W|\geq 2^k-\frac{2^k}{k^{100}}$.

\medskip

\noindent {\it Matrix $M$ and set $R$.}
Now, let us define a $2^k \times N$ matrix $M$ whose rows are indexed by $\sigma \in \{\pm1\}^{k}$ and columns are indexed by $j \in [N]$. The entry 
$$
    M_{\sigma,j} := \big\langle X_j,\sum_{i=1}^k \sigma_i y_i\big\rangle\;
    \cdot {\bf 1}_{\{ |\langle X_j,\sum_{i=1}^k \sigma_i y_i\rangle| \ge  t\,\sqrt{k}\}}\,,
$$
which is a ``tail'' of the inner product between the vector $\sum_{i=1}^k \sigma_i y_i$ and the random vector $X_j$.

We remark that the way we truncate the inner products ensures that for each $j \notin J(t)$ (with $J(t)$ defined according to
Definition~\ref{def: Jt}), the column $M_j$ has strictly less than $2^k/k^{100}$ non-zero entries, and each non-zero entry of $M_j$ is bounded above
by $\tilde C_{\text{\tiny\ref{lem:dotprodstat}}}\,k^{11}$. This allows us to bound the Hilbert--Schmidt norm of $M_{J^c(t)}$, where
$J^c(t) = [N] \setminus J(t)$:
$$\|M_{J^c(t)}\|_{\rm HS}^2 \le 
\big(\tilde C_{\text{\tiny\ref{lem:dotprodstat}}}\,k^{11}\big)^2\,
\frac{2^k}{k^{100}}N \,.$$
Now, we define 
$$
    R := \Big\{ \sigma \in W :\; \|M_{\sigma,J^c}\|_2 \le \sqrt{n}\Big\}.
$$
From the above Hilbert--Schmidt norm bound, we have $|R| \ge 2^k - \frac{C'}{k^{78}}2^k$, for a constant $C'>0$.

\bigskip

\noindent {\it Decomposition of $J^c(t)$.}
For every $\sigma\in W$, and for $\tau_\sigma:=\frac{t^4}{n}\,\|\beta^\sigma\|_1$,
let 
$$
S_\sigma = 
\Big\{ j \in T^\sigma(\tau_\sigma)\,:\, \big|\big\langle X_j,\sum_{i=1}^k \sigma_i y_i\big\rangle\big|
<  t\,\sqrt{k} \Big\}\,. 
$$
For each $\sigma \in R$, we decompose the set $J^c(t)$ into three parts: 
\begin{align*}
    J^c(t)  
    =  \big(\underbrace{(T^\sigma(\tau_\sigma) \setminus S_\sigma) \cap J^c(t)}_{:=L_\sigma}\big) \sqcup \big(S_\sigma \cap J^c(t)\big) \sqcup \big((T^\sigma(\tau_\sigma))^c \cap J^c(t)\big) \,.
\end{align*}

\noindent {\it Treatment of
$L_\sigma :=(T^\sigma(\tau_\sigma)\setminus S_\sigma) \cap J^c(t)$.}
The set $T^\sigma(\tau_\sigma) \setminus S_\sigma$ is the collection of indices $j$ such that $M_{\sigma,j} = \langle X_j,\sum_{i=1}^k \sigma_i y_i\rangle$,
and at the same time $|\beta_j^\sigma|
\ge \tau_\sigma$.
From the definition of event $\Omega_{\text{\tiny\ref{lem:sparse_singular_values}}}$
and Corollary~\ref{cor:approxcor},
\begin{align*}
    \big\|A_{L_\sigma} \beta_{L_\sigma}^\sigma\big\|_2 
&\le 
    \frac{4\sqrt{N}}{n}\,\| n\beta_{L_\sigma}^\sigma\|_2 
\leq
\frac{4\sqrt{N}}{n}\,
\Big\|(A_{L_\sigma})^\top \sum_{i=1}^k \sigma_i y_i\Big\|_2
+
\frac{4\sqrt{N}}{n}\,\Big\|(A_{L_\sigma})^\top \sum_{i=1}^k \sigma_i y_i
-n\beta_{L_\sigma}^\sigma\Big\|_2\\
&\leq 
\frac{4\sqrt{N}}{n}\,
\Big\|(A_{L_\sigma})^\top \sum_{i=1}^k \sigma_i y_i\Big\|_2
+
O
  \bigg(\sqrt{\frac{k\log k}{n}}\,\log\Big(\frac{\tau_\sigma\,n}{\|\beta^\sigma\|_1}\Big) \sqrt{\frac{\|\beta^\sigma\|_1}{\tau_\sigma}}
  +\sqrt{\tau_\sigma\,n\,\|\beta^\sigma\|_1}\bigg),
\end{align*}
where we emphasize that
$(A_{L_\sigma})^\top \sum_{i=1}^k \sigma_i y_i 
= M_{\sigma,L_\sigma}$. 

\bigskip

\noindent {\it Treatment of
$S_\sigma \cap J^c(t)$ and $(T^\sigma(\tau))^c \cap J^c(t)$.}
For any subset $S \subset S_\sigma$, 
by Corollary~\ref{cor:approxcor} and the condition
$S\subset T^\sigma(\tau)$,
\begin{align*}
\Big\| (A_{S})^\top \sum_{i=1}^k \sigma_i y_i 
- n\, \beta^\sigma_{S} \Big\|_2
&\le 
\Big\| (A_{T^\sigma(\tau_\sigma)})^\top \sum_{i=1}^k \sigma_i y_i 
- n\, \beta^\sigma_{T^\sigma(\tau_\sigma)} \Big\|_2\\
&\leq
  O
  \bigg(\sqrt{k\log k}\,\log\Big(\frac{\tau_\sigma\,n}{\|\beta^\sigma\|_1}\Big) \sqrt{\frac{\|\beta^\sigma\|_1}{\tau_\sigma}}
  +n \sqrt{\tau_\sigma\,\|\beta^\sigma\|_1}\bigg).
\end{align*}
Using that
$\big|\big\langle X_j,\sum_{i=1}^k \sigma_i y_i\big\rangle\big| <
t\,\sqrt{k}$ for $j \in S_\sigma$, we obtain
\begin{equation}\label{eq: norm_A_S_sigma_lambda}
     \|n\, \beta^\sigma_{S} \|_2 
\le 
     t\,\sqrt{k\,|S|} + 
     O
  \bigg(\sqrt{k\log k}\,\log\Big(\frac{\tau_\sigma\,n}{\|\beta^\sigma\|_1}\Big) \sqrt{\frac{\|\beta^\sigma\|_1}{\tau_\sigma}}
  +n \sqrt{\tau_\sigma\,\|\beta^\sigma\|_1}\bigg).
\end{equation}
We apply \eqref{eq: norm_A_S_sigma_lambda} with $S = S_\sigma \cap J^c(t)$ to obtain 
\begin{align*}
    \Big\|A_{S_\sigma \cap J^c(t)}
    \beta_{S_\sigma \cap J^c(t)}^\sigma\Big\|_2
&=
    \sqrt{n}\;
    O\big(\big\|\beta_{S_\sigma \cap J^c(t)}^\sigma\big\|_2\big)\\
&\leq 
O\bigg(t\,\sqrt{\frac{k\,|S_\sigma \cap J^c(t)|}{n}} +
\sqrt{\frac{k\log k}{n}}\,\log\Big(\frac{\tau_\sigma\,n}{\|\beta^\sigma\|_1}\Big) \sqrt{\frac{\|\beta^\sigma\|_1}{\tau_\sigma}}
  + \sqrt{\tau_\sigma\,n\,\|\beta^\sigma\|_1}\bigg).
\end{align*}
Further, 
by Holder's inequality and
definition of $\Omega_{\text{\tiny\ref{lem:sparse_singular_values}}}$,
$$
\Big\|A_{(T^\sigma(\tau_\sigma))^c \cap J^c(t)} \beta_{(T^\sigma(\tau_\sigma))^c \cap J^c(t)}^\sigma\Big\|_2
=O(\sqrt{n}\,\|\beta_{(T^\sigma(\tau_\sigma))^c}^\sigma\|_2)
=O(\sqrt{\tau_\sigma\,n\,\|\beta^\sigma\|_1}).
$$
With the estimates of the three parts, we can conclude via triangle inequality that for every $\sigma\in R$,
\begin{align*}
    &\Big\|A_{J^c(t)}\, \beta_{J^c(t)}^\sigma\Big\|_2
\leq
\big\|A_{L_\sigma} \beta_{L_\sigma}^\sigma\big\|_2
+
\big\|A_{S_\sigma \cap J^c(t)} \beta_{S_\sigma \cap J^c(t)}^\sigma\big\|_2
+
 \big\|A_{(T^\sigma(\tau_\sigma))^c \cap J^c(t)}
 \beta_{(T^\sigma(\tau_\sigma))^c \cap J^c(t)}^\sigma\big\|_2\\
 &\leq O\bigg(\frac{1}{\sqrt{n}}
 \big\|M_{\sigma,L_\sigma}\big\|_2
 +
 t\,\sqrt{\frac{k\,|S_\sigma \cap J^c(t)|}{n}} +
\sqrt{\frac{k\log k}{n}}\,\log\Big(\frac{\tau_\sigma\,n}{\|\beta^\sigma\|_1}\Big) \sqrt{\frac{\|\beta^\sigma\|_1}{\tau_\sigma}}
  + \sqrt{\tau_\sigma\,n\,\|\beta^\sigma\|_1}\bigg)\\
  &\leq
  O\bigg(1
 +
 t\,\sqrt{\frac{k\,\|\beta^\sigma\|_1}{\tau_\sigma\,n}} +
\sqrt{\frac{k\log k}{n}}\,\log\Big(\frac{\tau_\sigma\,n}{\|\beta^\sigma\|_1}\Big) \sqrt{\frac{\|\beta^\sigma\|_1}{\tau_\sigma}}
  + \sqrt{\tau_\sigma\,n\,\|\beta^\sigma\|_1}\bigg).
\end{align*}
Applying the definition of $\tau_\sigma$,
$$
\Big\|A_{J^c(t)}\, \beta_{J^c(t)}^\sigma\Big\|_2
\leq 
O\bigg(1
 +
 \frac{1}{t}\,\sqrt{k} +
\frac{1}{t^2}\,\sqrt{k\log k}\,\log t 
  + t^2\,\|\beta^\sigma\|_1\bigg).
$$
The result follows.
\end{proof}

The next proposition, which is the main result of the subsection,
shows that for a typical realization of $P_{N,n}$,
there is no $k$--tuple of unit vectors $y_1,\dots,y_k$
with $\|\cdot\|_{P_{N,n}}$--norms of constant order
such that their random sign combinations
$\big\|
\sum\nolimits_{i=1}^k \sigma_i y_i
\big\|_{P_{N,n}}$ are small on average.

\begin{proposition}\label{prop:spansofcomp}
Let $K'\geq K>1$.
There is a constant $C_{\text{\tiny\ref{prop:spansofcomp}}}\geq 1$
depending only on $K,K'$ with the following property.
Assume $n\geq c_{\text{\tiny\ref{lem:sparse_singular_values}}}^{-1}$
and $K'\geq N/n\geq K$.
Condition on any realization of $X_1,\dots,X_N$
from the intersection
$\Omega_{\text{\tiny\ref{lem:sparse_singular_values}}}
\cap \Omega_{\text{\tiny\ref{lem:dotprodstat}}}$.
Let $1\leq k\leq n/2$, and let $y_1,\dots,y_k$
be vectors in $\R^n$ of unit Euclidean length
such that $\|\beta(y_i)\|_{1}=\|y_i\|_{P_{N,n}}\geq
C_{\text{\tiny\ref{prop:spansofcomp}}}\,k^{-1/9}$
for every $i\leq k$, where $\beta(y_i)$ is given in Definition~\ref{def: beta 09812720894}. Then,
with $\beta^\sigma$ from Definition~\ref{def: betasigma 15262},
$$
\Exp_\sigma\,\|\beta^\sigma\|_1
=\Exp_\sigma\,\Big\|
\sum\nolimits_{i=1}^k \sigma_i y_i
\Big\|_{P_{N,n}}\geq k^{1/8}.
$$
\end{proposition}
\begin{proof}
We will assume that $C_{\text{\tiny\ref{prop:spansofcomp}}}$
is large.
Observe that, in view of the in-radius estimate
$c_{\text{\tiny\ref{PN:inradius}}}\leq r(P_{N,n})$
from Lemma~\ref{PN:inradius},
every unit vector $y_i$ must satisfy $\|y_i\|_{P_{N,n}}\leq c_{\text{\tiny\ref{PN:inradius}}}^{-1}$.
Therefore, using the condition
$\|y_i\|_{P_{N,n}}\geq
C_{\text{\tiny\ref{prop:spansofcomp}}}\,k^{-1/9}$,
we can (and will) assume that
$$
k\geq \big(c_{\text{\tiny\ref{PN:inradius}}}\,
C_{\text{\tiny\ref{prop:spansofcomp}}}\big)^9.
$$
In particular, $k$ can be made greater than an arbitrary universal constant.

We start with the 
following standard observation. For every collection
of vectors $z_1,z_2,\dots,z_k$,
$$
\Exp_\sigma\,\Big\|
\sum\nolimits_{i=1}^k \sigma_i z_i\Big\|_2^2
=\sum_{i=1}^k \big\|z_i\big\|_2^2,
$$
and
$$
\Exp_\sigma\,\Big\|
\sum\nolimits_{i=1}^k \sigma_i z_i\Big\|_2^4
=
O\bigg(\sum_{i=1}^k \big\|z_i\big\|_2^2\bigg)^2,
$$
where the implicit constant is universal.
Accordingly, applying the Paley--Zygmund inequality, 
there is a universal constant $c\in(0,1)$
such that for every $s>0$ with
$$
\Prob\bigg\{\Big\|\sum\nolimits_{i=1}^k \sigma_i z_i\Big\|_2\leq c\,s\bigg\}
\geq 1-c,
$$
we have
$$
\sum_{i=1}^k \big\|z_i\big\|_2^2\leq s^2,
$$
and hence there is an index $i\leq k$
with $\big\|z_i\big\|_2\leq s/\sqrt{k}$.

\medskip

To prove the proposition, we will
argue by contradiction.
We suppose that $\Exp_\sigma\,\|\beta^\sigma\|_1\leq k^{1/8}$.
Applying Lemma~\ref{lem: Abetasigma},
we obtain a subset  
  $ R \subset \{\pm1\}^{k}$
  with $|R| \ge 2^k -
  \frac{C_{\text{\tiny\ref{lem: Abetasigma}}}}{k^{78}}2^k$ 
  such that for every $\sigma \in R$ and
  for every $t\geq
\tilde C_{\text{\tiny\ref{lem:dotprodstat}}}
\log^{2} k$,
  \begin{align*}
    \Big\|A_{ J^c(t)} \beta_{J^c(t)}^\sigma\Big\|_2 
    \leq C
    \Big(1
 +
 \frac{1}{t}\,\sqrt{k}
  + t^2\,\|\beta^\sigma\|_1\Big),
  \end{align*}
  where $C$ depends only on $K,K'$,
  and where the set $J(t)$ is given by Definition~\ref{def: Jt}.
Choose $t:=k^{1/8}$ (note that we can suppose that
$k^{1/8}\geq \tilde C_{\text{\tiny\ref{lem:dotprodstat}}}
\log^{2} k$). Further, note that in view 
of the bound $\Exp_\sigma\,\|\beta^\sigma\|_1\leq k^{1/8}$,
we have $\|\beta^\sigma\|_1\leq 2\,c^{-1}\,k^{1/8}$
for at least $2^{k}-\frac{c}{2}\,2^k$ signs $\sigma$,
where the constant $c$ is taken from the claim at the beginning of the proof.
Then, from the above,
$$
\Big\|A_{ J^c(t)} \beta_{J^c(t)}^\sigma\Big\|_2 
    \leq 4\,C\,c^{-1}\,k^{3/8}\quad\mbox{for at least
    $2^{k}-\frac{c}{2}\,2^k-\frac{C_{\text{\tiny\ref{lem: Abetasigma}}}}{k^{78}}2^k$ signs $\sigma$.}
$$
Denote by $\Proj$ the orthogonal projection onto 
the linear span of vectors $X_j$, $j\in J(t)$,
and let $\Proj_\perp$ be the orthogonal projection onto its complement.
It is easy to see that for every $\sigma$,
$$
\big\|A_{ J^c(t)} \beta_{J^c(t)}^\sigma\big\|_2
\geq \Big\|\Proj_\perp\Big(
\sum\nolimits_{i=1}^k \sigma_i\,y_i
\Big)\Big\|_2.
$$
Thus,
$$
\Prob_\sigma\bigg\{
\Big\|\Proj_\perp\Big(
\sum\nolimits_{i=1}^k \sigma_i\,y_i
\Big)\Big\|_2\leq 
4\,C\,c^{-1}\,k^{3/8}
\bigg\}\geq 
1-\frac{c}{2}-\frac{C_{\text{\tiny\ref{lem: Abetasigma}}}}{k^{78}}
>1-c,
$$
where we used the assumption that 
$
k\geq \big(c_{\text{\tiny\ref{PN:inradius}}}\,
C_{\text{\tiny\ref{prop:spansofcomp}}}\big)^9
$
and $C_{\text{\tiny\ref{prop:spansofcomp}}}$ is large.
Applying the claim from the beginning of the argument,
we get that there is an index $i_0\leq k$
with $\big\|\Proj_\perp(y_{i_0})\big\|_2\leq 4\,C\,c^{-2}\,k^{-1/8}$,
and, following the in-radius estimates, 
$\big\|\Proj_\perp(y_{i_0})\big\|_{P_{N,n}}
\leq 4\,C\,c^{-2}\,c_{\text{\tiny\ref{PN:inradius}}}^{-1}\,k^{-1/8}$.
Further, 
according to Remark~\ref{rem: Jtsize} and our choice of $t$,
$$
|J(t)|\leq 
n\,\frac{\tilde C_{\text{\tiny\ref{lem:dotprodstat}}}\log t\;
\log^{4} k }{t^2}
+n\,k^{-10}\leq 
\tilde C_{\text{\tiny\ref{lem:dotprodstat}}}\,
n\,k^{-1/4}\;
\log^{5} k
+n\,k^{-10}\leq c_{\tref{lem:sparse_singular_values}}\,n,
$$
where in the last inequality we used once again that $k$ is large.
Applying Lemma~\ref{PN:inradius}, we get that the (relative) in-radius
of $\conv\,\big\{\pm X_j,\quad j\in J(t)\big\}$ satisfies
$$
r\big(\conv\,\big\{\pm X_j,\quad j\in J(t)\big\}\big)
\geq
\frac{1}{2}\sqrt{\frac{n}{|J(t)|}},
$$
and hence
$$
\big\|\Proj(y_{i_0})\big\|_{P_{N,n}}
\leq 2\,\sqrt{|J(t)|/n}<\frac{1}{2}k^{-1/9}.
$$
Combining the upper bounds for $\big\|\Proj_\perp(y_{i_0})\big\|_{P_{N,n}}$
and $\big\|\Proj(y_{i_0})\big\|_{P_{N,n}}$, we get
$$
\|y_{i_0}\|_{P_{N,n}}< k^{-1/9},
$$
leading to contradiction.
\end{proof}

\subsection{Embedding $\ell_\infty^k$ into $(\R^n,\|\cdot\|_{P_{N,n}})$}

The main result of this section is the following

\begin{theorem}\label{th: Gluskinlinfty}
There is a universal constant $\alpha>0$,
and constants $c_{\text{\tiny\ref{th: Gluskinlinfty}}}>0$
and $C_{\text{\tiny\ref{th: Gluskinlinfty}}}
\geq c_{\text{\tiny\ref{lem:sparse_singular_values}}}^{-1}$
depending only on $K,K'>1$ with the following property.
Let integers $N,n$
satisfy $K\leq \frac{N}{n}\leq K'$ and $n\geq
C_{\text{\tiny\ref{th: Gluskinlinfty}}}$.
Condition on any realization of $X_1,\dots,X_N$
from the intersection
$\Omega_{\text{\tiny\ref{lem:sparse_singular_values}}}
\cap \Omega_{\text{\tiny\ref{lem:dotprodstat}}}
$,
where the events are defined within respective lemmas.
Then
the polytope $P_{N,n}$ has the following property:
for every $1\leq k\leq n/2$ and every $k$--dimensional
subspace $E$ of $(\R^n,\|\cdot\|_{P_{N,n}})$,
$$
d_{BM}(\ell_\infty^k,E)
\geq c_{\text{\tiny\ref{th: Gluskinlinfty}}}\,k^\alpha.
$$
\end{theorem}
In view of the probability estimates for the events
$\Omega_{\text{\tiny\ref{lem:sparse_singular_values}}}$ and
$\Omega_{\text{\tiny\ref{lem:dotprodstat}}}$
from Lemmas~\ref{lem:sparse_singular_values}, and~\ref{lem:dotprodstat}
the above theorem immediately implies Theorem~B from the introduction, which, in turn,
through the aforementioned result of Maurey and Pisier \cite{MP76}, yields Theorem~A.

Proposition~\ref{prop:spansofcomp} proved in the previous
subsection, can be viewed as a strongly specialized
version of the above theorem.
Our strategy in proving the result
is by reduction of the general setting to the one treated in
Proposition~\ref{prop:spansofcomp}.
Assume that vectors $y_1,\dots,y_k$
span a subspace of $(\R^n,\|\cdot\|_{P_{N,n}})$
with a small Banach--Mazur distance to $\ell_\infty^k$.
Analysis of vector tuples $(y_1,\dots,y_k)$ below
splits into two major cases: either
there are many pairs of indices $i$ and $r$ corresponding
to large values of
the ratios $\frac{m_\delta(y_i,r)}{2^{-2r}\,n}$ (see Definition~\ref{def: myr 334q}),
or these expressions
are well-controlled from above uniformly in $i$ and $r$.
In the former case, the coefficient vectors $\beta(y_i)$
are ``spiky''. We will show that this feature
ultimately contradicts Lemma~\ref{PN:incompcomb} dealing with
incompressibility of coefficient vectors producing 
zero sums of $X_i$'s, i.e cannot happen with a non-negligible probability
(see Lemma~\ref{lem:spikymdelta}).
On the other hand, in the latter case---which corresponds
to ``well-behaved'' $\beta(y_i)$'s---we are able
to ``truncate'' the coefficients leading to
a collection of new vectors $\tilde y_i$
with roughly comparable $\|\cdot\|_2$ and $\|\cdot\|_{P_{N,n}}$--norms (Lemma~\ref{lem:stairs}).
We emphasize that the proof of Theorem~\ref{th: Gluskinlinfty}
is essentially deterministic after conditioning.

\medskip

The first step of the proof is a strong corollary
of Lemma~\ref{PN:incompcomb} applied to coefficient vectors
$\beta(\cdot)$:

\begin{lemma}[Pseudo-incompressibility of $(\sqrt{\sum_{i\in L}\beta_h(y_i)^2})_{h\in[N]}$]\label{lem:generalPNbasic}
For every $K'\geq K>1$
there is a constant $c_{\text{\tiny\ref{lem:generalPNbasic}}}\in(0,1]$
depending only on $K,K'$
with the following properties.
Assume that $n\geq c_{\text{\tiny\ref{lem:sparse_singular_values}}}^{-1}$
and $K'\geq \frac{N}{n}\geq K$.
Condition on any realization of $X_1,\dots,X_N$
in $\Omega_{\text{\tiny\ref{lem:sparse_singular_values}}}$.
Assume that $J\subset [N]$
is a non-empty subset of indices with
$|J|\leq c_{\text{\tiny\ref{lem:generalPNbasic}}}\,n$.
Further, let $y_i$, $i\in L$, be a finite collection of non-zero vectors
in $\R^n$, let $\sigma$ be a uniform random
vector of signs indexed over $L$, and 
let $\beta^\sigma$ be the coefficient vector
corresponding to $\sum_{i\in L}\sigma_i\,y_i$.
Assume further that
\begin{equation}\label{eq: 315aslkfnl}
c_{\text{\tiny\ref{lem:generalPNbasic}}}\,
\sum_{h\in J}\sqrt{\sum_{i\in L}\beta_h(y_i)^2}
\geq 
\Exp_{\sigma}\,\|\beta^\sigma\|_1.
\end{equation}
Then there are at least $c_{\text{\tiny\ref{lem:generalPNbasic}}}\,n$
indices $j\in J^c$ satisfying
$$
\sqrt{\sum_{i\in L}\beta_j(y_i)^2}
\geq 
\frac{c_{\text{\tiny\ref{lem:generalPNbasic}}}}{\sqrt{|J|\,n}}\,
\sum_{h\in J}
\,\sqrt{\sum_{i\in L}\beta_h(y_i)^2}.
$$
\end{lemma}
As was discussed in the proof overview within the introduction,
the above statement is viewed as a property of the vector $(\sqrt{\sum_{i\in L}\beta_h(y_i)^2})_{h\in[N]}$
akin to incompressibility, although somewhat weaker.
The proof below relies, in its core, on the (proper) incompressibility of
the differences $\sum_{i\in L}\sigma_i\,\beta(y_i)-\beta^\sigma$, guaranteed by Lemma~\ref{PN:incompcomb}.
\begin{proof}[Proof of Lemma~\ref{lem:generalPNbasic}]
    We will assume that $c_{\text{\tiny\ref{lem:generalPNbasic}}}
    =c_{\text{\tiny\ref{lem:generalPNbasic}}}(K,K')$
    is a sufficiently small positive constant.
    For the rest of the proof, we condition 
    on any realization of $X_1,\dots,X_N$ from
    $\Omega_{\text{\tiny\ref{lem:sparse_singular_values}}}$.
    Denote by $J'$ the (non-random) collection of all
    $j\in J^c=[N]\setminus J$ such that
    $$
    \sqrt{\sum_{i\in L}\beta_j(y_i)^2}
    \geq 
    \frac{c_{\text{\tiny\ref{lem:generalPNbasic}}}}{\sqrt{|J|\,n}}\,
    \sum_{h\in J}\sqrt{\sum_{i\in L}\beta_h(y_i)^2}
    $$
    The rest of our argument is by contradiction:
    we assume that $|J'|<c_{\text{\tiny\ref{lem:generalPNbasic}}}\,n$.
    For every $i\in L$, we write
    $$
    y_i=y_i'+y_i'',
    $$
    where $y_i':=\sum_{j=1}^N \beta_j(y_i)\,
    {\bf 1}_{\{j\notin J\}}\,X_j$
    and 
    $y_i'':=\sum_{j=1}^N \beta_j(y_i)\,
    {\bf 1}_{\{j\in J\}}\,X_j$.
    We have
    \begin{align*}
    \Big\|
    \Big(\sum_{i\in L}\sigma_i\,\beta_j(y_i)-\beta^\sigma_j\Big)_{j\in J}\Big\|_1
    \geq 
    \sum_{j\in J}\Big|\sum_{i\in L}\sigma_i\,\beta_j(y_i)\Big|
    -\sum_{j\in J}|\beta^\sigma_j|
    \geq 
    \sum_{j\in J}\Big|\sum_{i\in L}\sigma_i\,\beta_j(y_i)\Big|
    - \|\beta^\sigma\|_1.
    \end{align*}
    Standard anti-concentration estimates
    (for example, the Paley--Zygmund inequality)
    imply that for a universal constant
    $\tilde c\in(0,1)$,
    $$
    \Prob_\sigma\bigg\{
    \sum_{j\in J}\Big|\sum_{i\in L}\sigma_i\,\beta_j(y_i)\Big|
    \geq 
    \tilde c\,\sum_{j\in J}\sqrt{\sum_{i\in L}\beta_j(y_i)^2}
    \bigg\}
    \geq \tilde c.
    $$
    On the other hand, by the assumptions on $\Exp_{\sigma}\,\|\beta^\sigma\|_1$
    and by Markov's inequality,
    $$
    \Prob_\sigma\bigg\{
    \|\beta^\sigma\|_1\leq 2(\tilde c)^{-1}\;
c_{\text{\tiny\ref{lem:generalPNbasic}}}\,
\sum_{j\in J}\sqrt{\sum_{i\in L}\beta_j(y_i)^2}
    \bigg\}\geq 1-\frac{\tilde c}{2}.
    $$
We can (and will) assume 
that $2(\tilde c)^{-1}\,c_{\text{\tiny\ref{lem:generalPNbasic}}}
\leq \frac{\tilde c}{2}$.
Then from the above we get
$$
\Prob_\sigma\bigg\{
\Big\|
\Big(\sum_{i\in L}\sigma_i\,\beta_j(y_i)-\beta^\sigma_j\Big)_{j\in J}\Big\|_1
\geq 
\frac{\tilde c}{2}\,\sum_{j\in J}\sqrt{\sum_{i\in L}\beta_j(y_i)^2}
\bigg\}
\geq \frac{\tilde c}{2},
$$
implying
    \begin{equation*}
\Prob_\sigma\bigg\{
\Big\|
\Big(\sum_{i\in L}\sigma_i\,\beta_j(y_i)-\beta^\sigma_j\Big)_{j\in J}\Big\|_2
\geq 
\frac{\tilde c}{2\sqrt{|J|}}\,\sum_{j\in J}\sqrt{\sum_{i\in L}\beta_j(y_i)^2}
\bigg\}
\geq \frac{\tilde c}{2}.
    \end{equation*}
    Further, we note that for any choice of parameter $\tau>0$,
    $$
    \tau\,\Exp_\sigma\,\big|\big\{j\leq N:\;|\beta^\sigma_j|\geq \tau\big\}\big|
    \leq \Exp_{\sigma}\,\|\beta^\sigma\|_1
    \leq
    c_{\text{\tiny\ref{lem:generalPNbasic}}}\,
    \sum_{j\in J}\sqrt{\sum_{i\in L}\beta_j(y_i)^2}
    $$
    Hence, denoting by $\tilde J$ the (random)
    collection of all indices
    $j\in [N]\setminus J$ such that
    $$
    |\beta^\sigma_j|\geq
    \frac{1}{n}\,
    \sum_{h\in J}\sqrt{\sum_{i\in L}\beta_h(y_i)^2},
    $$
    we obtain
    $$
    \Exp_\sigma\,|\tilde J|
    \leq c_{\text{\tiny\ref{lem:generalPNbasic}}}\,n,
    $$
    while at the same time for all realizations of $\sigma$,
    $$
    \Big\|\big(\beta^\sigma_j\big)_{
    j\in [N]\setminus(\tilde J\cup J)}\Big\|_2
    \leq
    \frac{\sqrt{N}}{n}\,
    \sum_{h\in J}\sqrt{\sum_{i\in L}\beta_h(y_i)^2}.
    $$
    Define $J_0:=J\cup \tilde J\cup J'$.
    Then, by the above,
    \begin{equation}\label{eq:J0upperbound}
    \Exp_\sigma\,|J_0|\leq 3c_{\text{\tiny\ref{lem:generalPNbasic}}}\,n,
    \end{equation}
    whereas
    \begin{equation}\label{eq:probtildec}
    \Prob_\sigma\bigg\{
\Big\|
\Big(\sum_{i\in L}\sigma_i\,\beta_j(y_i)-\beta^\sigma_j\Big)_{j\in J_0}\Big\|_2
\geq 
\frac{\tilde c}{2\sqrt{|J|}}\,
\sum_{h\in J}\sqrt{\sum_{i\in L}\beta_h(y_i)^2}
\bigg\}
\geq \frac{\tilde c}{2},
    \end{equation}
    and
    \begin{align}
    \Exp_{\sigma}\,\Big\|
    \Big(\sum_{i\in L}\sigma_i\,\beta_j(y_i)-\beta^\sigma_j\Big)_{j\in J_0^c}\Big\|_2
    &\leq 
    \Exp_{\sigma}\,\Big\|
    \Big(\sum_{i\in L}\sigma_i\,\beta_j(y_i)\Big)_{j\in J_0^c}\Big\|_2
    +
    \Exp_{\sigma}\,\Big\|
    \big(\beta^\sigma_j\big)_{j\in J_0^c}\Big\|_2\nonumber\\
    &\leq
    \sqrt{\Exp_{\sigma}\,\Big\|
    \Big(\sum_{i\in L}\sigma_i\,\beta_j(y_i)\Big)_{j\in J_0^c}\Big\|_2^2}
    +
    \frac{\sqrt{N}}{n}\,
    \sum_{h\in J}\sqrt{\sum_{i\in L}\beta_h(y_i)^2}\nonumber\\
    &\leq
    \sqrt{\sum_{j\in [N]\setminus(J\cup J')}\sum_{i\in L}\beta_j(y_i)^2}
    +
    \frac{\sqrt{N}}{n}\,
    \sum_{h\in J}\sqrt{\sum_{i\in L}\beta_h(y_i)^2}\nonumber\\
    &\leq
    \frac{c_{\text{\tiny\ref{lem:generalPNbasic}}}\,\sqrt{N}}{\sqrt{|J|\,n}}\,
    \sum_{h\in J}\sqrt{\sum_{i\in L}\beta_h(y_i)^2}
    +
    \frac{\sqrt{N}}{n}\,
    \sum_{h\in J}\sqrt{\sum_{i\in L}\beta_h(y_i)^2}\nonumber\\
    &\leq
    \frac{2\,\sqrt{c_{\text{\tiny\ref{lem:generalPNbasic}}}\,N}}{\sqrt{|J|\,n}}\,
    \sum_{h\in J}\sqrt{\sum_{i\in L}\beta_h(y_i)^2},\label{eq:J0cbound}
    \end{align}
    where we used that $|J|\leq c_{\text{\tiny\ref{lem:generalPNbasic}}}\,n$.

Combining \eqref{eq:J0upperbound}, \eqref{eq:probtildec},
and \eqref{eq:J0cbound} with Markov's inequality,
we get that for some realization of $\sigma$,
\begin{align*}
&\Big\|
\Big(\sum_{i\in L}\sigma_i\,\beta_j(y_i)-\beta^\sigma_j\Big)_{j\in J_0}\Big\|_2
\geq 
\frac{\tilde c}{2\sqrt{|J|}}\,\sum_{h\in J}\sqrt{\sum_{i\in L}\beta_h(y_i)^2}
>0
;\\
&|J_0|\leq \frac{24c_{\text{\tiny\ref{lem:generalPNbasic}}}\,n}{\tilde c};\\
&\Big\|
    \Big(\sum_{i\in L}\sigma_i\,\beta_j(y_i)-\beta^\sigma_j\Big)_{j\in J_0^c}\Big\|_2
\leq
\frac{16\sqrt{c_{\text{\tiny\ref{lem:generalPNbasic}}}\,N}}{\tilde c\,\sqrt{|J|\,n}}\,
    \sum_{h\in J}\sqrt{\sum_{i\in L}\beta_h(y_i)^2}.
\end{align*}
As a final element of the proof,
we apply Lemma~\ref{PN:incompcomb}.
Note that, as long as $c_{\text{\tiny\ref{lem:generalPNbasic}}}$
is sufficiently small, we have
$$
\frac{24c_{\text{\tiny\ref{lem:generalPNbasic}}}}{\tilde c}
<\delta_{\text{\tiny\ref{PN:incompcomb}}},
$$
whereas
$$
\rho_{\text{\tiny\ref{PN:incompcomb}}}\cdot
\frac{\tilde c}{2}>\frac{16\sqrt{c_{\text{\tiny\ref{lem:generalPNbasic}}}\,N}}{\tilde c\,\sqrt{n}}.
$$
Thus, the Euclidean normalization of the vector 
$
\big(\sum_{i\in L}\sigma_i\,\beta_j(y_i)-\beta^\sigma_j\big)_{j\in [N]}
$
is $(\delta_{\text{\tiny\ref{PN:incompcomb}}},
\rho_{\text{\tiny\ref{PN:incompcomb}}})$--compress\-ible,
leading to contradiction in view of Lemma~\ref{PN:incompcomb}.
The result follows.
\end{proof}

As the first application of the above lemma,
we consider the setting where vectors $y_i$, $i\in L$, in $(\R^n,\|\cdot\|_{P_{N,n}})$
satisfy
$
m_\delta(y_i,r)\leq |L|^{\alpha}\,2^{-2r}\,n,
$
for all $0\leq r\leq \log_2 \sqrt{|L|}$ and a small parameter $\alpha>0$,
and $y_i$'s induce a low-distortion embedding of $\ell_\infty^{|L|}$
into $(\R^n,\|\cdot\|_{P_{N,n}})$.
In the end, our goal is to show that for a typical realization of $P_{N,n}$
such vectors do not exist.
As the main step in verifying the assertion, in the next lemma we show
that for a significant
fraction of these vectors, one can define ``truncations'' $\tilde y_i$
having roughly comparable $\|\cdot\|_{P_{N,n}}$ and $\|\cdot\|_2$--norms
and also producing a low-distortion embedding of $\ell_\infty$.
Thus, we transfer the problem back to the setting of Proposition~\ref{prop:spansofcomp}
and ultimately arrive at contradiction at the end of this section, when completing
the proof of Theorem~\ref{th: Gluskinlinfty}.

\begin{lemma}\label{lem:stairs}
For every admissible choice of $K,K'$
there is a constant
$C_{\text{\tiny\ref{lem:stairs}}}\geq 2$
depending only on $K,K'$,
and a {\it universal constant} $c_{\text{\tiny\ref{lem:stairs}}}>0$
with the following property.
Assume that $n\geq c_{\text{\tiny\ref{lem:sparse_singular_values}}}^{-1}$
and $K'\geq \frac{N}{n}\geq K$.
Condition on the event
$\Omega_{\text{\tiny\ref{lem:sparse_singular_values}}}$.
Let $L$ be a finite index set
with $|L|\geq C_{\text{\tiny\ref{lem:stairs}}}$,
let $\delta\in(0,1]$ be a parameter,
and let $y_i$, $i\in L$, be
unit vectors in $\R^n$ with $\|y_i\|_{P_{N,n}}\in [\delta,2\delta]$.
Assume that 
for every $y_i$, $i\in L$, and every integer
$0\leq r\leq \log_2 \sqrt{|L|}$,
$$
m_\delta(y_i,r)\leq |L|^{\alpha}\,2^{-2r}\,n,
$$
where $\alpha\in(0,c_{\text{\tiny\ref{lem:stairs}}}]$
is a parameter.
Assume further that for every 
choice of numbers $v=(v_i)_{i\in L}$,
$$
\Big\|\sum_{i\in L}v_i\,y_i\Big\|_{P_{N,n}}
\leq \delta\,\|v\|_\infty\,|L|^{\alpha}.
$$
Then there is a collection $\tilde y_i$, $i\in U$,
of vectors of unit Euclidean length
with $|U|=\big\lfloor |L|^{1/16}\big\rfloor$ such that for every $\tilde y_i$,
\begin{itemize}
    \item $\|\tilde y_i\|_{P_{N,n}}\geq
    C_{\text{\tiny\ref{lem:stairs}}}^{-1}\;|L|^{-\alpha/2}\,
    (\log |L|)^{-1}$;
    \item For every choice of signs $\sigma_i$, $i\in U$, $\big\|
    \sum_{i\in U}\sigma_i\,\tilde y_i
    \big\|_{P_{N,n}}\leq C_{\text{\tiny\ref{lem:stairs}}}\,|L|^{\alpha}$.
\end{itemize}
\end{lemma}
\begin{proof}
We will assume that $C_{\text{\tiny\ref{lem:stairs}}}\geq 2$
is large. We do not attempt to optimize
the choice of $c_{\text{\tiny\ref{lem:stairs}}}$,
and can take $c_{\text{\tiny\ref{lem:stairs}}}:=0.01$
in the argument below.
Denote $\tilde r:=\lfloor \log_2\sqrt{|L|}\rfloor$,
and take $\varepsilon:=|L|^{-1/16}$.
Consider two subcases.
\begin{itemize}
    \item[1.] For at least $|L|/2$ vectors $y_i$,
    $$
    \sum_{j=1}^N |\beta_j(y_i)|\,
    {\bf 1}_{\{|\beta_j(y_i)|>\delta\,2^{\tilde r}/n\}}
    \leq \varepsilon\,\delta.
    $$
    Denote the corresponding set of indices by $L_1\subset L$.
    For every $i\in L_1$, define
    $$
    y_i':=\sum\limits_{j=1}^N \beta_j(y_i)\,
    {\bf 1}_{\{|\beta_j(y_i)|\leq \delta\,2^{\tilde r}/n\}}\,X_j.
    $$
    Note that
    \begin{equation}\label{eq:leqepsdelta}
    \|y_i-y_i'\|_{P_{N,n}}\leq \varepsilon\,\delta,
    \end{equation}
    implying
    \begin{equation}\label{eq:ell2ofyiprime}
    \|y_i'\|_{P_{N,n}}=\Omega(\delta),\quad
    \|y_i'\|_2=\Omega(\delta),
    \end{equation}
    where the latter relation follows from the 
    former by applying the definition of
    $\Omega_{\text{\tiny\ref{PN:inradius}}}$
    from the ``in-radius'' Lemma~\ref{PN:inradius},
    and where the implicit constants depend only on $K,K'$.
    Further, we have, in view of conditioning on 
    $\Omega_{\text{\tiny\ref{lem:sparse_singular_values}}}$,
    \begin{align*}
    \|y_i'\|_2
    &=O\bigg(\sqrt{n}\;\Big\|\Big(\beta_j(y_i)\,
    {\bf 1}_{\{|\beta_j(y_i)|\leq \delta\,2^{\tilde r}/n\}}\Big)_{j\in[N]}\Big\|_2\bigg)\\
    &=O\bigg(\sqrt{n}\;\sum\limits_{r=0}^{\tilde r}
    \sqrt{m_\delta(y_i,r)}\,\frac{\delta}{n}\,2^{r}
    \bigg)\\
    &= O\bigg(\sum\limits_{r=0}^{\tilde r}
    \sqrt{|L|^{\alpha}\,2^{-2r}\,n}\;
    \frac{\delta}{n}\,2^{r}\,\sqrt{n}\bigg)\\
    &=O\big(|L|^{\alpha/2}
    \,\delta\,\tilde r\big),
    \end{align*}
    where the implicit constants may only depend on $K,K'$.
    Renormalizing, we obtain a collection of {\it unit}
    vectors $\tilde y_i=\frac{y_i'}{\|y_i'\|_2}$, $i\in L_1$, such that
    $$
    \|\tilde y_i\|_{P_{N,n}}
    =\Omega\big(|L|^{-\alpha/2}\,{\tilde r}^{-1}\big).
    $$
    Further, using \eqref{eq:leqepsdelta} and \eqref{eq:ell2ofyiprime},
    as well as the assumptions of the lemma,
    for every non-empty subset $U_1\subset L_1$
    and every assignment of signs $\sigma_i$, $i\in U_1$,
    \begin{align*}
    \Big\|
    \sum_{i\in U_1}\sigma_i\,\tilde y_i
    \Big\|_{P_{N,n}}
    &\leq
    \Big\|
    \sum_{i\in U_1}\sigma_i\,\Big(\frac{y_i'}{\|y_i'\|_2}
    -
    \frac{y_i}{\|y_i'\|_2}
    \Big)
    \Big\|_{P_{N,n}}
    +
    \Big\|
    \sum_{i\in U_1}\sigma_i\,\frac{y_i}{\|y_i'\|_2}
    \Big\|_{P_{N,n}}\\
    &\leq 
    \sum_{i\in U_1} \frac{1}{\|y_i'\|_2}\,\|y_i'-y_i\|_{P_{N,n}}
    +\delta\,|L|^{\alpha}\,\Big\|\Big(\frac{\sigma_i}{\|y_i'\|_2}\Big)_{i\in U_1}\Big\|_{\infty}
    \\
    &\leq O\big(\varepsilon\,|U_1|\big)+
    O\big(|L|^{\alpha}\big).
    \end{align*}
    Taking $U_1$ to be any subset of $L_1$
    of size $\big\lfloor|L|^{1/16}\big\rfloor$,
    we get the result.

    \item[2.] For at least $|L|/2$ vectors $y_i$,
    $$
    \sum_{j=1}^N |\beta_j(y_i)|\,
    {\bf 1}_{\{|\beta_j(y_i)|>\delta\,2^{\tilde r}/n\}}
    > \varepsilon\,\delta.
    $$
    Denote the corresponding set of indices by $L_2$.
    Let $U_2\subset L_2$ be a subset of $L_2$
    of size 
    $|U_2|= \big\lfloor |L|^{1/8+3\,c_{\text{\tiny\ref{lem:stairs}}}}\big\rfloor$.
    Recall that by the assumptions of the lemma, for every
    choice of the signs $\sigma_i$, $i\in U_2$,
    \begin{equation}\label{eq: akjfnlakjfnlkfjn}
    \Big\|\sum_{i\in U_2}\sigma_i\,y_i\Big\|_{P_{N,n}}
    \leq \delta\,|L|^{\alpha}.
    \end{equation}
    Define $J$ to be the collection of all indices
    $j\in[N]$ such that
    $|\beta_j(y_i)|>\delta\,2^{\tilde r}/n$
    for some $i\in U_2$.
    Observe that, in view of the upper bound $\|y_i\|_{P_{N,n}}\leq 2\delta$,
    the assumption that the constant 
    $C_{\text{\tiny\ref{lem:stairs}}}$ is large,
    and the definition of $U_2$,
    \begin{equation}\label{eq:jhbfasafjnalkjnfkajn}
    |J|\leq |U_2|\cdot \frac{2n}{2^{\tilde r}}
    \leq 4n\,\frac{|U_2|}{\sqrt{|L|}}
    \leq c_{\text{\tiny\ref{lem:generalPNbasic}}}\,n.
    \end{equation}
    Further, by our assumption that $L$ is large,
    \begin{align*}
    c_{\text{\tiny\ref{lem:generalPNbasic}}}\,
    \sum_{h\in J}\sqrt{\sum_{i\in U_2}\beta_h(y_i)^2}
    &\geq
    \frac{c_{\text{\tiny\ref{lem:generalPNbasic}}}}{\sqrt{|U_2|}}\,
    \sum_{i\in U_2}\sum_{j\in J}|\beta_j(y_i)|\\
    &\geq c_{\text{\tiny\ref{lem:generalPNbasic}}}\,
    \varepsilon\,\delta\,\sqrt{|U_2|}\\
    &= c_{\text{\tiny\ref{lem:generalPNbasic}}}\,
    |L|^{-1/16}\,\delta\,\sqrt{|U_2|}
    \geq 
    \delta\,|L|^{\alpha},
    \end{align*}
    so, in view of the upper bound \eqref{eq: akjfnlakjfnlkfjn}, the assumption \eqref{eq: 315aslkfnl}
    from Lemma~\ref{lem:generalPNbasic} (with $L$ replaced by $U_2$) is satisfied.
    Applying Lemma~\ref{lem:generalPNbasic},
    we get that there must exist 
    at least $c_{\text{\tiny\ref{lem:generalPNbasic}}}\,n$
    indices $j\in J^c$ satisfying
    \begin{align}
    \sum_{i\in U_2}|\beta_j(y_i)|
    &\geq
    \sqrt{\sum_{i\in U_2}\beta_j(y_i)^2}
    \geq \frac{c_{\text{\tiny\ref{lem:generalPNbasic}}}}{\sqrt{|J|\,n}}\,
\sum_{h\in J}
\,\sqrt{\sum_{i\in U_2}\beta_h(y_i)^2}\nonumber\\
    &\geq 
    \frac{c_{\text{\tiny\ref{lem:generalPNbasic}}}}{\sqrt{|U_2|\,|J|\,n}}\,\sum_{i\in U_2}\sum_{h\in J}|\beta_h(y_i)|\nonumber\\
    &\geq
    \frac{c_{\text{\tiny\ref{lem:generalPNbasic}}}\,
    \varepsilon\,\delta\sqrt{|U_2|}}
    {\sqrt{|J|\,n}}\nonumber\\
    &\geq \frac{c_{\text{\tiny\ref{lem:generalPNbasic}}}\,
    \varepsilon\,\delta\,|L|^{1/4}}
    {2n},\label{eq:40958y2305938450398}
    \end{align}
    where we applied the first inequality from \eqref{eq:jhbfasafjnalkjnfkajn} to bound $|J|$.
   At the same time, in view of the assumptions
   of the lemma, for every $i\in U_2$,
   $$
   \sum_{j\in[N]\setminus J}
   |\beta_j(y_i)|
   \leq \sum\limits_{r=0}^{\tilde r}
    |L|^{\alpha}\,2^{-2r}\,n\;
    \frac{\delta}{n}\,2^{r}
    \leq 2\,|L|^{\alpha}\,\delta,
   $$
implying
$$
\sum_{j\in[N]\setminus J}
\sum_{i\in U_2}
   |\beta_j(y_i)|
\leq 2\,|U_2|\,|L|^{\alpha}\,\delta.
$$
Combining the last relation with \eqref{eq:40958y2305938450398},
we get
$$
2\,|U_2|\,|L|^{\alpha}\,\delta
\geq 
c_{\text{\tiny\ref{lem:generalPNbasic}}}\,n\cdot
\frac{c_{\text{\tiny\ref{lem:generalPNbasic}}}\,
    \varepsilon\,\delta\,|L|^{1/4}}
    {2n},
$$
leading to contradiction. Thus, the second subcase
is infeasible, and the result follows.
\end{itemize}
\end{proof}

Having considered the case of ``controlled decay'' of the parameters
$m_\delta(y_i,r)$, in the next two lemmas we treat the complimentary
case of ``spiky''
coefficients $\beta(y_i)$. Our goal is to show that vectors
satisfying such conditions cannot induce low-distortion embeddings of $\ell_\infty$.

\begin{lemma}[Preprocessing lemma]\label{lem:cleaning II}
Condition on any non-degenerate realization of $X_1,\dots,X_N$.
Let $|L|\geq 2$,
let $\alpha,\delta\in(0,1]$ and $\varepsilon\in(0,1/2]$
be parameters,
and let $y_i$, $i\in L$, be
unit vectors in $\R^n$ with $\|y_i\|_{P_{N,n}}\in [\delta,2\delta]$.
Assume that 
for every $y_i$, $i\in L$,
$$
\max\limits_{0\leq r\leq \log_2 \sqrt{|L|}}
\frac{m_\delta(y_i,r)}{2^{-2r}\,n}
>
|L|^{\alpha}.
$$
Then there is a collection $\tilde L\subset L$
of size at least
$c_{\text{\tiny\ref{lem:cleaning II}}}\,(\log|L|)^{-2}|L|$
(for some constant
$c_{\text{\tiny\ref{lem:cleaning II}}}>0$
which may only depend on $K'$)
such that for 
some integer $0\leq r\leq 3\log_2|L|+4$,
a real number
$$p\geq 
|L|^{\alpha-\varepsilon}\,2^{-2r}\,
2^{\max(0,r-\log_2\sqrt{|L|})}\,n,$$ 
and
every $i\in\tilde L$,
$$
m_\delta(y_i,r)\in[p,2p],
$$
and
\begin{align*}
&m_\delta(y_i,h)
< 2^{(-2+\varepsilon)(h-r)}\,m_\delta(y_i,r),
\quad
&0\leq h< r;\\
&m_\delta(y_i,h)
< 2^{-(h-r)/2}\,m_\delta(y_i,r),
\quad
&h> r.
\end{align*}
\end{lemma}
\begin{proof}
Fix for a moment any $i\in L$.
Consider an algorithm constructing a sequence of numbers
$r_0(i),r_1(i),\dots$ as follows.
Set $r_0(i)$ to be arbitrary $0\leq r\leq \log_2 \sqrt{|L|}$
such that $m_\delta(y_i,r)
>
|L|^{\alpha}\,2^{-2r}\,n$
(such a number exists by the assumptions of the lemma).
Further, given a number $r_\ell(i)$ obtained at the $\ell$--th
step of the algorithm, 
we let $r_{\ell+1}(i)$ be an arbitrary non-negative integer $r$
such that either
\begin{equation}\label{eq:aokjbfoaskfjnpfkajn}
\frac{m_\delta(y_i,r)}
{2^{(-2+\varepsilon)\,r}}
\geq \frac{m_\delta(y_i,r_\ell(i))}
{2^{(-2+\varepsilon)\,r_{\ell}(i)}}
\quad\mbox{and}\quad
0\leq r< r_\ell(i),    
\end{equation}
or
\begin{equation}\label{eq:alkdshjfbsofhjbsofjh}
\frac{m_\delta(y_i,r)}
{2^{-r/2}}
\geq
\frac{m_\delta(y_i,r_\ell(i))}{2^{-r_{\ell}(i)/2}}
\quad\mbox{and}\quad
r> r_\ell(i).
\end{equation}
Whenever such a number $r_{\ell+1}(i)$
cannot be found, we stop the process and define $\tilde r(i)$
to be the last element of the constructed sequence.

We claim that the algorithm described above, terminates,
and that the obtained index $\tilde r(i)$ satisfies
\begin{equation}\label{eq:ojhboajhbfoasjfbslj}
\begin{split}
m_\delta(y_i,\tilde r(i))
&>
|L|^{\alpha}\,2^{-2\tilde r(i)-\varepsilon\log_2\sqrt{|L|}}\,
2^{(3/2-\varepsilon)\max(0,\tilde r(i)-\log_2\sqrt{|L|})}\,n\\
&\geq 
|L|^{\alpha-\varepsilon}\,2^{-2\tilde r(i)}\,
2^{\max(0,\tilde r(i)-\log_2\sqrt{|L|})}\,n    
\end{split}
\end{equation}
and
\begin{equation}\label{eq:aoiubfoaskjfnslfk}
2^{\tilde r(i)}\leq 16\,|L|^{3-2\alpha}.
\end{equation}
Indeed, observe that, in view of 
\eqref{eq:aokjbfoaskfjnpfkajn} and
\eqref{eq:alkdshjfbsofhjbsofjh},
the ratios
$$
\frac{m_\delta(y_i,r_\ell(i))}{2^{-2r_\ell(i)+\varepsilon\,r_\ell(i)}},
\quad \ell=0,1,\dots,
$$
form a non-decreasing sequence,
and, moreover,
\begin{equation}\label{eq: oiyugoajlafkjnslfkjn}
\frac{m_\delta(y_i,r_{\ell+1}(i))}
{2^{-2r_{\ell+1}(i)+\varepsilon\,r_{\ell+1}(i)}}
\geq
2^{(3/2-\varepsilon)(r_{\ell+1}(i)-r_\ell(i))}\;
\frac{m_\delta(y_i,r_\ell(i))}{2^{-2r_\ell(i)+\varepsilon\,r_\ell(i)}},
\quad \mbox{whenever $r_{\ell+1}(i)>r_\ell(i)$}.
\end{equation}
Since
$$
\sup\limits_{r\geq 0}\frac{m_\delta(y_i,r)}
{2^{-2r+\varepsilon\,r}}<\infty,
$$
the above observation implies that the algorithm must terminate.
Moreover, in view of \eqref{eq: oiyugoajlafkjnslfkjn},
we have
$$
\frac{m_\delta(y_i,\tilde r(i))}
{2^{-2\tilde r(i)+\varepsilon\,\tilde r(i)}}
>
2^{(3/2-\varepsilon)\max(0,\tilde r(i)-r_0(i))}\;
\frac{m_\delta(y_i,r_0(i))}{2^{-2r_0(i)+\varepsilon\,r_0(i)}},
$$
and, in particular,
\begin{align}
m_\delta(y_i,\tilde r(i))
&>|L|^{\alpha}\,2^{-2r_0(i)}\,n\;
\frac{2^{(3/2-\varepsilon)\max(0,\tilde r(i)-r_0(i))}\;
2^{-2\tilde r(i)+\varepsilon\,\tilde r(i)}}
{2^{-2r_0(i)+\varepsilon\,r_0(i)}}\nonumber\\
&=
|L|^{\alpha}\,2^{-2\tilde r(i)}\,n\;
2^{(3/2-\varepsilon)\max(0,\tilde r(i)-r_0(i))}\;
2^{\varepsilon\,\tilde r(i)-\varepsilon\,r_0(i)}
\label{eq:ajhbkjfhbskfjhsb}\\
&\geq
|L|^{\alpha}\,2^{-2\tilde r(i)-\varepsilon\log_2\sqrt{|L|}}\,
2^{(3/2-\varepsilon)\max(0,\tilde r(i)-\log_2\sqrt{|L|})}\,n,\nonumber
\end{align}
implying \eqref{eq:ojhboajhbfoasjfbslj}.
Finally, whenever $\tilde r(i)>r_0(i)$,
we necessarily have in view of the bound $\|y_i\|_{P_{N,n}}\leq 2\delta$
and \eqref{eq:ajhbkjfhbskfjhsb},
$$
4n\,2^{-\tilde r(i)}=
2\delta\;\Big(\frac{\delta}{n}\,2^{\tilde r(i)-1}\Big)^{-1}
\geq
m_\delta(y_i,\tilde r(i))
\geq
|L|^{\alpha}\,2^{-2\tilde r(i)}\,n\;
2^{1.5(\tilde r(i)-r_0(i))}
=
|L|^{\alpha}\,2^{-\tilde r(i)/2}\,n\;
2^{-1.5 r_0(i)}.
$$
Taking squares and rearranging, we obtain
$$
2^{\tilde r(i)}\leq
16\,|L|^{-2\alpha}\;
2^{3 r_0(i)},
$$
implying \eqref{eq:aoiubfoaskjfnslfk}, and the claim is established.

The relation \eqref{eq:aoiubfoaskjfnslfk}
implies that there are at most $O(\log |L|)$
admissible values for $\tilde r(i)$ for each $i\in L$.
Furthermore, in view of \eqref{eq:ojhboajhbfoasjfbslj},
for all $i\in L$, we have
$N\geq m_\delta(y_i,\tilde r(i))\geq |L|^{-O(1)}\,n$,
and a combination of a dyadic partitioning 
of the range $|L|^{-O(1)}\,n\dots N$
with the pigeonhole principle
implies the result.
\end{proof}

\begin{lemma}\label{lem:spikymdelta}
For every admissible choice of $K,K'$
and every $\alpha\in(0,1]$
there is a constant
$C_{\text{\tiny\ref{lem:spikymdelta}}}\geq 2$
depending only on $K,K',\alpha$
with the following property.
Assume that $n\geq c_{\text{\tiny\ref{lem:sparse_singular_values}}}^{-1}$.
Condition on 
any realization of $X_1,\dots,X_N$
from $\Omega_{\text{\tiny\ref{lem:sparse_singular_values}}}$.
Let $L$ be a finite index set
with $|L|\geq C_{\text{\tiny\ref{lem:spikymdelta}}}$,
let $\delta\in(0,1]$ be a parameter,
and let $y_i$, $i\in L$, be
unit vectors in $\R^n$ with $\|y_i\|_{P_{N,n}}\in [\delta,2\delta]$.
Assume that 
for every $y_i$, $i\in L$,
$$
\max\limits_{0\leq r\leq \log_2 \sqrt{|L|}}
\frac{m_\delta(y_i,r)}{2^{-2r}\,n}
>
|L|^{\alpha}.
$$
Then there is a tuple $v=(v_i)_{i\in L}$ with $\|v\|_\infty=1$
such that
$$
\Big\|\sum_{i\in L}v_i\,y_i\Big\|_{P_{N,n}}
\geq \delta\,|L|^{\alpha/5}.
$$
\end{lemma}
\begin{proof}
Define $\varepsilon:=\alpha/2$.
We will assume that the constant $C_{\text{\tiny\ref{lem:spikymdelta}}}$
is large, and, in particular,
$|L|^{\alpha-\varepsilon}\,n> N$.
Applying Lemma~\ref{lem:cleaning II},
we obtain a collection $\tilde L\subset L$
of size at least
$c_{\text{\tiny\ref{lem:cleaning II}}}\,(\log|L|)^{-2}|L|$
such that for 
some $0\leq r\leq 3\log_2|L|+4$, some
\begin{equation}\label{eq:ajhbjfbkjfhlgjgnljn}
p\geq 
|L|^{\alpha-\varepsilon}\,2^{-2r}\,
2^{\max(0,r-\log_2\sqrt{|L|})}\,n
\geq |L|^{\alpha-\varepsilon}\,2^{-2r}\,n,    
\end{equation}
and
every $i\in\tilde L$,
$$
m_\delta(y_i,r)\in[p,2p],
$$
and
\begin{align*}
&m_\delta(y_i,h)
< 2^{-2(h-r)-\varepsilon(r-h)}\,m_\delta(y_i,r),
\quad
&0\leq h< r;\\
&m_\delta(y_i,h)
< 2^{-(h-r)/2}\,m_\delta(y_i,r),
\quad
&h> r.
\end{align*}
Note that in view of the trivial bound $p\leq N$ and
\eqref{eq:ajhbjfbkjfhlgjgnljn}, we have
\begin{equation}\label{eq:ajhbajhfbalkvkjbnrpk}
r\geq \frac{1}{2}\log_2\frac{|L|^{\alpha-\varepsilon}\,n}{N}
=\Omega(\log|L|)
>0.    
\end{equation}
Let $r'$ be the largest integer in the interval $[1,r-1]$ 
such that
\begin{equation}\label{eq:ajhbasfjnpfkjsnfosakjfnl}
\frac{1}{\varepsilon}
\,2^{\varepsilon(r'-r)}
<
\frac{c_{\text{\tiny\ref{lem:generalPNbasic}}}^3}{64}\,
\frac{1}{\log_2|L|+2}
\end{equation}
(note that $r'$ is well defined as long as $C_{\text{\tiny\ref{lem:spikymdelta}}}$ is sufficiently large).
Further, let $U\subset \tilde L$ be a subset of $\tilde L$
of size 
$$
|U|= \min\bigg(\bigg\lfloor 
\frac{c_{\text{\tiny\ref{lem:generalPNbasic}}}\,n}
{8\,p\cdot 2^{(-2+\varepsilon)(r'-r)}}\bigg\rfloor,|\tilde L|\bigg).
$$
The bound $\|y_i\|_{P_{N,n}}\leq 2\delta$, $i\in\tilde L$,
together with \eqref{eq:ajhbajhfbalkvkjbnrpk}
imply
$$
p\leq \bigg(\frac{\delta\,2^{r-1}}{n}\bigg)^{-1}\,2\delta
=\frac{4n}{2^r}
\leq \frac{4n}{\sqrt{\frac{|L|^{\alpha-\varepsilon}\,n}{N}}}
\leq \frac{4\sqrt{K'}\,n}
{|L|^{(\alpha-\varepsilon)/2}}.
$$
The last inequality, together with \eqref{eq:ajhbasfjnpfkjsnfosakjfnl}
and the definition of $U$, implies 
$$
|U|\geq \frac{|L|^{(\alpha-\varepsilon)/2}}{(\log|L|)^{C_\alpha}}>0,
$$
for some $C_\alpha\geq 1$ depending only on $\alpha,K',K$.

We are going to apply Lemma~\ref{lem:generalPNbasic}
to get the required result.
We will argue by contradiction.
Assume that for every
choice of the signs $\sigma_i$, $i\in U$,
\begin{equation}\label{eq:oajhbfoajhfboajsflaknf}
\Big\|\sum_{i\in U}\sigma_i\,y_i\Big\|_{P_{N,n}}
\leq \delta\,|L|^{0.49(\alpha-\varepsilon)}.    
\end{equation}
Denote by $J'$ the collection of all indices $j\in[N]$
such that 
$|\beta_j(y_i)|>\delta\,2^{r'-1}/n$
for some $i\in U$,
and, similarly, let $J$
be the collection of all indices $j\in[N]$
such that 
$|\beta_j(y_i)|>\delta\,2^{r-1}/n$
for some $i\in U$.
From the above definitions of $J'$ and $U$
and the bounds on $m_\delta(y_i,h)$, it follows that
\begin{equation}\label{eq: JJprimesize}
\begin{split}
|J|\leq
|J'|&\leq 
\sum_{i\in U}
\sum_{h\geq r'}
m_\delta(y_i,h)\\
&\leq
2p\,|U|
\sum_{h=r'}^{r-1} 2^{(-2+\varepsilon)(h-r)}
+2p\,|U|+2p\,|U|\sum_{h>r} 2^{-(h-r)/2}\\
&<
4p\,|U|
\cdot 2^{(-2+\varepsilon)(r'-r)}
\leq \frac{c_{\text{\tiny\ref{lem:generalPNbasic}}}}{2}\,n.
\end{split}
\end{equation}
We further partition $J$ as
$$
J=\bigcup\limits_{w=1}^{\lceil\log_2|U|\rceil+1} J_w,
$$
where
$$
J_w:=\big\{j\in J:\;
|\{i\in U:\;|\beta_j(y_i)|>\delta\,2^{r-1}/n\}|
\in[2^{w-1},2^w)\big\}.
$$
Observe that
$$
\sum_{w=1}^{\lceil\log_2|U|\rceil+1}
|J_w|\,2^{w-1}
\leq
\sum_{j\in[N]}\sum_{i\in U}
{\bf 1}_{\{|\beta_j(y_i)|>\delta\,2^{r-1}/n\}}
=\sum_{i\in U}\sum_{h\geq r}m_\delta(y_i,h)
\leq \sum_{w=1}^{\lceil\log_2|U|\rceil+1}
|J_w|\,2^{w},
$$
where, by the condition on $m_\delta(y_i,r)$,
$$
p\,|U|\leq
\sum_{i\in U}m_\delta(y_i,r)\leq
\sum_{i\in U}\sum_{h\geq r}m_\delta(y_i,h)
<4\,\sum_{i\in U}m_\delta(y_i,r)
\leq 8p\,|U|.
$$
Denote by $w_0$ an integer in $1\dots \lceil\log_2|U|\rceil+1$
corresponding to the largest value of the product $|J_w|\,2^{w}$
(with an arbitrary tie-breaking), so that
$$
|J_{w_0}|\,2^{w_0}
\geq \frac{1}{\lceil\log_2|U|\rceil+1}
\sum_{w=1}^{\lceil\log_2|U|\rceil+1}
|J_w|\,2^{w}
\geq \frac{p\,|U|}{\lceil\log_2|U|\rceil+1}.
$$
Further, we have
\begin{align}
c_{\text{\tiny\ref{lem:generalPNbasic}}}\,
\sum_{h\in J_{w_0}}\sqrt{\sum_{i\in U}\beta_h(y_i)^2}
&\geq
c_{\text{\tiny\ref{lem:generalPNbasic}}}\,\delta\,2^{r-1}\,
2^{(w_0-1)/2}\,
\frac{|J_{w_0}|}{n}\nonumber\\
&\geq \frac{c_{\text{\tiny\ref{lem:generalPNbasic}}}\,\delta\,2^{r-1}}{2^{(w_0+1)/2}\,n}\,\frac{p\,|U|}{\lceil\log_2|U|\rceil+1}\nonumber\\
&\geq 
\frac{c_{\text{\tiny\ref{lem:generalPNbasic}}}\,\delta\,2^{r}}{24\,n}\,\frac{p\,\sqrt{|U|}}{\log_2|L|}\label{eq:kajhfbkasjfnourygowf},
\end{align}
where we used that $w_0\leq\lceil\log_2|U|\rceil+1$.
Consider two cases. 
\begin{itemize}
\item 
$\Big\lfloor 
\frac{c_{\text{\tiny\ref{lem:generalPNbasic}}}\,n}
{8\,p\cdot 2^{(-2+\varepsilon)(r'-r)}}\Big\rfloor\leq |\tilde L|$.
In this case, by the definition of $U$ we have
$$|U|=\bigg\lfloor 
\frac{c_{\text{\tiny\ref{lem:generalPNbasic}}}\,n}
{8\,p\cdot 2^{(-2+\varepsilon)(r'-r)}}\bigg\rfloor,$$
and
\begin{align*}
2^{r}\,p\,\sqrt{|U|}
&\geq 2^{r-1}
\frac{c_{\text{\tiny\ref{lem:generalPNbasic}}}^{1/2}\,\sqrt{n\,p}}
{2\sqrt{2}\cdot 2^{(-1+\varepsilon/2)(r'-r)}}\\
&>
\frac{c_{\text{\tiny\ref{lem:generalPNbasic}}}^{1/2}}{8}
\,2^{(1-\varepsilon/2)(r'-r)}\,
|L|^{(\alpha-\varepsilon)/2}\,
2^{\max(0,r-\log_2\sqrt{|L|})/2}\,n,
\end{align*}
where we used the assumptions on parameter $p$.
Combining the last relation with \eqref{eq:kajhfbkasjfnourygowf}
and \eqref{eq:oajhbfoajhfboajsflaknf}
and the definition of $r'$, and taking into consideration that
$L$ is large,
we obtain
\begin{equation}\label{eq:ljahboafhbpigjnlkjn}
c_{\text{\tiny\ref{lem:generalPNbasic}}}\,
\sum_{h\in J_{w_0}}\sqrt{\sum_{i\in U}\beta_h(y_i)^2}
\geq 
\frac{c_{\text{\tiny\ref{lem:generalPNbasic}}}^{3/2}\,\delta}
{200}\,\frac{|L|^{(\alpha-\varepsilon)/2}}{\log_2|L|}
\,2^{(1-\varepsilon/2)(r'-r)}
\geq 
\Exp_{\sigma}\,\|\beta(y^\sigma)\|_1.    
\end{equation}

\item $\Big\lfloor 
\frac{c_{\text{\tiny\ref{lem:generalPNbasic}}}\,n}
{8\,p\cdot 2^{(-2+\varepsilon)(r'-r)}}\Big\rfloor> |\tilde L|$.
Then $|U|=|\tilde L|$, and
$$
2^{r}\,p\,\sqrt{|U|}
\geq
|L|^{\alpha-\varepsilon}\,2^{-r}\,
2^{\max(0,r-\log_2\sqrt{|L|})}\,n\,\sqrt{|\tilde L|}
\geq |L|^{\alpha-\varepsilon}\,n,
$$
and hence, again applying \eqref{eq:kajhfbkasjfnourygowf}
and \eqref{eq:oajhbfoajhfboajsflaknf},
\begin{equation}\label{eq:aljhbfkgjnpjnfkjnlaksjn}
c_{\text{\tiny\ref{lem:generalPNbasic}}}\,
\sum_{h\in J_{w_0}}\sqrt{\sum_{i\in U}\beta_h(y_i)^2}
\geq 
\frac{c_{\text{\tiny\ref{lem:generalPNbasic}}}\,\delta}{24}\,\frac{|L|^{\alpha-\varepsilon}}{\log_2|L|}
\geq \Exp_{\sigma}\,\|\beta(y^\sigma)\|_1.
\end{equation}

\end{itemize}

In either case, 
\eqref{eq: JJprimesize},
\eqref{eq:ljahboafhbpigjnlkjn}, and
\eqref{eq:aljhbfkgjnpjnfkjnlaksjn}
allows us to apply
Lemma~\ref{lem:generalPNbasic},
    we get that there must exist 
    at least $c_{\text{\tiny\ref{lem:generalPNbasic}}}\,n$
    indices $j\in J^c_{w_0}$ satisfying
    \begin{align*}
    \sqrt{\sum_{i\in U}\beta_j(y_i)^2}
    &\geq 
    \frac{c_{\text{\tiny\ref{lem:generalPNbasic}}}}{\sqrt{|J_{w_0}|\,n}}\,
    \sum_{h\in J_{w_0}}
    \,\sqrt{\sum_{i\in U}\beta_h(y_i)^2}\\
    &\geq 
    \frac{c_{\text{\tiny\ref{lem:generalPNbasic}}}\,
    \sqrt{|J_{w_0}|}}{\sqrt{n}}\,
    \frac{\delta\,2^{r-1}}{n}\,
    2^{(w_0-1)/2},
    \end{align*}
    implying (in view of the bound $|[N]\setminus J'|
    \geq N-\frac{c_{\text{\tiny\ref{lem:generalPNbasic}}}}{2}\,n$
    in \eqref{eq: JJprimesize})
    \begin{equation}\label{eq: aoijfbofkjnsflksjnflk}
    \sum_{j\in[N]\setminus J'}\sum_{i\in U}\beta_j(y_i)^2
    \geq 
    \frac{c_{\text{\tiny\ref{lem:generalPNbasic}}}^3
    \delta^2\,2^{2r}\,2^{w_0}|J_{w_0}|}{16\,n^2}.
    \end{equation}
   At the same time, in view of definition of $J'$
   and the bounds on $m_\delta(y_i,h)$,
   for every $i\in U$,
   $$
   \sum_{j\in[N]\setminus J'}
   \beta_j(y_i)^2
   \leq 2p\,\sum\limits_{h=0}^{r'-1}
    2^{(-2+\varepsilon)(h-r)}\,2^{2h}\,\frac{\delta^2}{n^2}
    <
    \frac{4\,p\,\delta^2\,2^{(-2+\varepsilon)(r'-r)}\,2^{2r'}}{\varepsilon\,n^2},
   $$
implying
$$
\sum_{j\in[N]\setminus J'}
\sum_{i\in U}
   \beta_j(y_i)^2
\leq \frac{4\,p\,\delta^2\,2^{(-2+\varepsilon)(r'-r)}\,2^{2r'}\,|U|}{\varepsilon\,n^2}.
$$
Combining the last relation with \eqref{eq: aoijfbofkjnsflksjnflk},
we get
$$
\frac{1}{\varepsilon}
\,p\,2^{\varepsilon(r'-r)}\,|U|
\geq 
\frac{1}{64}\,c_{\text{\tiny\ref{lem:generalPNbasic}}}^3
\,2^{w_0}|J_{w_0}|
\geq 
\frac{c_{\text{\tiny\ref{lem:generalPNbasic}}}^3}{64}\,
\frac{p\,|U|}{\lceil\log_2|U|\rceil+1}.
$$
leading to contradiction in view of the definition of $r'$.
\end{proof}

Having treated both settings where the vectors $\beta(y_i)$ are either ``flat''
(Lemma~\ref{lem:stairs}) or ``spiky'' (Lemma \ref{lem:spikymdelta}), the proof of Theorem~\ref{th: Gluskinlinfty}
is essentially complete. What remains are some simple reduction steps
and a formal link between Lemma~\ref{lem:stairs} and Proposition~\ref{prop:spansofcomp}.

\begin{proof}[Proof of Theorem~\ref{th: Gluskinlinfty}]
We will assume that the constants $\alpha\in(0,0.01]$
and $c_{\text{\tiny\ref{th: Gluskinlinfty}}}$
are small, and $C_{\text{\tiny\ref{th: Gluskinlinfty}}}>0$ is large.
We can therefore assume without loss of generality that $k$ is large.
We will argue by contradiction.
Let $E$ be a $k$--dimensional linear subspace of $(\R^n,\|\cdot\|_{P_{N,n}})$
with 
$$
d_{BM}(\ell_\infty^k,E)
\leq c_{\text{\tiny\ref{th: Gluskinlinfty}}}\,k^\alpha.
$$
Applying Lemma~\ref{lem: ellinftyreg},
we find $\tilde k:=\lfloor c_{\text{\tiny\ref{lem: ellinftyreg}}}\,k/\log k\rfloor$
unit vectors $y_1,y_2,\dots,y_{\tilde k}$
and a number $\delta>0$ such that
$\|y_i\|\in[\delta,2\delta]$ for every $i\leq \tilde k$,
and for every choice of numbers $v=(v_i)_{i\leq \tilde k}$,
    \begin{equation}\label{eq:alksjbfasopfjnlfkjnla}
    \Big\|
    \sum_{i=1}^{\tilde k}v_i\,y_i
    \Big\|_{P_{N,n}}\leq C_{\text{\tiny\ref{lem: ellinftyreg}}}\,c_{\text{\tiny\ref{th: Gluskinlinfty}}}\,\delta\,k^\alpha\,\|v\|_\infty
    <\delta\,(\tilde k/2)^{2\alpha}\,\|v\|_\infty.
    \end{equation}
We can (and will) assume that
$\tilde k=
\lfloor c_{\text{\tiny\ref{lem: ellinftyreg}}}\,k/\log k\rfloor\geq
2\,\max(C_{\text{\tiny\ref{lem:spikymdelta}}},
C_{\text{\tiny\ref{lem:stairs}}})$.
Suppose first that for $\lceil \tilde k/2\rceil $ indices $i$,
$$
\max\limits_{0\leq r\leq \log_2 \sqrt{\lceil \tilde k/2\rceil}}
\frac{m_\delta(y_i,r)}{2^{-2r}\,n}
>
\lceil \tilde k/2\rceil^{20\alpha}.
$$
Denote the corresponding set by $L$.
Applying Lemma~\ref{lem:spikymdelta}
with $20\alpha$ in place of $\alpha$,
we obtain a tuple $v=(v_i)_{i\in L}$ with $\|v\|_\infty=1$
such that
$$
\Big\|\sum_{i\in L}v_i\,y_i\Big\|_{P_{N,n}}
\geq \delta\,|L|^{4\alpha}=\delta\,\lceil\tilde k/2\rceil^{4\alpha},
$$
which contradicts \eqref{eq:alksjbfasopfjnlfkjnla}.
We conclude that there is a collection of indices $L'$
of size $\lceil \tilde k/2\rceil\geq C_{\text{\tiny\ref{lem:stairs}}}$
such that
$$
\max\limits_{0\leq r\leq \log_2 \sqrt{\lceil \tilde k/2\rceil}}
\frac{m_\delta(y_i,r)}{2^{-2r}\,n}
\leq
\lceil \tilde k/2\rceil^{20\alpha},\quad i\in L'.
$$
Note that we assume $20\alpha\leq c_{\text{\tiny\ref{lem:stairs}}}$.
Applying Lemma~\ref{lem:stairs} with $20\alpha$
in place of $\alpha$,
we obtain a collection
$\tilde y_i$, $i\in U$,
of vectors of unit Euclidean length
with $|U|=\big\lfloor |L'|^{1/16}\big\rfloor$ such that for every $\tilde y_i$,
\begin{itemize}
    \item $\|\tilde y_i\|_{P_{N,n}}\geq
    C_{\text{\tiny\ref{lem:stairs}}}^{-1}\;|L'|^{-10\alpha}\,
    (\log |L'|)^{-1}$;
    \item For every choice of signs $\sigma_i$, $i\in U$, $\big\|
    \sum_{i\in U}\sigma_i\,\tilde y_i
    \big\|_{P_{N,n}}\leq C_{\text{\tiny\ref{lem:stairs}}}\,|L'|^{20\alpha}$.
\end{itemize}
It remains to note that, as long as $\alpha$
is sufficiently small, the existence of such collection
$\tilde y_i$, $i\in U$, contradicts
Proposition~\ref{prop:spansofcomp}.
The result follows.
\end{proof}

\subsection{Theorems~B and~A}

Theorem~B from the introduction follows directly
from Theorem~\ref{th: Gluskinlinfty}
by observing that the intersection of events
$\Omega_{\text{\tiny\ref{lem:sparse_singular_values}}}
\cap \Omega_{\text{\tiny\ref{lem:dotprodstat}}}$
has probability $1-\frac{O(1)}{n}$, in view of Lemmas~\ref{lem:sparse_singular_values} 
and~\ref{lem:dotprodstat}.
Further, Theorem~A follows from Theorem~B by applying the following
``non-asymptotic'' version of the Maurey--Pisier theorem \cite{MP76}:

\begin{theorem}[A non-asymptotic restatement of the Maurey--Pisier theorem \cite{MP76}]\label{prop: na}
For every $k_0\geq 2$ there are finite constants $q\geq 2$ and $C_q>0$
depending {\bf only} on $k_0$, with the following property.
Let ${\bf X}$ be a Banach space with $\dim {\bf X}> k_0$, and suppose that
for every $k_0$--dimensional subspace $E$ of ${\bf X}$,
the Banach--Mazur distance from $E$ to $\ell_\infty^{k_0}$ is greater than $100$.
Then
necessarily ${\bf X}$ has cotype $q$ with constant $C_q$.
\end{theorem}


\section{Proof of Theorem~C}

In this section, we provide a construction of
a Banach space of finite cotype from Theorem~C.
According to the results of
\cite{Gluskin,LGl}, there exists a universal constant
$C>1$ such that for every $n\geq 1$
and $N:=\lceil Cn\rceil$,
two independent copies $P_{N,n}$ and $\tilde P_{N,n}$
of the Gaussian polytope \eqref{def:polydef} in $\R^n$
satisfy
$$
d_{BM}(P_{N,n},\tilde P_{N,n})\geq cn
$$
with probability at least $\frac{1}{2}$.
Combining the estimate with Theorem~A,
we obtain that for 
some universal constants $n_0$, $q,C_q<\infty$, and
all $n\geq n_0$,
the event
\begin{align*}
\Event_{n}:=\Big\{&d_{BM}(P_{\lceil Cn\rceil,n},
\tilde P_{\lceil Cn\rceil,n})\geq cn\mbox{ and both
$(\R^n,\|\cdot\|_{P_{\lceil Cn\rceil,n}})$ 
and $(\R^n,\|\cdot\|_{\tilde P_{\lceil Cn\rceil,n}})$}\\
&\mbox{are of cotype $q$ with constant at most $C_q$}
\Big\},
\end{align*}
has a positive probability.

In what follows, for every $n\geq n_0$
we denote by $P_{n}$ and $\tilde P_{n}$
any realization of the random polytopes
$P_{\lceil Cn\rceil,n}$ and $\tilde P_{\lceil Cn\rceil,n}$
such that $\Event_n$ holds.
Define a Banach space ${\bf X}$
as an $\ell_q$ direct sum of $(\R^n,\|\cdot\|_{P_{n}})$
and $(\R^n,\|\cdot\|_{\tilde P_{n}})$, $n\geq n_0$,
namely, {\bf X} consists of
sequences
of pairs of vectors $\big(x_n,\tilde x_n\big)_{n\geq n_0}$, $x_n,\tilde x_n\in \R^n$,
with the norm given by
$$
\big\|\big(x_n,\tilde x_n\big)_{n\geq n_0}\big\|_{\bf X}^q
:=
\sum_{n\geq n_0}\big(\big\|x_n\big\|_{P_{n}}^q+
\big\|\tilde x_n\big\|_{\tilde P_{n}}^q
\big).
$$
Next, we verify that the constructed space has cotype $q$.
Let $y^{(1)},\dots,y^{(k)}$
be an arbitrary collection of vectors in ${\bf X}$,
where each $y^{(i)}$
can be written in form 
$$y^{(i)}
=\big(y^{(i)}_n,\tilde y^{(i)}_n\big)_{n\geq n_0},
\quad y^{(i)}_n,\tilde y^{(i)}_n\in\R^n,$$
and let $\sigma=(\sigma_1,\dots,\sigma_k)$
be a uniform random vector of signs.
By the definition of ${\bf X}$,
$$
\Exp\,\Big\|
    \sum_{i=1}^k \sigma_i\,y^{(i)}
    \Big\|_{\bf X}^q
    =
    \sum_{n\geq n_0}\Big(
    \Exp\,\Big\|
    \sum_{i=1}^k \sigma_i\,y^{(i)}_n
    \Big\|_{P_{n}}^q
    +
    \Exp\,\Big\|
    \sum_{i=1}^k \sigma_i\,\tilde y^{(i)}_n
    \Big\|_{\tilde P_{n}}^q
    \Big).
$$
Further, by our assumptions on the polytopes
$P_{n}$, we have
$$
\Exp\,\Big\|
    \sum_{i=1}^k \sigma_i\,y^{(i)}_n
    \Big\|_{P_{n}}^q
    \geq \frac{1}{C_q^q}
    \sum_{i=1}^k \big\|y^{(i)}_n
    \big\|_{P_{n}}^q,
$$
and analogous inequality holds for $\tilde P_{n}$.
We conclude that
$$
\Exp\,\Big\|
    \sum_{i=1}^k \sigma_i\,y^{(i)}
    \Big\|_{\bf X}^q
    \geq 
    \frac{1}{C_q^q}
    \sum_{n\geq n_0}\Big(
    \sum_{i=1}^k \big\|y^{(i)}_n
    \big\|_{P_{n}}^q
    +
    \sum_{i=1}^k \big\|\tilde y^{(i)}_n
    \big\|_{\tilde P_{n}}^q
    \Big)
    =\frac{1}{C_q^q}
    \sum_{i=1}^k \big\|y^{(i)}
    \big\|_{\bf X}^q,
$$
and the result follows.


\begin{thebibliography}{99}

\bibitem{AS92}
{
F. Affentranger\ and\ R. Schneider, Random projections of regular simplices, Discrete Comput. Geom. {\bf 7} (1992), no.~3, 219--226. MR1149653
}

\bibitem{AGP2015}
{
D. Alonso-Guti\'{e}rrez\ and\ J. Prochno, On the Gaussian behavior of marginals and the mean width of random polytopes, Proc. Amer. Math. Soc. {\bf 143} (2015), no.~2, 821--832. MR3283668
}




\bibitem{B1989}
{
I. B\'{a}r\'{a}ny, Intrinsic volumes and $f$-vectors of random polytopes, Math. Ann. {\bf 285} (1989), no.~4, 671--699. MR1027765
}

\bibitem{Barany2004}
{
I. B\'ar\'any: Random polytopes, convex bodies, and approximation. In: A. Baddeley, I. B\'ar\'any, R. Schneider, W. Weil, Stochastic Geometry (C.I.M.E. Course, Martina Franca, 2004), Lecture Notes Math., Springer.
}

\bibitem{BV94}
{
Y.~M. Baryshnikov\ and\ R.~A. Vitale, Regular simplices and Gaussian samples, Discrete Comput. Geom. {\bf 11} (1994), no.~2, 141--147. MR1254086
}

\bibitem{Bor2009}
{
K. B\"{o}r\"{o}czky\ et al., Mean width of random polytopes in a reasonably smooth convex body, J. Multivariate Anal. {\bf 100} (2009), no.~10, 2287--2295. MR2560369
}

\bibitem{BourgainD}
{
J. Bourgain, On finite-dimensional homogeneous Banach spaces, in {\it Geometric aspects of functional analysis (1986/87)}, 232--238, Lecture Notes in Math., 1317, Springer, Berlin. MR0950984
}

\bibitem{Dv}
{
A. Dvoretzky, Some results on convex bodies and Banach spaces, in {\it Proc. Internat. Sympos. Linear Spaces (Jerusalem, 1960)}, 123--160, Jerusalem Academic Press, Jerusalem. MR0139079
}

\bibitem{E1965}
{
B. Efron, The convex hull of a random set of points, Biometrika {\bf 52} (1965), 331--343. MR0207004
}


\bibitem{Gluskin}
{
E.~D. Gluskin, The diameter of the Minkowski compactum is roughly equal to $n$, Funktsional. Anal. i Prilozhen. {\bf 15} (1981), no.~1, 72--73. MR0609798
}

\bibitem{GKZ20}
{
T. Godland, Z. Kabluchko, D. Zaporozhets,
Angle sums of random polytopes,
arXiv:2007.02590
}


\bibitem{HMR04}
{
D. Hug, G. Munsonius\ and\ M. Reitzner, Asymptotic mean values of Gaussian polytopes, Beitr\"{a}ge Algebra Geom. {\bf 45} (2004), no.~2, 531--548. MR2093024
}

\bibitem{JvN35}
{
P. Jordan\ and\ J. von~Neumann, On inner products in linear, metric spaces, Ann. of Math. (2) {\bf 36} (1935), no.~3, 719--723. MR1503247
}

\bibitem{Ka2021}
{
Z.~A. Kabluchko, Angles of random simplices and face numbers of random polytopes, Adv. Math. {\bf 380} (2021), Paper No. 107612, 68 pp. MR4205707
}

\bibitem{KZ19}
{
Z.~A. Kabluchko\ and\ D.~N. Zaporozhets, Expected volumes of Gaussian polytopes, external angles, and multiple order statistics, Trans. Amer. Math. Soc. {\bf 372} (2019), no.~3, 1709--1733. MR3976574
}

\bibitem{KZ2020}
{
Z.~A. Kabluchko\ and\ D.~N. Zaporozhets, Absorption probabilities for Gaussian polytopes and regular spherical simplices, Adv. in Appl. Probab. {\bf 52} (2020), no.~2, 588--616. MR4123647
}


\bibitem{Khatri}
{
Khatri, C. G. On certain inequalities for normal distributions and their applications to simultaneous confidence bounds. Ann. Math. Statist. 38 (1967), 1853--1867.
}


\bibitem{K1994}
{
K.-H. K\"{u}fer, On the approximation of a ball by random polytopes, Adv. in Appl. Probab. {\bf 26} (1994), no.~4, 876--892. MR1303867
}

\bibitem{Kwapien}
{
S. Kwapie\'{n}, Isomorphic characterizations of inner product spaces by orthogonal series with vector valued coefficients, Studia Math. {\bf 44} (1972), 583--595. MR0341039
}

\bibitem{LGl}
{
R. Lata\l a\ et al., Banach-Mazur distances and projections on random subgaussian polytopes, Discrete Comput. Geom. {\bf 38} (2007), no.~1, 29--50. MR2322114
}




\bibitem{LPRT2005}
{
A.~E. Litvak\ et al., Smallest singular value of random matrices and geometry of random polytopes, Adv. Math. {\bf 195} (2005), no.~2, 491--523. MR2146352
}

\bibitem{MTJD}
{
P. Mankiewicz\ and\ N. Tomczak-Jaegermann, A solution of the finite-dimensional homogeneous Banach space problem, Israel J. Math. {\bf 75} (1991), no.~2-3, 129--159. MR1164587
}



\bibitem{MP76}
{
B. Maurey\ and\ G. Pisier, S\'{e}ries de variables al\'{e}atoires vectorielles ind\'{e}pendantes et propri\'{e}t\'{e}s g\'{e}om\'{e}triques des espaces de Banach, Studia Math. {\bf 58} (1976), no.~1, 45--90. MR0443015
}



\bibitem{Mil}
{
V.~D. Milman, A new proof of A. Dvoretzky's theorem on cross-sections of convex bodies, Funkcional. Anal. i Prilo\v{z}en. {\bf 5} (1971), no.~4, 28--37. MR0293374
}

\bibitem{Muller89}
{
J.~S. M\"{u}ller, On the mean width of random polytopes, Probab. Theory Related Fields {\bf 82} (1989), no.~1, 33--37. MR0997430
}




\bibitem{Retzner2005-}
{
M. Reitzner, Central limit theorems for random polytopes, Probab. Theory Related Fields {\bf 133} (2005), no.~4, 483--507. MR2197111
}

\bibitem{Retzner2005}
{
M. Reitzner, The combinatorial structure of random polytopes, Adv. Math. {\bf 191} (2005), no.~1, 178--208. MR2102847
}

\bibitem{Reitzner2010}
{
M. Reitzner, Random polytopes, in {\it New perspectives in stochastic geometry}, 45--76, Oxford Univ. Press, Oxford. MR2654675
}

\bibitem{RS1963}
{
A. R\'{e}nyi\ and\ R. Sulanke, \"{U}ber die konvexe H\"{u}lle von $n$ zuf\"{a}llig gew\"{a}hlten Punkten, Z. Wahrsch. Verw. Gebiete {\bf 2} (1963), 75--84 (1963). MR0156262
}

\bibitem{Sidak}
{
Z. \v{S}id\'{a}k, Rectangular confidence regions for the means of multivariate normal distributions, J. Amer. Statist. Assoc. {\bf 62} (1967), 626--633.
}

\bibitem{Schneider2008}
{
R. Schneider, Recent results on random polytopes, Boll. Unione Mat. Ital. (9) {\bf 1} (2008), no.~1, 17--39. MR2387995
}

\bibitem{SW2008}
{
R. Schneider\ and\ W. Weil, {\it Stochastic and integral geometry}, Probability and its Applications (New York), Springer, Berlin, 2008. MR2455326
}

\bibitem{Szarek81}
{
S.~J. Szarek, Nets of Grassmann manifold and orthogonal group, in {\it Proceedings of research workshop on Banach space theory (Iowa City, Iowa, 1981)}, 169--185, Univ. IA, Iowa City, IA. MR0724113
}


\bibitem{WW93}
{
W. Weil\ and\ J.~A. Wieacker, Stochastic geometry, in {\it Handbook of convex geometry, Vol. A, B}, 1391--1438, North-Holland, Amsterdam. MR1243013
}

\bibitem{Ver18}
{
R. Vershynin, {\it High-dimensional probability}, Cambridge Series in Statistical and Probabilistic Mathematics, 47, Cambridge Univ. Press, Cambridge, 2018. MR3837109
}

\bibitem{Vu2005}
{
V.~H. Vu, Sharp concentration of random polytopes, Geom. Funct. Anal. {\bf 15} (2005), no.~6, 1284--1318. MR2221249
}

\bibitem{Vu2006}
{
V.~H. Vu, Central limit theorems for random polytopes in a smooth convex set, Adv. Math. {\bf 207} (2006), no.~1, 221--243. MR2264072
}



\end{thebibliography}
\end{document}